\documentclass[aos]{imsart}
\newcommand{\ignore}[1]{}

\usepackage[T1]{fontenc}
\usepackage[utf8]{inputenc}
\usepackage[english]{babel}
\usepackage[font=small, labelfont = bf]{caption}
\usepackage{graphicx}
\usepackage{algorithm}
\usepackage{algorithmic}
\usepackage {graphicx,color}
\usepackage{capt-of}
\usepackage{booktabs}
\usepackage{varwidth}
\usepackage{subcaption}
\usepackage{amsmath}
\usepackage{amsbsy}
\usepackage{bm}
\usepackage{textcomp}
\usepackage{amsthm}
\usepackage{hyperref}
\usepackage{blindtext}
\usepackage[T1]{fontenc}
\usepackage[utf8]{inputenc}
\usepackage[english]{babel}
\usepackage[font=small, labelfont = bf]{caption}
\usepackage{graphicx}
\usepackage{algorithm}
\usepackage{algorithmic}
\usepackage {graphicx,color}
\usepackage{capt-of}
\usepackage{booktabs}
\usepackage{varwidth}
\usepackage{comment}
\usepackage{lineno}
\usepackage{amssymb}
\usepackage{amsmath}
\usepackage{amsbsy}
\usepackage{bm}
\usepackage{textcomp}
\usepackage{amsthm}
\usepackage{hyperref}
\usepackage{blindtext}
\usepackage{booktabs}
\RequirePackage{amsthm,amsmath,amsfonts,amssymb}
\RequirePackage[numbers]{natbib}

\usepackage{siunitx}
\usepackage{multirow}
\usepackage{makecell}
\numberwithin{equation}{section}
\newtheorem{thm}{Theorem}[section]
\newtheorem{lemma}{Lemma}[section]
\newtheorem{condition}{Condition}[section]
\newtheorem{proposition}{Proposition}[section]

\newtheorem{remark}{Remark}[section]

\newtheorem{definition}{Definition}

\newcommand{\beq}{\begin{equation}}
\newcommand{\eeq}{\end{equation}}
\newcommand{\beas}{\begin{eqnarray*}}
\newcommand{\eeas}{\end{eqnarray*}}
\newcommand{\bea}{\begin{eqnarray}}
\newcommand{\eea}{\end{eqnarray}}
\newcommand{\bei}{\begin{itemize}}
\newcommand{\eei}{\end{itemize}}
\newcommand{\ben}{\begin{enumerate}}
\newcommand{\een}{\end{enumerate}}
\newcommand{\bet}{\begin{theorem}}
\newcommand{\eet}{\end{theorem}}
\newcommand{\bel}{\begin{lemma}}
\newcommand{\eel}{\end{lemma}}
\newcommand{\bep}{\begin{proposition}}
\newcommand{\eep}{\end{proposition}}
\newcommand{\bed}{\begin{definition}}
\newcommand{\eed}{\end{definition}}
\newcommand{\bec}{\begin{corollary}}
\newcommand{\eec}{\end{corollary}}
\newcommand{\bex}{\begin{example}}
\newcommand{\eex}{\end{example}}
\newcommand{\N}{\mathbb{N}}
\newcommand{\Z}{\mathbb{Z}}

\newcommand{\R}{\mathbb{R}}
\newcommand{\E}{\mathbb{E}}

\def\limsup{\mathop{\overline{\rm lim}}}

\usepackage{enumerate}
\renewcommand*{\theenumi}{\thesection.\arabic{enumi}}

\makeatletter
\def\saveenum{\xdef\@savedenum{\the\c@enumi\relax}}
\def\resetenum{\global\c@enumi\@savedenum}
\makeatother

\newcommand{\IE}{\mathbb{E}}
\newcommand{\IP}{\mathbb{P}}
\newcommand{\IB}{\mathbb{B}}

\ignore{
\setlength{\textwidth}{6.5in}
\setlength{\textheight}{8.5in} \setlength{\topmargin}{0.25in}
\setlength{\headheight}{0in} \setlength{\oddsidemargin}{0in}
\setlength{\evensidemargin}{0in} \flushbottom
}

\graphicspath{{Figures/}}

\begin{document}

\title{Gaussian Approximation For Non-stationary Time Series with Optimal Rate and Explicit Construction}
\author{Soham Bonnerjee$^{1}$, Sayar Karmakar$^{2}$ and Wei Biao Wu$^{1}$}
\address{$^{1}$ Department of Statistics, University of Chicago,\\ $^{2}$ Department of Statistics, University of Florida}

\begin{abstract}Statistical inference for time series such as curve estimation for time-varying models or testing for existence of change-point have garnered significant attention. However, these works are generally restricted to the assumption of independence and/or stationarity at its best. The main obstacle is that the existing Gaussian approximation results for non-stationary processes only provide an existential proof and thus they are difficult to apply. In this paper, we provide two clear paths to construct such a Gaussian approximation for non-stationary series. While the first one is theoretically more natural, the second one is practically implementable. Our Gaussian approximation results are applicable for a very large class of non-stationary time series, obtain optimal rates and yet have good applicability. Building on such approximations, we also show theoretical results for change-point detection and simultaneous inference in presence of non-stationary errors. Finally we substantiate our theoretical results with simulation studies and real data analysis.
\end{abstract}


\maketitle

\section{Introduction}
\label{sec:intro}
Statistical inference for time series is an important topic that has garnered significant attention over the past several decades. There is a well-developed asymptotic theory of Gaussian approximation for stationary processes that in turn yields a solid foundation for doing asymptotic inference. However, in practice, non-stationary time series processes are more ubiquitous, and unfortunately, similar Gaussian approximation tools for non-stationary processes are either not sharp enough or difficult to apply. Our main goal in this paper is to establish optimal KMT-type Gaussian approximations for non-stationary time series that also provide an explicit construction strategy and thus enable asymptotic inference for such series. 

We now discuss some motivations for theoretical development for non-stationary time series. Stationarity is an idealized assumption for any real-life series observed over a long period of time. In the parlance of analyzing such long series, when parametric models are used, typically this translates to systematic deviation of the parameters. Even without such a parametric guide, one can observe intrinsic changes in how the dependence evolves over time. Apart from these, different external factors such as recession, war, politics, pandemic etc. affect time series and can introduce abrupt paradigm shifts. Such shifts could be of different types- either a shift in mean, or shock events that change a process that was varying slowly or in a more stationary way. These two approaches are captured in the literature of time-varying models and change-point analyses respectively.

The literature of time-varying models tries to address this issue by allowing model parameters to vary smoothly over time. See \cite{fan99}, \cite{fan2000}, \cite{hoover98}, \cite{huang04}, \cite{lin01}, \cite{ramsay05}, \cite{zhang02}, \cite{cai07} among others. The inference questions arise naturally while choosing a time-varying model in contrast of a time-constant one. Such hypothesis testing frameworks are discussed in \cite{regression2012}, \cite{regression2015}, \cite{chow60}, \cite{brown75}, \cite{nabeya88}, \cite{leybourne89}, \cite{nyblom89}, \cite{ploberger89}, \cite{andrews93} and \cite{lin99}. Moving from pointwise inference, \cite{zhouwu10}, \cite{zhibia}, \cite{karmakar2022simultaneous} discussed obtaining more challenging simultaneous confidence bands\ignore{ for linear and non-linear time series}. Such simultaneous inference requires Gaussian approximation beyond the central limit theorem, and motivates for KMT-type Gaussian approximations as spelled out in \eqref{eq:errorrate}. The second approach- the analysis of change-points, originated in quality control literature (\cite{page1954continuous, page1955test}), but has since become an integral part of a wide variety of fields, among them economics (\cite{perron2006dealing}), finance (\cite{andreou2009structural}), climatology (\cite{reeves2007review}) and engineering (\cite{stoumbos2000state}). Building on estimation techniques, these problems discuss different types of inference problems such as the existence of change-point or creating confidence bands for means of different pieces etc. The test statistic for testing existence of change-points may be viewed as two-sample tests adjusted for the unknown break location, thus leading to max-type procedures. Such tests also need a Gaussian approximation as mentioned in \eqref{eq:errorrate} to provide correct cut-off. For some useful references on these see \cite{bai1998estimating} and \cite{csorgo1997limit} among others. Structural break estimation can also be viewed as a model selection problem; see \cite{davis2006structural}, \cite{lu2010mdl} and \cite{robbins2011changepoints}. See also \cite{aue2013structural} and \cite{jandhyala2013inference} for excellent reviews on change-point inference literature.

However, in both of these paradigms, typically the error process is assumed to be stationary and thus the techniques involved do not go beyond what we already know for stationary series. In other words, the non-stationarity has generally been reflected only in the signal and not in the noise process. This posits a challenging but a fundamental problem. The literature on inference for non-stationary time series is sparse due to difficulty of obtaining a sharp, explicit Gaussian approximation. The existing results are either not as sharp as those for stationary processes, or are difficult to construct.

We now proceed to mathematically introduce the problem. For independent and identically distributed $X_i$ with $\mathbb E(X_i) = 0, \mathbb E(|X_i|^p) < \infty$, $p > 2$, \citet{MR0375412, MR0402883} obtained an optimal Gaussian approximation: for $S_i:= \sum_{j=1}^i X_j$,
\begin{eqnarray}\label{eq:errorrate}
 \max_{1 \leq j \leq n}|S_j'- \mathbb{B}( \mathbb E (S^2_j))| = o_{\rm a.s.}(\tau_n),
\end{eqnarray}
where $\mathbb E(S^2_j) = j \mathbb E (X_1^2)$, $\mathbb{B}(\cdot)$ is the standard Brownian motion and $S_n'$ is constructed on a richer space; such that $(S_i)_{i \ge 1} =_{\mathbb{D}} (S_i')_{i \ge 1}$, and the approximation rate $\tau_n = n^{1/p}$ is optimal when only finite $p$th moment is assumed. \color{black} Henceforth, throughout this paper, we will assume $p>2$ unless specified explicitly. \color{black} The Gaussian approximation \eqref{eq:errorrate} substantially generalizes the Central Limit Theorem ${S_n}/\sqrt{n} \Rightarrow N(0, \mathbb E (X_1^2))$, and it allows for a systematic study of statistical properties of estimators based on independent data. The optimal rate of $n^{1/p}$ was matched for a large class of stationary time series in the seminal work by \citet{kmt}. In the latter work, they assume the stationary causal representation for $X_i$, and are able to replace $\mathbb E(S^2_j) = j \mathbb E (X_1^2)$ in \eqref{eq:errorrate} by $j \sigma_{\infty}^2$ where $\sigma_{\infty}^2= \sum_{i \in \mathbb{Z}} \IE(X_0X_i)$ is the long-run variance of the time series. One can see that $\sigma_{\infty}^2= \lim_{n \to \infty} \IE(S_n^2)/n$ and thus $S_i$ being approximated by a Gaussian process with variance $i \sigma_{\infty}^2$ makes intuitive sense from the idea of preserving a second order property. Unfortunately, for a non-stationary process, one does not have the notion of such a long-run variance and thus the existing Gaussian approximation results are somewhat abstract and unclear.

To characterize the non-stationary process $(X_t)$, we view $X_t$ as outputs from a physical system with the following causal representation: 
\begin{equation}
\label{eq:model_Z}
X_{t}=g_t (\mathcal{F}_t), \ \text{with $\mathcal{F}_t=(\ldots, \varepsilon_{t-1}, \varepsilon_{t})$,}
\end{equation}
where $(\varepsilon_{i})_{i \in \Z}$ are i.\,i.\,d.\ inputs of this system and $g_t: \R^{\infty} \to \R$ are measurable functions. A Gaussian approximation for such non-stationary processes was obtained by \cite{zhouzhousinica}, with a suboptimal rate and only for $2 < p \leq 4$. On the other hand, for inferential procedures it is important to establish an approximation for the process $(S_i)_{i=1}^n$. They did provide a regularization $G_j=\sum_{i=1}^j\Sigma_i^{1/2} Z_i$, where $\Sigma_i = \operatorname{Var}(\sum_{k=i}^{\infty}(\IE(X_k|\mathcal{F}_i)-\IE(X_k|\mathcal{F}_{i-1})))$ and $Z_i$ are i.i.d. Gaussian; however, $\Sigma_i$'s are not naturally estimable quantities. This result was improved upon by \cite{KarmakarWu2020}, who obtained optimal rate $n^{1/p}$ rate for all $p>2$. However, even their approximating Gaussian process is not regularized as it only provides approximation for blocks of partial sums, and not all $S_j$ as \eqref{eq:errorrate} does. Moreover, the variance of the approximating Gaussian process was difficult to interpret and connect with that of the original process.\ignore{
A Gaussian approximation for non-stationary multiple series with optimal rate $n^{1/p}$ was obtained in \citet{KarmakarWu2020}. Their rate improves upon \citet{zhouzhousinica} and other similar works for multiple non-stationary series. However, upon careful examination one could see that the approximating Gaussian process is not regularized as it only provides approximation for blocks of partial sums, and not all $S_j$ as \eqref{eq:errorrate} does. However, for inferential procedures it is often important to establish an approximation for each $S_i$. Thus just the approximations for blocks, although fine mathematically, will fall short of application, if not regularized for every $S_i$. Such a block-based Gaussian approximation for non-stationary process was also obtained in an older paper \citet{zhouzhousinica}, albeit only for $2<p<4$. They did provide a regularization for each $S_j$, but it involved quantities that cannot be estimated in practice. Even more seriously this only attains a non-optimal rate.
}
Recently, \cite{mies1} used a local long-run covariance matrix as a proxy to the variance of the approximating Brownian motion. Their proof relies on martingale embedding strategy of \cite{eidan2020aop} to bound Wasserstein distance of the partial sums and their Gaussian analogues. Nonetheless, their rate is sub-optimal.\color{black}

\color{black}
Keeping the main goal of regularizing the approximating Gaussian process, we note that, it is possible to preserve the second order property without the notion of long-run variance if the approximating (of $S_j$) Gaussian process can be written as $G_i=\sum_{j \leq i} Y_i$ with $\IE(S_i^2)=\IE(G_i^2)$. We start with one such approximation which ensures this; in fact we are able to establish a Gaussian approximation that ensures $\operatorname{Cov}(X_i,X_j)=\operatorname{Cov}(Y_i,Y_j)$ which entails $\IE(S_i^2)=\IE(G_i^2)$. Assumption of Gaussianity is frequently used in many areas of statistics where, as further specification, one puts a covariance structure on $(X_i)$. Our Gaussian approximation provides theoretical validation that for non-stationary process, one can still obtain an approximating Gaussian process that matches the covariance at a modular level. To the best of our knowledge, such covariance-matching Gaussian approximations, despite being quite natural for non-stationary processes, are rarely discussed in the literature. In particular, for a possible non-stationarity in covariance, such second-order preserving approximation seems to be a first such result that additionally maintains optimal rate. 

Our first result is applicable in situations where the practitioner knows the covariance structure of the observed processes. However, for general non-stationary processes with unknown covariance structure, the practical implementation with this novel Gaussian approximation remains somewhat challenging. Our second set of Gaussian approximation results first embed the approximating Gaussian process in a Brownian motion with evolving variance and then regularize the latter. As expected, the variance generally does not increase linearly as it does in \cite{kmt} for the stationary case. However, in our approximation $S_j$ is approximated by a Brownian motion valued at $\IE(S_j^2)$, which is same as \eqref{eq:errorrate}. Unlike \cite{mies1}, the variance of our approximating Gaussian process is simply $\IE(S_i^2)$, which immediately suggests intuitive estimators of that variance. 

\color{black}
Next we address the issue of estimating the variance of the approximating Gaussian processes. We first derive a block version of our theoretical Gaussian approximation which in turn yields a conditional Gaussian approximation where estimated block variances are used to construct the variances of the approximating theoretical Gaussian process. We are able to achieve $n^{1/4+ \varepsilon}$ rate here which is nearly optimal when variances are to be estimated. This also means that to achieve such results, assumptions on only slightly higher than 4-th moments suffice. Here, we also reflect on an alternative estimation procedure, and show that our "Block-based Running Variance (BRV)" estimate gives better rates for all $p>2$. Finally, we apply our results to three prominent inference problems, namely the inference problem related to existence of change-point, the simultaneous confidence bands for non-stationary time series and asymptotic distribution of wavelet coefficient process. As mentioned above already, stationarity and/or Gaussianity were standard assumptions in all these literature throughout and this paper erases this barrier and establishes theoretical guarantees for a much larger class of time series. 
 
Our main contributions are summarized below.
\begin{itemize}
  \item We obtain the sharp KMT-type Gaussian approximations of the order $n^{1/p}$ for non-stationary time series with minimal conditions. In particular, 
  \begin{itemize}
    \item in our first result, we observe a novel Gaussian approximation which matches the covariance structure. Despite being intuitively very natural for non-stationary processes, ours is probably the first such approximation result in the literature.
    \item We also explore a second type of Gaussian approximation which involves embedding a Brownian motion much like \cite{kmt} or \cite{KarmakarWu2020}. Crucially, we recover the sharp $n^{1/p}$ rate modulo a logarithmic factor without the lower bound assumption of block variance needed in \cite{KarmakarWu2020}. 
  \end{itemize}
  \item We discuss estimation of the running variance of the approximating Brownian motion and show consistency of such estimators using uniform deviation inequalities. Such maximal deviation bounds for quadratic forms based on non-stationary processes may be of independent interest. 
  \item Finally, we show applications of such Gaussian approximation through the lens of three prominent inference problems, namely the inference problem related to change-point, the simultaneous confidence bands for non-stationary time series and asymptotic distributions of wavelet coefficient processes. As mentioned above already, stationarity and/or Gaussianity were standard assumptions in all these literature throughout and this paper overcomes these limitations to arrive at much more general results. 
  \item We also provide some simulations to corroborate our Gaussian approximations and an analysis of an interesting dataset that highlights our applications. 

  \end{itemize}

\subsection{Organization of the paper}
The rest of the paper is organized as follows. In Section \ref{subsec:fdm}, we discuss a functional dependence measure that allows us to encode dependence in a mathematically tractable way for a large class of non-stationary time series. We also discuss other general assumptions there. Sections \ref{subsec:thm:new_GA} and \ref{subsec:thm:GA} discuss the two Gaussian approximations, which are the main theoretical contributions of our paper. Next, Section \ref{ssc:estvar} is used to describe the block-bootstrap Gaussian approximation and related results, featuring a result on a novel deviation inequality for non-stationary quadratic forms. We discuss three important inference problems in Section \ref{sec:application}. The hypothesis testing related to test existence of change-point is discussed in \ref{subsec:cpt}. Subsequently, we discuss simultaneous confidence bands for non-stationary time series, which is deferred to Subsection \ref{subsec:scb}. Finally, the discussion on wavelet coefficient process is deferred till Section \ref{sec:wavelet}. Next, we use Section \ref{sec:simu} to demonstrate through simulations that we achieve better approximations with the regularization spelt out in theoretical results than the prototypical block-sum variance. We also show extensive simulation results for the first two of the above-mentioned applications. For space constraint, some of these simulations are deferred to Appendix Section \ref{appendix:simu}. Finally, we show advantage of our theory and estimates by analyzing a recent archaeological dataset in Section \ref{sec:datanalysis}. All the proofs are postponed to Appendix Sections \ref{sec:section2proofs}, \ref{sec2proofsnew}, \ref{sec:section3proofs} and \ref{appendix:scb}.

\color{black}

\ignore{
The $d$-variate normal distribution with mean $\mu$ and covariance matrix $\Sigma$ is denoted by $N(\mu, \Sigma)$. Denote by $I_d$ the $d \times d$ identity matrix. For a matrix $A= (a_{ij})$, we define its Frobenius norm as $|A|= (\sum a_{ij}^2)^{1/2}$. For a positive semi-definite matrix $A$ with spectral decomposition $A= Q D Q\tran$, where $Q$ is orthonormal and $D = (\lambda_1, \ldots, \lambda_d)$ with $\lambda_1 \ge \ldots \ge \lambda_d$, write the Grammian square root $A^{1/2}= Q D^{1/2} Q\tran$, $\rho_*(A) = \lambda_d$ and $\rho^*(A) = \lambda_1$.
}
\subsection{Notation}
For a random variable $Y$, write $Y \in \mathcal{L}_p, \ p > 0$, if $\|Y \|_p := \mathbb E(|Y |^p ) ^{1/p} < \infty$. For $\mathcal{L}_2$ norm write $\|\cdot \| =\| \cdot \|_ 2$. Throughout the text, we use $C$ for constants that might take different values in different lines unless otherwise specified. For two positive sequences $a_n$ and $b_n$, if $a_n/b_n\to 0$, write $a_n =o(b_n)$. Write $a_n \lesssim b_n$ or $a_n=O(b_n)$ if $a_n \leq C b_n$ for all sufficiently large $n$ and some constant $C<\infty$. Similarly for a sequence of random variables $(X_n)_{n \geq 1}$ and a positive sequence $y_n$, if $X_n / y_n \to 0$ in probability, we write $X_n =o_{\IP}(y_n)$, and if $X_n/y_n$ is stochastically bounded, we write $X_n=O_{\IP}(y_n)$. 


\section{Gaussian approximation results}
Before we proceed to discuss the Gaussian approximation results for a general class of non-stationary time series, we first provide a concise introduction of similar results for independent random variables.
\color{black}Note that in principle such Gaussian approximations for random variables $(X_i)_{i=1}^n$ require a common, possibly enriched probability space $(\Phi_c, \mathcal{A}_c, \IP_c)$ on which the approximating Gaussian processes and random variables $(X_i^c)_{1 \leq i \leq n} =_{\mathbb{D}}(X_i)_{1 \leq i \leq n}$ can be defined. In order for better readability, we omit this technicality and simply state our results in terms of the original random variables $X_i$'s. \color{black}
\color{black}
\label{sec:main_thm}
\subsection{Gaussian approximation for independent random variables}
\label{ssc:sakhanenko}
For i.i.d. random variables, the mentioned result \eqref{eq:errorrate} by \cite{MR0375412, MR0402883} represented the culmination of a series of results on \textit{strong invariance principle} starting from \cite{erdoskac} and \cite{doob}. Subsequently, the seminal paper by \citet{Sakhanenko2006} essentially generalized the KMT-type Gaussian approximation for independent but possibly not identically distributed random variables. The following theorem follows easily from \cite{Sakhanenko2006}.
\begin{thm}\label{thm:ind}
Let $(X_i)_{1\leq i \leq n}$ be independent but possibly not identically distributed random variables with $\IE(X_i) = 0$ and for a $p > 2$, $\max_{1 \leq i \leq n} \|X_i\|_p = O(1)$, and there exists $\gamma \geq 2$ such that 
  \begin{equation} \label{shakhanenkocondn}
    \sum_{i=1}^n \IE[\min\{|X_i|^\gamma/{n^{\gamma/p}}, |X_i|^2/{n^{2/p}} \}] = o(1).
  \end{equation}
  Then, there exists a Brownian motion $\mathbb{B}(\cdot)$, such that the following holds
  \begin{equation}\label{sakhanenkoind}
      \max_{1 \leq i \leq n} | S_i - \mathbb{B}(\IE(S_i^2))|=o_{\IP}(n^{1/p}). 
  \end{equation}
\end{thm}
 The readers can look into \cite{zaitsev1998, zaitsev19982} and \cite{zaitsev19983} for a review of similar approximations for independent but possibly non-identically distributed random variables. For time series, \cite{kmt} represents the optimal result for stationary processes in this direction, while \cite{KarmakarWu2020} shows an optimal existential result for non-stationary multivariate processes. However, \cite{KarmakarWu2020} does not provide any result about the covariance structure of the approximating Gaussian processes, apart from them having independent increments. 
 \color{black} However, in the search for an explicit covariance regularization of the Gaussian approximations, it is natural to conjecture that the approximating Gaussian processes have the same second-order structure as that of the original non-stationary process $X_t$. To deal with such results, we need to characterize the dependency set-up of the wide class of the non-stationary processes we consider in \eqref{eq:model_Z}. This structural premise is laid out in the next section.

\subsection{Functional dependence measure for non-stationary processes}
\label{subsec:fdm}

To deal with the dependency structure of a non-stationary process, we employ the framework of functional dependence measure \cite{Wu2005}. We will work with \eqref{eq:model_Z}, which is quite general and arises naturally from writing the joint distribution of $(X_1, \ldots, X_n)$ in terms of compositions of conditional quantile functions of i.i.d. uniform random variables. \ignore{ ; see \citet{wu2010new} for details. To see this, for $U_1, \ldots, U_n {\sim}_{i.i.d.} U[0,1] $, one can write $X_1=_{\mathbb{D}}G_1(U_1)$ for the quantile function $G_1$. Similarly, $[X_2 | X_1=x]=_{\mathbb{D}}G_2'(x, U_2)$ for the conditional quantile function $G_2'$. Thus, 
\[ (X_1, X_2) =_{\mathbb{D}} (X_1, G_2(X_1, U_2)) =_{\mathbb{D}} (G_1(U_1), G_2'(g_1(U_1), U_2))=_{\mathbb{D}}  (G_1(U_1), G_2(U_1, U_2)) \]
where the last equality follows from suitably defining $G_2$ in terms of $G_2'$ and $G_1$. Generalizing the above argument, we arrive at 
\[ (X_1, \ldots, X_n) =_{\mathbb{D}} (G_1(U_1), G_2(U_1, U_2), \ldots, G_n(U_1, \ldots, U_n)). \]
This representation is the motivational idea behind \eqref{eq:model_Z}.} With this system, given $k \ge 0$, a time lag, we measure the dependence from how much the outputs $X_i$ of this system will change if we replace the input information at time $i - k$ with an i.i.d. copy $\varepsilon_{i-k}'$. For $p \ge 1$, define the uniform functional dependence \color{black}
\begin{align}\label{eq:fdm}
\delta_{p}(k) :=\sup_{i} (\E |X_{i}-X_{i,\{i-k\}}|^p)^{1/p}, \mbox{ where } X_{i,\{i-k\}} = g_i(\ldots,\varepsilon_{i-k-1}, \varepsilon_{i-k}',\varepsilon_{i-k+1}, \ldots, \varepsilon_{i} )
\end{align}
\color{black}
is a coupled version of $X_i$. We will assume $\IE(X_i)=0$. Note that $(\E |X_{i}-X_{i,\{i-k\}}|^p)^{1/p}$ encapsulates the dependence of $X_i$ in $\varepsilon_{i-k}$. Since $X_i$ is a non-stationary process, the physical mechanism process $g_i$ is allowed to be different for every $i$. Thus we have defined the functional dependence measure in a uniform manner, by taking supremum over all $i$. This measure \eqref{eq:fdm} is directly related to the data-generating mechanism, and we will express our dependence condition in terms of 
\vspace{-0.1in}
\begin{align}\label{theta}
  \Theta_{i,p}=\sum_{k=i}^{\infty} \delta_p(k) \ ,  i \geq 0. 
\end{align}
Observe that $\sup_i\|X_i\|_p \leq \Theta_{0,p}$. With this framework, we are able to conveniently propose conditions on temporal dependence for the non-stationary time series models we will use. \vspace{-0.1in}
\ignore{
\subsection{Central limit theorem for non-stationary time series}
\label{subsec:thm:clt}
Let $(X_t)_{t \in \Z}$ be a non-stationary process with mean zero, and let $S_j = \sum_{t =1}^j X_i$. $v_i := \operatorname{var} (S_j)$ is defined as the variance of this sequence of partial sums. To begin with, we want to propose a central limit theorem for non-stationary time series under our framework.

\begin{condition}
\label{cond:fdm1}
$(X_i)_{i=1}^n$ satisfies that the $p$-th moment of every $X_i$ is finite. Assume the summation of the uniform functional dependence measure of this non-stationary process is also finite:
$$\sum_{k=0}^{\infty} \delta_{p}(k) < \infty.$$
\end{condition}
}
\ignore{

\begin{thm}
\label{thm:clt_non-stationary}
Assume that Condition~\ref{cond:fdm1} holds with $p>2$ and there exists a constant $c > 0$ such that
\begin{equation}
  \label{cond:var_psum}
  \frac{v_n}{n} \geq c
\end{equation}
Then 
\begin{equation}
  \label{result:clt}
  \frac{S_n}{\sqrt{v_n}} \Rightarrow N(0,1).
\end{equation}
\end{thm}

\color{black}The proof of Theorem \ref{thm:clt_non-stationary} can be found in \cite{Wu_2004}. We would like to discuss about why Condition~(\ref{cond:var_psum}) is necessary for our result~(\ref{result:clt}). Intuitively, if (\ref{cond:var_psum}) didn't hold, $(X_n)$ could be degenerate, which makes it impossible for any CLT type result to stand. We do not have to worry about degeneration problem in classical CLT results for stationary process as the variance is constant.
\color{black}

We can give the following example with non-degenerate $X$, which also shows why the Condition \ref{cond:var_psum} is necessary. Consider $X_i=\varepsilon_i-\varepsilon_{i-1}$ for i.i.d. $\{ \varepsilon_i\}$'s with mean $0$ and variance $1$. Note that $S_n=\varepsilon_n-\varepsilon_0$ and thus $v_n=2$. Then clearly $\frac{S_n}{\sqrt{v_n}} \not \Rightarrow Z$. 

In fact, similar to what \cite{Wu_2004} mention in their argument involving martingale approximation of $S_n$ after conditioning on $X_0$, in general the CLT for non-stationary time series may fail to hold if $v_n \sim \sqrt{n} l(n)$ for a \textit{slowly varying function}\footnote{$l$ is a slowly varying function is $\frac{l(\lambda n)}{l(n)} \to 1$ for all $\lambda >0$ as $n \to \infty$} $l$ with $l(n) \to 0$ as $n \to \infty$. \cite{kmt} does not have such a condition as they are dealing with stationary time series with long-range variance $\sigma^2$, for which $v_n/n \to \sigma^2>0$ automatically. However, the above discussion shows the necessity of this condition for dealing with non-stationarity. Similar conditions have been previously proposed in \cite{KarmakarWu2020}, for non-stationary multiple time series. Technically, in that paper, such lower bound would be used to normalize quantities related to functional dependence measure, which are further used to determine the block size in the blocking approximation.

Theorem \ref{thm:clt_non-stationary} gives a CLT result which is typically useful for point-wise inference. For simultaneous inference, we need a strong-invariance type result involving Gaussian approximation of the entire time-spectrum, which is provided by the following Theorem \ref{thm:GA} and Theorem \ref{thm:improvedGA}, our main contributions of this paper.

}
\color{black}
\subsection{Gaussian approximation maintaining covariance structure}
\label{subsec:thm:new_GA}
As discussed in Section \ref{subsec:fdm}, to state our Gaussian approximation result, we need to properly control the temporal decay by putting mild assumptions on $\Theta_{i,p}$. In particular, we will need that $\Theta_{i,p}$ decays with a polynomial rate. 
\begin{condition}
\label{cond:fdm3}
Consider \eqref{eq:model_Z}. Suppose that $\Theta_{0,p} < \infty$ for some $p>2$. Assume there exists $A> 1$ and constant $C > 0$, such that the uniform dependency-adjusted norm 
\begin{align}\label{eq:Thetaip}
\mu_{p,A}:=\sup_{i\geq 0} (i+1)^{A}\Theta_{i,p} \leq C < \infty.
\end{align}

\end{condition} \noindent
Condition \ref{cond:fdm3} is satisfied by a large class of processes. Some examples are mentioned in Section \ref{ssc:examples}. The assumption $\Theta_{0,p} < \infty$ can be interpreted as the cumulative dependence of $(X_i)_{i \geq k}$ on $\varepsilon_k$ being finite. If it fails, the process can be long-range dependent, and in such cases the Brownian motion approximations of the partial sum processes may fail. Since the process $(X_i)_i$ is non-stationary, in order to better control its distributional behavior, we need a uniform integrability condition:

\begin{condition}
\label{cond:ui}
\color{black}For the same $p$ as in Condition \ref{cond:fdm3}, the series $(| X_i|^p)$ satisfies the truncated uniform integrability condition: 
\[ \text{For any fixed $a>0$, }\sup_{i} \IE\left(|X_i|^p\mathbb{I}_{\{|X_i|^p\geq an\}}\right) \to 0 \ \text{as} \ n \to \infty.\]
\end{condition}
The classical uniform integrability condition for $(|X_t|^p)_t$ is $\sup_{i} \IE (|X_i|^p \mathbb{I}_{\{|X_i|^p \geq k\} } ) \to 0$ as $k \to \infty$. Note that Condition \ref{cond:ui} is weaker.
To avoid degeneracy we will also require a mild non-singularity condition on the block variance of the original process $(X_t)$. 
\begin{condition}
\label{cond:regularity_new}
For all sequences $(m_n) \in \N$ with $m_n \to \infty$ and $m_n < n$, the process $(X_i)$ satisfies that \color{black}$\lim_{n \to \infty} \min_{1 \leq i \leq n-m_n} \|X_i + \ldots + X_{i+m_n}\|^2 = \infty$\color{black}.
\end{condition}
This non-singularity condition is a very natural one. A simple counter-example may be given for the case where absence of such assumption entails failure of even the Central Limit Theorem. For $t \in \N$, consider the process $X_t=\varepsilon_t -\varepsilon_{t-1}$, and $\varepsilon_i$ are i.i.d.non-Gaussian with mean $0$ and variance $\sigma^2>0$. Then for $n \in \N$, clearly $S_i=\varepsilon_i- \varepsilon_0$ for $1 \leq i \leq n$, and thus both Condition \ref{cond:regularity_new} and Central Limit Theorem $S_n/\|S_n\| \Rightarrow N(0,1)$ fails to hold. With this condition, we begin by presenting a Gaussian approximation for the truncated partial sum process
\begin{equation}\label{trunc}
S_i^{\oplus}:=\sum_{j=1}^i (X_j^{\oplus}- \IE(X_j^{\oplus}) ), \text{ where $X_i^{\oplus}=T_{n^{1/p}}(X_i)$, $i=1, \cdots, n,$}
\end{equation}
with $T_b(w)=\max \{\min\{w,b\}, -b\}$. The following is the first main result of this paper. 
\begin{thm}\label{corol:new}
Let $p>2$. For the process $(X_t)_{t}$, assume Conditions \ref{cond:ui}, \ref{cond:regularity_new}, and \ref{cond:fdm3} with 
\begin{equation}\label{eq:Alowerbound}
A >A_0: =\max\left\{\frac{p^{2}-p-2+(p-2) \sqrt{p^{2}+10 p+1}}{4 p} \ , 1\right\}. 
\end{equation}
Then there exists a Gaussian process $Y_t$ with $\operatorname{Cov}(X_s, X_t)=\operatorname{Cov}(Y_s, Y_t)$, such that 
  \begin{equation}\label{corol_new}
     \max_{1 \leq i \leq n}|S_i - \sum_{j=1}^i Y_j|=o_{\IP}(n^{1/p}\sqrt{\log n}). 
  \end{equation}
  In fact, there also exists a Gaussian process $Y_t^{\oplus}$, with $\operatorname{Cov}(Y_s^{\oplus}, Y_t^{\oplus})=\operatorname{Cov}(X_s^{\oplus}, X_t^{\oplus})$, such that 
  \begin{equation}\label{corol_new_trunc}
     \max_{1 \leq i \leq n}|S_i - \sum_{j=1}^i Y_j^{\oplus}|=o_{\IP}(n^{1/p}). 
  \end{equation}
\end{thm}

\color{black}

Here it is important to note that, although \eqref{corol_new_trunc} has a better rate than \eqref{corol_new}, the approximating process has covariance structure matched with the truncated value of the original process $X_i$. However, we still present this result since it shows that theoretically it is possible to achieve the optimal $n^{1/p}$ rate without the stronger non-singularity condition as \cite{KarmakarWu2020}. Proving such a result also necessisates novel techniques which are different compared to both \cite{KarmakarWu2020} and \cite{kmt}. 

Finally, if one were to assume non-singularity condition as written below, we show that it is possible to achieve $n^{1/p}$ rate even with the approximating process matching covariances exactly with the original $(X_t)$ process.
\color{black}
\begin{condition}
\label{cond:regularity}
The series $(X_i)$ satisfies the following condition: There exists a constant $c>0$ and $l_0 \in \mathbb N$, such that for all $l \ge l_0$, \color{black}$\min_{1\le j \le n-l+1} \|X_{j}+\ldots + X_{j+l-1}\|^2 / l \geq c$.\color{black}
\end{condition}

\noindent At the cost of making this extra assumption, we are also able to improve the decay rate condition on $\Theta_{i,p}$ from that in Theorem \ref{corol:new}, matching exactly the optimal cut-off given in \cite{KarmakarWu2020}.

\begin{thm}\label{thm:multGA}
Assume the process $(X_t)_{t \geq 1}$ satisfies Conditions \ref{cond:ui}, \ref{cond:regularity} and \ref{cond:fdm3} with
\begin{eqnarray}\label{eq:Alowerbound2}
A > A_0':= \max\Bigg\{\frac{p^{2}-4+(p-2) \sqrt{p^{2}+20 p+4}}{8 p}, 1\Bigg\}.
\end{eqnarray} 
Then, there exists a Gaussian process $(Y_t)$ with $\operatorname{Cov}(Y_s, Y_t):=\operatorname{Cov}(X_s, X_t)$, such that 
  \begin{equation}\label{eq:multGA_nocondn}
    \max_{1 \leq i \leq n}|S_i - \sum_{j=1}^i Y_j|=o_{\IP}(n^{1/p}). 
  \end{equation}
\end{thm}
\color{black}

\ignore{
\noindent Observe that in Theorem \ref{thm:GA} we assume a decay rate that is slightly higher than that in Theorem \ref{thm:multGA}. This higher decay rate allows us to do away with the non-singularity condition. In fact, we can go all the way and connect the approximating Gaussian process with the variance of the partial sum of $X_i$ process itself, without ever requiring Condition \ref{cond:regularity}. In the following theorem, we provide one such result, although we incur a penalty of $\sqrt{\log n}$ in our rate. 
}
\color{black}

\vspace{-0.35in}
\subsection{Gaussian approximation with independent increments}
\label{subsec:thm:GA}
In addition to having a natural interpretation, the Gaussian approximations in the previous Section \ref{subsec:thm:new_GA} also enjoy applicability when information about the covariance structure of the original process is available, such as for stationary processes \cite{xiaowu} or processes from a defined parametric structure. However, for a general non-stationary processes, the precise correlation structure of $X_t$ process may not be available, and therefore simulating the $Y_t$ process becomes a challenge. Therefore, it is important to investigate if we can further obtain a Gaussian approximation of the form \eqref{sakhanenkoind}, i.e. involving Brownian motion with independent increments, where the involved $\IE(S_i^2)$ is estimable. The following two theorems address these issues and yield Gaussian approximations with this desired structure. Our first result is analogous to Theorem \ref{corol:new}. However, in this result, we no longer require any non-singularity condition, and yet we almost recover the optimal $n^{1/p}$ rate (up to a $\log$ factor). Again, we recover the exact optimal rate if our Gaussian approximation involves the moments of the truncated process. \color{black}
\begin{thm}
\label{thm:GA}
For the process $(X_t)_{t \geq 1}$, assume Conditions \ref{cond:ui} and \ref{cond:fdm3} with $A>A_0$; see \eqref{eq:Alowerbound}. Then there exists a Brownian motion $\mathbb{B}(\cdot)$, such that 
\begin{equation}
  \label{eq:GA}
  \max_{1\leq j \leq n} |S_j - \mathbb{B}(\IE(S_j^{\oplus^2}))| = \operatorname{o}_{\IP} (n^{1/p}).
\end{equation}
Further, it holds that 
\begin{equation}
  \label{eq:improvedGA}
  \max_{1\leq j \leq n} |S_j - \mathbb{B}(\IE(S_j^{2}))| = \operatorname{o}_{\IP} (n^{1/p}\sqrt{\log n}).
\end{equation}
\end{thm}

\color{black} 

\color{black}
A similar remark to the one following Theorem \ref{corol:new} is in order. Note that, in Theorem \ref{thm:GA}, again using the moments of the original process in the Gaussian approximation entails a penalty of $\sqrt{\log n}$ in our rate. However, it turns out that under the more stringent non-singularity condition of Theorem \ref{thm:multGA}, we are not only able to recover the optimal rate of $n^{1/p}$ from using the $X_t$ process itself, but also able to relax the decay rate. 
\color{black}


\begin{thm}
\label{thm:GA2}
Under conditions of Theorem \ref{thm:multGA}, there exists a Brownian motion $\mathbb{B}(\cdot)$ such that 
\begin{equation}
  \label{result:GA_ver2}
  \max_{1\leq j \leq n} \left|S_j - \mathbb{B}(\IE(S_j^{2}))\right| = \operatorname{o}_{\IP} (n^{1/p}).
\end{equation}

\end{thm}
\begin{remark}\label{remark:uiexample}
Necessity of the truncated uniform integrability Condition \ref{cond:ui}:
{\rm We show that the uniform integrability condition is necessary as otherwise the Gaussian approximation might fail. Let $n>2$. \color{black} Let $X_1, X_2, \ldots$ be independent with $\IP(X_i = \pm (i+1)^{1/p})=1/(i+1)$ and $\IP(X_i =\pm 1)= 1/2 - 1/(i+1)$. Note that, Condition \ref{cond:ui} is violated since $\max_{1\leq i \leq n}\IE[|X_i|^p \mathbb{I}\{|X_i|^p> n/2\}] = 2.$ For the sake of contradiction, suppose the Gaussian approximation \eqref{result:GA_ver2} holds, which implies
 \begin{equation} \label{contra}
\max_{1 \leq i \leq n}|X_i - (\mathbb{B}(\IE(S_i^2)) - \mathbb{B}(\IE(S_{i-1}^2)))|=o_{\IP}(n^{1/p}).
 \end{equation}
 Since $X_i$'s are independent, and $\max_{1 \leq i \leq n}\IE(X_i^2)\leq 2^{2/p}+1$, therefore, by property of increments of Brownian motion, $ \max_{1 \leq i \leq n} | \mathbb{B}(\IE(S_i^2)) - \mathbb{B}(\IE(S_{i-1}^2))| =O_{\IP}( (\log n)^{1/2})$. Thus, if one assumes that \eqref{contra} is true, then we will have
  $ 
    \max_{1 \leq i \leq n} |X_i| =o_{\IP}(n^{1/p}).
  $ 
Now we show that the latter is false. Clearly, $|X_i|\leq n^{1/p}/2$ w.p. 1 if $i \leq n/2^p -1$, and therefore
\begin{align*}
  \IP\left(\max_{1 \leq i \leq n} |X_i| > \frac{n^{1/p}}{2} \right) = 1 - \! \!\! \!\!\!\prod_{i=\lceil n/2^p \rceil \vee 1}^n \! \! \! \IP\left(|X_i| \leq \frac{n^{1/p}}{2}\right) &\geq 1-\left(1-\frac{2}{n+1}\right)^{n(1-\frac{1}{2^{p-1}})} \\ &\to 1- e^{2^{2-p}-2},
\end{align*}
as $n \to \infty$. This contradiction shows that Theorem \ref{thm:GA2} fails to hold. This vouches for the necessity of our uniform integrability condition; clearly, the reason the Gaussian approximation fails to hold in this example is due to Condition \ref{cond:ui} not being satisfied. It can be noted that, in this example, Theorem \ref{thm:ind} does not apply; \eqref{shakhanenkocondn} can be verified to be violated in this case. 
\color{black}
\ignore{Let $X_1, \ldots, X_n$ be i.i.d. with $\IP(X_i = \pm n^{1/p})=1/n$ and $\IP(X_i =\pm 1)= 1/2 - 1/n$, $1 \le i \le n$. Note that, Condition \ref{cond:ui} is violated since $\IE[|X_i|^p \mathbb{I}\{|X_i|^p> n/2\}] = 2.$ For the sake of contradiction, suppose the Gaussian approximation \eqref{result:GA_ver2} holds, which implies
 \begin{equation} \label{contra}
\max_{1 \leq i \leq n}|X_i - (\mathbb{B}(\IE(S_i^2)) - \mathbb{B}(\IE(S_{i-1}^2)))|=o_{\IP}(n^{1/p}).
 \end{equation}
Since $\IE(S_i^2) = i \IE(X_1^2)$ and $\IE(X_1^2) = 1 - 2/n + 2 n^{2/p -1} \to 1$, by property of increments of Brownian motion, $ \max_{1 \leq i \leq n} | \mathbb{B}(\IE(S_i^2)) - \mathbb{B}(\IE(S_{i-1}^2))| =O_{\IP}( (\log n)^{1/2}).$ Thus, if one assumes that \eqref{contra} is true, then we will have
  $ 
    \max_{1 \leq i \leq n} |X_i| =o_{\IP}(n^{1/p}).
  $ 
However, the latter is false, since as $n \to \infty$
  \[ \IP\left(\max_{1 \leq i \leq n} |X_i| > \frac{n^{1/p}}{2} \right) = 1 - (1 - \frac{2}{n})^n \to 1 - e^{-2}. \]
This shows that Theorem \ref{thm:GA2} fails to hold. \color{black}This vouches for the necessity of our uniform integrability condition; clearly the reason the Gaussian approximation fails to hold in this example is due to Condition \ref{cond:ui} not being satisfied. \color{black}}
 } 
\end{remark}
\subsection{Examples}\label{ssc:examples}
We now show some examples of non-stationary time series which satisfy Condition \ref{cond:fdm3}. \color{black} For $t \in \Z$, let $\mathcal{F}_t = (\ldots, \varepsilon_{t-1}, \varepsilon_t)$, where $\varepsilon_t$ are i.i.d. random variables. Consider the model 
\begin{equation} \label{generalnon-stationary}
  X_t= g(\theta_t, \mathcal{F}_t), \ 1\leq t\leq n,
\end{equation}\color{black}
where $\theta_t \in \Gamma$, a parameter space, and \color{black}$g(\cdot, \mathcal{F}_t):\Gamma \to \R$ is a progressively measurable function such that the process $X_t(\theta)= g(\theta, \mathcal{F}_t)$ is well-defined\color{black}. {{We can view (\ref{generalnon-stationary}) as a general modulated stationary process. \cite{adak1998time} and \cite{zhao2013inference} considered the special case of multiplicative modulated stationary processes with a linear form.}} Define the functional dependence measures as
\begin{align}
  \delta^{\Gamma}_{p}(k) := \sup_{\theta \in \Gamma} \|g(\theta, \mathcal{F}_t) -g(\theta, \mathcal{F}_{t, \{t-k\}}) \|_p \geq 
  \sup_t \| g(\theta_t, \mathcal{F}_t)- g(\theta_t, \mathcal{F}_{t, \{t-k\}}) \|_p =: \delta^X_p(k). 
\end{align}
Thus, we only need to assume that $\Theta_{i,p}^{\Gamma}:= \sum_{k=i}^{\infty} \delta_p^{\Gamma}(k)$ satisfies Condition \eqref{cond:fdm3}.
We mention a couple of examples from the general class of non-stationary processes satisfied by \eqref{generalnon-stationary}.
\vspace{-0.1in}
\subsubsection{{Cyclostationary process}}
Taking $\theta_t= \phi_{t\text{ mod }T}$ in \eqref{generalnon-stationary} for some period $T$, and $\{ \phi_t\}_{t=1}^T\in \Gamma$, yields cyclostationary process. These can be thought of as generalizations of stationary processes, incorporating periodicity in its properties, and were introduced as a model of communications systems in \cite{Bennet} and \cite{franks1969signal}. Apart from communication and signal detection, cyclostationary processes have enjoyed wide use in econometrics \cite{PARZEN1979137}, atmospheric sciences \cite{atmos} and across many other disciplines- the reader is encouraged to look into \cite{Gardner1994CyclostationarityIC}, \cite{NAPOLITANO2016385}, and the references therein for an introduction and a comprehensive list of all its applications. Despite this huge literature, there is no unified asymptotic distributional theory for the cyclostationary processes. Our Gaussian approximation result allows a systematic study of asymptotic distributions of statistics of such processes.
{\ignore{
\color{black}
Let $\varepsilon_i$ be i.i.d. random elements and $T \in \mathbb{N}$. 
Consider the special case of autoregressive cyclostationary process
\begin{equation}\label{eq:cyclostationary}
  X_i=F(\phi_{i \mod T}, X_{i-1}, \varepsilon_i).
\end{equation}
Let $X_i(\theta)=F(\theta, X_{i-1}(\theta), \varepsilon_i)$. Assume that  $\sup_{\theta \in \Gamma} \|K(X_1(\theta))\|_p < \infty$, where
\begin{equation}\label{eq:lipscitz}
 K(x)=\sup_{\theta \neq \theta'}\frac{F(\theta, x, \varepsilon_0) - F(\theta', x, \varepsilon_0)}{|\theta - \theta'|}.
\end{equation}
 Then one can generalize arguments in Lemma 4.5 of \citet{richterdahlhaus2017} to show that the process \eqref{eq:cyclostationary} is well-approximated by $\{X_i(\phi_{i \mod T})\}$, which is a process of the form \eqref{generalnon-stationary}. 
 }} 
\color{black}
\subsubsection{Locally stationary process}
In \eqref{generalnon-stationary}, let $\Gamma=[0,1]$. Assume that $g$ is stochastic Lipschitz continuous for some constant $L>0$, such that for all $\theta, \theta',$
\begin{equation}
   \| g(\theta, \mathcal{F}_t)- g(\theta', \mathcal{F}_t)\|_p \leq L | \theta- \theta'|.
\end{equation}
Then, the processes $X_{t,n}:= g(t/n, \mathcal{F}_t)$ are locally stationary in view of the approximation\[ \|X_{t, n}- X_t(\theta)\|_p \leq L |t/n- \theta| \text{ if } t/n \in (\theta- \color{black}\Delta\color{black}, \theta+\color{black}\Delta)\ \text{for some $\Delta>0$}\color{black}.\]
\citet{dalhaus97, dalhaus200a} introduced locally stationary processes in terms of time-varying spectrum. \cite{richterdahlhaus2017} provided a general asymptotic theory for such processes. For further examples, see \cite{zhang&wu2021}. 

Consider the special case of locally stationary version of Volterra processes, defined as follows:
  \begin{equation}\label{volterra}
    X_t = \sum_{0 \leq j_1 < \ldots < j_i} a(j_1, \ldots, j_i, \frac{t}{n}) \ \varepsilon_{t-j_1} \ldots \varepsilon_{t-j_i},
  \end{equation}
  where $\varepsilon_i$'s are i.i.d. with mean $0$, $\|\varepsilon_0\|_p < \infty$, $p>2$, and $a: \R^i \times [0,1] \to \R$ are called $i$-th order Volterra kernels. Then elementary calculations show that for a constant $c_p$ depending only on $p$, 
  \begin{equation}
    \delta_p(l)^2 \leq c_p \|\varepsilon_0\|_p^{2i} \sup_{k} A_{k,l,i}, \text{ where } A_{k,l,i}= \sum_{\substack{{0 \leq j_1 < \ldots < j_i}, \, {l \in \{j_1, \ldots, j_i\}}}} a^2(j_1, \ldots, j_i, \frac{k}{n}) < \infty. 
  \end{equation}

\ignore{
Let $\varepsilon_i$ be i.i.d. random elements and $T \in \mathbb{N}$. 
Consider the special case of autoregressive cyclostationary process
\begin{equation}\label{eq:cyclostationary}
  X_i=F_i(\phi_{i \mod T}, X_{i-1}, \varepsilon_i).
\end{equation}
 Assume that for some $x_0$, $\sup_{\theta \in \Theta}\sup_{1 \leq i \leq T} \|F_i(\theta, x_0, \varepsilon_1) \|_p < \infty$, and
\begin{equation}\label{eq:lipscitz}
  \sup_{1 \leq i \leq T}\| L_i(\varepsilon)\|_p <1 \ \text{where } L_i(\varepsilon)=\sup_{\theta \in \Theta}\sup_{x \neq y} \frac{F_{i}(\theta, x, \varepsilon_i)-F_{i}(\theta, y, \varepsilon_i)}{|x-y|}.
\end{equation}

 Under \eqref{eq:lipscitz} for a fixed $\theta \in \Theta$, following \citet{wushao2004}'s argument, we have the Geometric Moment Contraction: $\delta_{p,\theta}(j)=O(\rho^j)$ where $\rho=\sup_{1 \leq i \leq T}\| L_i(\varepsilon)\|_p$. Our Condition \ref{cond:fdm3} is much weaker than the much more common Geometric Moment Contraction (GMC) property (see \cite{Wu2005} Theorem 1). 
This shows that \eqref{eq:cyclostationary} satisfies our assumption \eqref{eq:Thetaip}, and therefore, Theorem \ref{thm:improvedGA} is applicable to yield the optimal $n^{1/p}$ rate of Gaussian approximation.

\subsubsection{\textbf{Locally stationary Volterra processes}}
  Consider the Volterra processes, which are non-linear processes of fundamental importance. We consider a locally stationary version of Volterra processes, defined as follows:
  \begin{equation}\label{volterra}
    X_t = \sum_{i=1}^{\infty} \sum_{0 \leq j_1 < \ldots < j_i} a_i(j_1, \ldots, j_i, \frac{t}{n}) \ \varepsilon_{t-j_1} \ldots \varepsilon_{t-j_i},
  \end{equation}
  where $\varepsilon_i$'s are i.i.d. with mean $0$, $\|\varepsilon_0\|_p < \infty$, $p>2$, and $a_i: \R^i \times [0,1] \to \R$ are called $i$-th order Volterra kernels. Let 
  \[ A_{k,l,i}= \sum_{\substack{{0 \leq j_1 < \ldots j_i}\\{l \in \{j_1, \ldots, j_i\}}}} a_i^2(j_1, \ldots, j_i, \frac{k}{n}) < \infty. \]
  Then elementary calculations show that for a constant $c_p$ depending only on $p$, 
  \begin{equation}
    \delta_p(l)^2 \leq c_p \sum_{i=1}^{\infty} \|\varepsilon_0\|_p^{2i} \sup_{k} A_{k,l,i}.
  \end{equation}
  Assume that for some $A>A_0$ and $B$, 
  \[\sum_{i=1}^{\infty}\|\varepsilon_0\|_p^{2i} \sup_k \sum_{\substack{{0 \leq j_1 < \ldots j_i}\\{j_i \geq m}}} a_i^2(j_1, \ldots, j_i, \frac{k}{n}) = O(m^{-1-2A} (\log m)^{-2B}). \]
  Then, 
  \[ \sum_{l=m}^{\infty} \delta_p(l)^2 = O(m^{-1-2A} (\log m)^{-2B}), \]
  which implies that \eqref{eq:Thetaip} is satisfied. That $\{|X
_t|^p\}_{t \geq 1}$ are uniformly integrable follows from noting that $\varepsilon_i$'s are i.i.d. and 
\[ |X_t| \leq \sum_{i=1}^{\infty} \sum_{0 \leq j_1 < \ldots < j_i} |\varepsilon_{t-j_1} \ldots \varepsilon_{t-j_i}| \sup_{s \in [0,1]} a_i(j_1, \ldots, j_i, s) .\]
Thus the sufficient conditions for Theorem \ref{thm:improvedGA} are true and one can obtain an optimal Gaussian approximation. 
}

\subsection{Outline of the proof of theorems}\label{ssc:outline}
Our proofs are quite involved and are given in Sections \ref{sec:section2proofs} and \ref{sec2proofsnew}. In particular, Theorems \ref{corol:new} and \ref{thm:GA} are based on similar assumptions (in fact Theorem \ref{thm:GA} works with a weaker set of conditions); and in the same vein, Theorems \ref{thm:multGA} and \ref{thm:GA2} require exactly the same conditions. Therefore, these two pairs of theorems are proven with each other. In particular, all the four theorems follow a general recipe of the proof outlined below.
\begin{itemize}
  \item \textbf{Truncation:} In Proposition \ref{prop1}, we truncate our process at level $n^{1/p}$ in order to exploit the uniform integrability condition, which is necessary due to non-stationarity.

\item \textbf{$m$-dependence:} In the second step, we use the $m$-dependence approximation in Proposition \ref{prop2} where $m$ increases with $n$. This limits the arbitrary non-stationary dependency structure to those only up to $m$ lags, and enables us to treat our series much like a stationary time series. We provide an optimal choice of $m$ so that the error rate of $n^{1/p}$ is achieved. 

\item \textbf{Blocking:} Our blocking step in Proposition \ref{prop3} is quite different from that in \cite{KarmakarWu2020} as well as \cite{kmt}; we consider a two-step blocking, with an inner layer of blocks of size $m$ being then combined into an outer layer of blocks of size $3$. This enables us to do the required mathematical manipulation to obtain an explicit form of the variance in terms of $m$-dependent processes.

\item \textbf{Conditional and Unconditional Gaussian approximation: } With the blocking step as mentioned above, we condition on the shared $\varepsilon$'s between the outer blocks (that occur at both the boundaries of each block). This results in conditional independence and thus we can use \cite{Sakhanenko2006}'s Theorem 1. Then we lift the conditioning random variables (the boundary $\varepsilon$'s) by taking another expectation over them, and apply the Theorem 1 from \cite{Sakhanenko2006} again to obtain the unconditional Gaussian approximation.

\item \textbf{Regularization of Variance: } From the variance in terms of $m$-dependent blocked processes as mentioned above, in order to obtain the variance approximation in a practically usable form as mentioned in the theorem, in this step we approximate it by $\IE( (S_i^{\oplus})^2)$ or by variances of sum of blocks in terms of original random process. 

\item \textbf{Final Gaussian approximation: } In this final step, we connect the approximated variance $\IE( (S_i^{\oplus})^2)$ to the new Gaussian process $(Y_i)_{i=1}^n$ (for Theorems \ref{corol:new} and \ref{thm:multGA}), via Propositions \ref{lemmaY} and \ref{lemmaYconstruct}, or to the final variance $\IE(S_i^2)$ (for Theorems \ref{thm:GA} and \ref{thm:GA2}). 
\end{itemize}

\section{Estimating the variance of the approximating Gaussian process}\label{ssc:estvar}

In this section, we address estimating the variance of the approximating process. It is well-known in the time series literature that $S_i^2$ is a poor estimate for $\IE(S_i^2)$. The usual practice is to use a kernel function or a particular weighing-mechanism. Such methods have been used throughout the literature to estimate spectral density matrices for one-dimensional or low-dimensional cases. For stationary processes, we recommend works by \citet{newey-west}, \citet{priestley1981spectral} and \citet{liu_wu_2010} among others for a comprehensive review of research in this direction. As a special case of kernel-based estimates, blocking techniques have been particularly popular in this area. \citet{carlstein} used non-overlapping blocks to consistently estimate $\IE(S_i^2)$ for a stationary process. From a bootstrap perspective, \citet{romano1994} uses non-overlapping blocks of random sizes to define a `stationary bootstrap'. Using the `flat-top kernel' methods of \cite{romano1995}, \cite{politis-white} obtains $O(n^{1/3})$ for the expected optimal block size for the stationary bootstrap. For detailed discussion, readers are encouraged to look into \citet{Lahiri2003ResamplingMF}, which combines ideas from \cite{HALL1985231}, \cite{carlstein}, \cite{carlstein_matched} and many others to deduce various resampling schemes for estimating the variance of a stationary process. 

The blocking method has been quite popular in the literature as a proof technique for obtaining optimal Gaussian approximations. See \cite{KarmakarWu2020}, \cite{zhouzhousinica} and \cite{LiuLin2009} for relevant references. Naturally, since the statements of our Theorems \ref{corol:new}-\ref{thm:GA2} \color{black}do not involve any blocks\color{black}, one may question if we can reach the optimal rate by expressing the variance directly in terms of some blocking mechanism. In the next section, we will provide a result that answers the above question in affirmative. The blocking mechanism we use is somewhat related to the Non-overlapping Block Bootstrap (NBB) method proposed in Chapter 2 of \cite{Lahiri2003ResamplingMF}. We describe the scheme in the following. 
\noindent
Usually the block length $m$ is taken so as $m\to \infty$ with $m/n \to 0$. Define \color{black} for $1\leq a, k, j \leq \lceil n/m \rceil$,
\color{black}\begin{equation} \label{eq:T_k}
  B_a := \sum_{i=(a-1)m + 1}^{am \wedge n} X_i; \ \ T_k= \sum_{a=1}^k B_a^{2} + 2\sum_{a=1}^{k-1} B_aB_{a+1}; \ R_j:= \mathbb{I}\{j/m \notin \N\}\sum_{i=\lfloor j/m \rfloor m+1}^j X_i .
\end{equation}\color{black}
Note that $S_j = \sum_{a=1}^k B_a + R_j$, where $k = \lfloor j/m \rfloor$.
We shall estimate $\IE(S_j^2)$ by the following `Block-based Running Variance' (BRV) estimator $\mathcal{T}_j$ where
\begin{equation}\label{Tj}
  \mathcal{T}_j:=T_{\lfloor j/m \rfloor} + R_j^2 + 2B_{\lfloor j/m \rfloor} R_j \text{ for all $1 \leq j \leq n$}
\end{equation}
simultaneously. Since $\mathcal{T}_j$'s may be negative, so instead of Brownian motion we use two-sided Brownian motion. A two-sided Brownian motion is defined as $\mathbb{W}(t)= 
  \mathbb{B}_1(t){\bf 1}_{t \geq 0} + 
  \mathbb{B}_2(-t){\bf 1}_{t < 0},$
where $\mathbb{B}_1$ and $\mathbb{B}_2$ are two independent standard Brownian motions starting at $0$. 

Next, we provide some theoretical properties of the BRV estimator $\mathcal{T}_j$. In particular, we bound the uniform deviation probability of $\mathcal{T}_j$. Such a deviation inequality for non-stationary processes is novel to the best of our knowledge. Thus we state it as a standalone result. 
\color{black}
\subsection{A maximal quadratic large deviation bound}
Quadratic large deviation bounds have a long history that started with the seminal work by \citet{hansenwright} and \citet{wright}. See \cite{rudelson} for an extensive overview. These are popularly referred as Hanson-Wright type inequalities in the literature. Subsequent work by \cite{BERCU199775}, \cite{KAKIZAWA2007992} and others established \textit{moderate deviation principles} for quadratic forms of stationary Gaussian processes. Moving beyond sub-Gaussianity, \citet{xiaowu} and \citet{zhang&wu2021} generalized the Hanson-Wright inequality for stationary process with finite polynomial moments and locally stationary processes, respectively. In this section we aim to (i) develop a maximal inequality i.e., derive tail probability bounds for the maximal partial sum, and (ii) relax the stationarity assumption by providing a result for the general non-stationary processes. Our proof is similar to the Theorem 6.1 of \cite{zhang&wu2021}; however, it differs in a crucial step. Since we aim to provide a maximal inequality, we use Borovkov's version of Nagaev inequality (\cite{borovkov}), instead of the usual bound of \cite{nagaev}. This, in particular, changes the treatment of a few important terms in our proof compared to that in \cite{zhang&wu2021}. Moreover, we also tackle the case when $2<p\leq 4$, something that is usually absent from other Hanson-Wright type inequalities in the literature. 
\begin{thm}\label{thm:maxpartialquad}
Let $p>2$. Assume Condition \ref{cond:fdm3} holds for $\Theta_{i,p}$. Let $Q_n=\sum_{1 \leq s \leq t \leq n}a_{s,t}X_sX_t$, with
$a_{s,t}=0$ if $|s-t|>\mathcal{D}_n$ for some $\mathcal{D}_n\leq n$, and $\sup |a_{s,t}|\leq 1$. Denote 
\begin{equation}\label{Vj}
R_k = \sum_{j=1}^k (V_j-\IE(V_j)), \mbox{ where } 
V_k=\sum_{t=(k-1)\mathcal{D}_n +1}^{(k\mathcal{D}_n) \wedge n} \sum_{1 \leq s \leq t}a_{s,t} X_s X_t, \text{ for } 1\leq k \leq \lceil n/\mathcal{D}_n \rceil.
\end{equation}
Then there exists constants $C_p$, depending only on $p$, such that for all $x > 0$,
\begin{equation}\label{eq3.4}
  \IP\left(\max_{1 \leq k \leq \lceil n/\mathcal{D}_n \rceil} | R_k | \geq x \right) \leq \begin{cases}
  & C_px^{-p/2} n\mathcal{D}_n^{p/4}\mu_{p,A}^p, \ 2<p \leq 4,\\
    & C_px^{-p/2}n \mathcal{D}_n^{p/2-1}\mu_{p,A}^p + C_p \exp \left(- \frac{C_p x^2}{n \mathcal{D}_n \mu_{4,A}^4} \right), \ p > 4.
  \end{cases}
\end{equation}
\end{thm}

\noindent The proof is given in Appendix Section \ref{proofmaxquad}. We emphasize that to avoid notational cumbersomeness, in \eqref{eq3.4} we have used same notation $C_p$ to denote multiple constants, each depending solely on $p$.
 \begin{remark}
   In view of \eqref{trunc}, $\delta_p^{\oplus}(j) \leq \delta_p(j)$ is satisfied by the functional dependence measure of the truncated process. Therefore, Theorem \ref{thm:maxpartialquad} also holds for $X_s$ replaced by $X_s^{\oplus} - \IE(X_s^{\oplus})$.
 \end{remark}
 \begin{remark}
   The bound in Theorem \ref{thm:maxpartialquad} should be contrasted with the bound obtained in Theorem 6 of \cite{zhang&wu2021}. In fact, our proof works for $A>1/2-1/q$ and matches their non-uniform bound for the corresponding case. A similar argument can be followed to yield a bound for a process satisfying $\mu_{p,A}< \infty$ for some general $A$. In view of our maximal inequality holding true for general non-stationary process, Theorem \ref{thm:maxpartialquad} is a more general result than those found in the literature. 
 \end{remark}

\subsection{Gaussian approximation rate with estimated variance} \label{subsec:est}
 Theorem \ref{thm:maxpartialquad} is useful in arriving at the estimation error of $\mathcal{T}_i$ as an estimate of $\IE(S_i^2)$. To begin with,\ignore{in Proposition \ref{prop:Bj2-EBj2}, we will find the rate of convergence of the plug-in estimate of the variance of the two-sided Brownian motion. \ignore{Subsequently, we will use the increment property of the two-sided Brownian motion, to arrive at the optimal rate of convergence for $\max_{i \leq n}\left| \mathbb{W}(\mathcal{T}_i)-\mathbb{W}(\IE(S_i^2)) \right |$}.} 
 note that $\mathcal{T}_i/2$ can be written in the form \eqref{Vj} with $a_{s,t}=1/2$ when $s=t$, and in general $|a_{s,t}|=0$ when $ |s-t|\geq 2m$ and $\sup |a_{s,t}|\leq 1$. Thus, taking $\mathcal{D}_n=2m$, Theorem \ref{thm:maxpartialquad} implies that, 
\begin{equation} \label{approx1}
  \max_{1 \leq k \leq \lfloor n/m \rfloor}\bigg|\sum_{j=1}^k (B_j^2 + 2B_j B_{j+1} - \IE[B_j^2 + 2B_j B_{j+1}])\bigg|=O_{\IP}(n^{\max\{2/p, 1/2\}} m^{1/2}). 
\end{equation}\color{black}
Moreover, by Lemma \ref{pmoment}, $\max_{1 \leq j \leq \lfloor n/m \rfloor}\IE[\max_{1 \leq k\leq m} |X_{mj+1} + \ldots +X_{mj+k}|^p]=O(m^{p/2})$. Hence,
\begin{equation}\label{approx2}
  \max_{1 \leq i \leq n} \bigg|\mathcal{T}_i - \sum_{j=1}^{\lfloor i/m \rfloor} (B_j^2 + 2B_j B_{j+1})\bigg|=O_{\IP}(n^{\max\{2/p, 1/2\}} m^{1/2}).
\end{equation} 
by Markov's inequality. \color{black}
Note that \eqref{approx2} takes care of the stochastic error of $\mathcal{T}_i$ as an estimate of $\IE(S_i^2)$ for $1 \leq i \leq n$. For the bias part, 
we need to control the order of the cross-product terms $\IE(B_iB_j
)$ for $i\neq j$. The following lemma, whose proof we give in Section \ref{prooflemmacp}, is thus necessitated.
\begin{lemma} \label{lemmacp}
  Let Condition \ref{cond:fdm3} hold with $A>1$. Then for $B_j$ as defined in \eqref{eq:T_k}, it holds that
  \begin{align}
    \max_{1\leq k \leq \lfloor n/m \rfloor} |\IE(B_k B_{k+1})| = O(1), \ \max_{1 \leq k \leq \lceil n/m \rceil} \sum_{i: |i-k|\geq 2}|\IE( B_iB_k)| &= O(m^{1-A}). \label{eq:blockproduct}
  \end{align}
\end{lemma}
\noindent Observe that \eqref{eq:blockproduct} readily yields
\begin{equation}\label{approx3}
  \max_{1 \leq i \leq n}\bigg|\IE(S_i^2) - \sum_{j=1}^{\lfloor i/m \rfloor} \IE(B_j^2 + 2B_j B_{j+1})\bigg|=O(nm^{-A}).
\end{equation}
Now, \eqref{approx1}, \eqref{approx2} and \eqref{approx3} can be summarized into the following proposition.
\begin{proposition}\label{prop:Bj2-EBj2}
Assume $p >2 $ and let Condition \ref{cond:fdm3} hold for $\Theta_{i,p}$ with $A>1$. Recall $B_j$ from \eqref{eq:T_k}, for a general $m \in \N$. Then the following holds:
\begin{equation} \label{1stprop3.3}
  \max_{1 \leq i \leq n}|\mathcal{T}_i - \IE(S_i^2)|=O_{\IP}(n^{\max\{2/p, 1/2 \}} m^{1/2}+ nm^{-A}).
\end{equation}
In particular, with $m \asymp n^{\zeta_1}$, where $\zeta_1 = \min\{1, 2-4/p\}/(1+2A)$, \eqref{1stprop3.3} implies
\begin{equation}\label{eq:BMbootstrap}
\max_{1\leq i \leq n}| \mathbb{W}(\mathcal{T}_{i })-\mathbb{B}_1(\IE(S_i^2)) |= O_{\IP^{*}}(n^{(1-A\zeta_1)/2}\sqrt{\log n}),
\end{equation}
where $\IP^*$ refers to the conditional distribution after observing $X_1,\ldots, X_n$, and $\IB_1(\cdot)$ is the same Brownian motion defining the positive half-line of $\mathbb{W}(\cdot)$. 
\end{proposition}
\color{black}Our choice of $m$ balances the bias ($nm^{-A}$) and the stochastic error ($n^{\max\{2/p, 1/2 \}} m^{1/2}$) together, and yields the rate in \eqref{eq:BMbootstrap} by increment property of Brownian motions. \color{black} However, the approximation rate in \eqref{eq:BMbootstrap} is worse than what we obtain in Section \ref{sec:main_thm}. But this also means that one can only assume moments slightly higher than $4$ and still achieve this rate. More importantly, a natural question is if we can relax our decay condition in Theorem \ref{thm:GA} when we are allowed to assume $p$ finite moments but want to achieve this comparatively large approximation rate. In other words, at the cost of sub-optimal rate, which anyway is the best for the empirical version, can we allow decay rate $A$ to be smaller? In what follows, we answer this question in affirmative. 

\begin{thm}\label{thm:suboptimal}
Let $p > 2$. Assume that the decay Condition~\ref{cond:fdm3} holds with $A>1$. Further grant the truncated uniform integrability Condition \ref{cond:ui}. Then there exists a Brownian motion $\mathbb{B}(\cdot)$ such that 
\begin{equation}
  \label{result:GA_suboptimal}
  \max_{1\leq j \leq n} \left|S_j - \mathbb{B}(\IE(S_j^2))\right| = \operatorname{o}_{\IP} (n^{(1-A\zeta_1)/2}\sqrt{\log n}).
\end{equation}
\end{thm}
\begin{remark}
  Note that in \eqref{result:GA_suboptimal} we no longer need the lower bound \eqref{eq:Alowerbound2}. 
\end{remark}
\subsection{Gaussian approximation without cross product blocks}\label{sec:nocp}
Having explored the asymptotic properties of BRV estimator $\mathcal{T}_j$ as an estimate of $\IE(S_j^2)$ for $1 \leq j \leq n$, let us discuss a natural variant of $\mathcal{T}_j$. Interestingly, in $\mathcal{T}_j$ we have included the cross-product terms $B_i B_{i+1}$, as opposed to another possible estimate $\mathcal{T}_i^{-}$ which can be defined without them:
\begin{equation}\label{T_noblock}
  \mathcal{T}_i^{-}= \sum_{j=1}^{\lfloor i/m \rfloor} B_j^{2} + R_i^2.
\end{equation}
An application of Theorem \ref{thm:maxpartialquad} and \eqref{eq:blockproduct} similar to that in Proposition \ref{prop:Bj2-EBj2} show $\mathcal{T}_i^{-}$ satisfies
\begin{equation}\label{result:GA_nocp}
  \max_{1 \leq i \leq n} |\mathcal{T}_i^{-} - \IE(S_i^2)|= O_{\IP}(n^{\max\{2/p, 1/2 \}} m^{1/2}+ nm^{-1})
\end{equation}
under Condition \ref{cond:fdm3}. The above bound is worse than (\ref{1stprop3.3}) and it is minimized at $m\asymp n^{\zeta_2} $, $\zeta_2=\min \{1, 2-4/p\}/3$. Since $A> 1$, $\zeta_2 < \zeta_1$, and therefore
\begin{equation}\label{bm:result_nocp}
  \max_{1 \leq i \leq n} |\mathbb{W}(\mathcal{T}_i^{-}) - \mathbb{B}(\IE(S_i^2)) |= O_{\IP^{*}}(n^{(1-\zeta_2)/2}\sqrt{\log n}).
\end{equation} Thus the conditional version (\ref{eq:BMbootstrap}) using $\mathcal{T}_i^{-}$ is also worse.

\ignore{ In the following section, we discuss the error rate of $\mathcal{T}_i^{-}$ as an estimator of $\IE(S_i^2)$, and explore why it is a worse choice compared to $\mathcal{T}_i$. Recall $T_i^{-}$ from \eqref{T_noblock} and $T_i^{\diamond}$ from \eqref{Ti_Mies}. It should be noted that if we go back to Section \ref{sec:nocp}, $\IE[\mathcal{T}_i]$ outperforms $\IE[\mathcal{T}_i^{-}]$ and $\IE[\mathcal{T}_i^{\diamond}]$ in terms of the Gaussian approximation rate even with our new choice $m=n^{\min\{\frac{1}{1+2A}, \frac{2p-4}{p+2Ap}\}}$. Indeed, it is easy to see that Theorem \ref{thm:improvedGA} and \eqref{eq:blockproduct} implies that, there exists a probability space $(\Phi_c, \mathcal{A}_c, \mathbb{P}_c)$ on which we can define random variables $X_i^c$ such that the $S_j^c=\sum_{i=1}^j X_i^c=_{\mathbb{D}}S_j$ and a two-sided Brownian motion $\mathbb{W}(\cdot)$ such that,
\begin{align}
  \max_{1 \leq i \leq n} | S_i^c - \mathbb{W}(\IE[\mathcal{T}_i])|&=o_{\IP}(n^{\max\{{1/p},\min\{\frac{1+A}{2(1+2A)}, \frac{1}{4}+\frac{1-2/p}{2(1+2A)}\}\} } \log n), \label{eq:GA_block}\\
  \max_{1 \leq i \leq n} | S_i^c - \mathbb{W}(\IE[\mathcal{T}_i^{-}])|&=o_{\IP}(n^{\max\{1/p, \min\{\frac{A}{1+2A}, \frac{1}{2}-\frac{1-2/p}{1+2A}\}\}} \log n), \label{eq:GA_noblock}\\
  \max_{1 \leq i \leq n} | S_i^c - \mathbb{W}(\IE[\mathcal{T}_i^{\diamond}])|&=o_{\IP}(n^{1/2} ) .\label{eq:GA_mies1}
\end{align}
Thus, to summarize, equation \eqref{eq:BMbootstrap}, \eqref{eq:GA_block}, \eqref{eq:GA_noblock}, \eqref{eq:GA_mies1} and \eqref{eq:improvedGA} together implies a roughly $n^{1/4}$ for the Gaussian approximation of $\max_i |S_i - \mathbb{W}(\mathcal{T}_{i})|$ and $n^{1/2}$ rate (ignoring the logarithmic terms) for $\max_i |S_i - \mathbb{W}(\mathcal{T}_{i}^{-})|$ and $\max_i |S_i - \mathbb{W}(\mathcal{T}_{i}^{\diamond})|$ respectively, when conditions of Theorem \ref{thm:improvedGA} is satisfied. 
and therefore
\begin{equation}\label{bm:result_nocp}
  \max_{1 \leq i \leq n} |\mathbb{W}(\mathcal{T}_i^{-}) - \mathbb{B}(\IE(S_i^2)) |= O_{\IP^{*}}(n^{(1-\zeta_2)/2}\sqrt{\log n}),
\end{equation}
So the bound (\ref{bm:result_nocp}) is worse.} 
\ignore{
Let us now investigate when can the rate of Theorem \ref{thm:GA}, i.e $1/p$ be worse than the rate in \eqref{result:GA_nocp}(ignoring the logarithmic factors). Elementary but tedious calculations show that for $p>2$, $A > A_0 \vee 1$,
\[ \max\{ 1/2 - L/2, 1/p\}=\begin{cases}
  1/p, &  \ p \in (2, \frac{7 + \sqrt{33}}{4}\approx 3.186) \ \text{ and} \\ & \hspace*{2cm} 1<A < A_2(p):=\frac{1}{2}\sqrt{\frac{p^3 + 2p^2 - 15p+16}{p(p-2)^2}} - \frac{p-3}{2(p-2)},\\
   1/2 - L'/2, &  \ \text{otherwise}.
\end{cases} \]
\begin{figure}
  \centering
  \includegraphics[scale=0.5]{New Figures/comparison.pdf}
  \caption{Plot of the two Gaussian approximation rates for $p=3$: with cross product as in Proposition \ref{prop:GAblock} and without cross product as in Proposition \ref{prop2.1}. The two rates intersect at $A_2(3)$.}
  \label{fig:comparison}
\end{figure}
This can be interpreted as follows: 

For the case where $p$ is small, i.e $\in (2, \frac{7 + \sqrt{33}}{4})$ and $A$ is not too big, the cross-product terms do not convey much additional information. The inclusion of only the square terms $\IE(B_j^2)$ in the variance of Gaussian approximation is enough to produce the optimal rate in these cases, explaining most of the information about the moments of the partial sums. However even for such small $p$, if we assume fast enough decay, i.e $A > \frac{1}{2}\sqrt{\frac{p^3 + 2p^2 - 15p+16}{p(p-2)^2}} - \frac{p-3}{2(p-2)}$, the inclusion of the cross-product term decreases the rate of Gaussian approximation to $1/p$. This can also be seen in Figure \ref{fig:comparison}.

On the other hand, for all $A$, as $p \geq \frac{7 + \sqrt{33}}{4}$, cross-product terms are no longer negligible. In this case, the rate of Theorem \ref{thm:improvedGA} (i.e the rate including the cross-product terms $\IE(B_{j}B_{j+1})$) decreases with $p$. As $p$ becomes even moderately large, $L$ behaves almost like a constant as a function of $p$ for fixed $A$. Hence $1/2 - L/2 >> 1/p$ for all reasonably large $p$ and all $A>A_0 \vee 1$. Therefore clearly the rate \eqref{result:GA_nocp} without the cross-product term is much worse for all $p$ except some small $p \in (2,4)$. 
}

Following the idea of the moving or overlapping block bootstrap method (cf. \cite{kunsch} and \cite{Liu1992MovingBJ}, \citet{zhouzhoujasa} and \citet{mies1}), consider the following estimate of $\IE(S_i^2)$ by
  \begin{equation}\label{Ti_Mies}
    \mathcal{T}^{\diamond}_i=\sum_{t=m}^i \frac{1}{m}\left(\sum_{s=t-m+1}^t X_s\right)^2.
  \end{equation}
  \noindent
A treatment similar to Proposition \ref{prop:Bj2-EBj2} shows that $\mathcal{T}_i^{\diamond}$ satisfies (\ref{result:GA_nocp}) as well. 
Thus, $\mathcal{T}_i$ has the best rate for estimating the variance of the Brownian motion among the three estimators discussed here. It should be noted that \cite{mies1} analyzes a different variance for the approximating Gaussian process (defined as a local long-range variance $\sigma^2_{{loc}_i}$), and $\mathcal{T}^{\diamond}_i$ has been proposed in that context. However, we point out that for fast enough decay, their rate of Gaussian approximation $\max_{1 \leq i \leq n}|S_i^{c} - \IB(\sigma^2_{{loc}_i})|=o_{\IP}(n^{p/(3p-2)}\sqrt{\log n})$ is suboptimal in $n$.

\section{Applications of Gaussian approximations}
\label{sec:application}
In this section, we are interested in obtaining Gaussian approximations of functionals of the form $$W(t):= \sum_{i=1}^n e_i w_i(t),$$ where $w_i(\cdot):[0,1]\to \R$ are weight functions and $(e_i)_{1 \leq i \leq n}$ are real-valued, mean-zero, possibly non-stationary processes. Such quantities are ubiquitous in various statistics of change point estimation, wavelet transform, and forming a simultaneous confidence band, among others. One can employ (\ref{eq:improvedGA}) of Theorem \ref{thm:GA} to deal with such quantities. A similar treatment is included in \cite{zhibia}. Let 
\begin{equation}\label{1stomega}
  W^{\diamond}(t)=\sum_{i=1}^n w_i(t) \left(\mathbb{B}(\IE(S_i^2))- \mathbb{B}(\IE(S_{i-1}^2)) \right)
\end{equation}
be the Gaussian process that we want to use to approximate $W(t)$, where $S_i = \sum_{j=1}^i e_j$. Let 
\begin{equation}\label{eq:omega}
  \Omega_n=\sup_{t \in (0,1)} \{|w_1(t)| + \sum_{i=2}^n |w_i(t)-w_{i-1}(t)| \}
\end{equation}
be the maximum variation of the weights $w_i(t)$. Then, 
\begin{align}\label{generalbound}
  \sup_{t \in (0,1)}|W(t)- W^{\diamond}(t)|\leq \Omega_n \max_{1 \leq i \leq n}|S_i- \mathbb{B}(\IE(S_i^2))|= \Omega_n o_{\IP}(n^{1/p}\sqrt{\log n}).
\end{align}
In the following, we detail three applications - testing for change-point, simultaneous confidence band building, and wavelet transform - using the above analysis. Each of these analysis requires providing a rate of $\Omega_n$ depending on certain conditions. 
\subsection{Change point detection}
\label{subsec:cpt}
Assume $X_i = \mu_i + Z_i$, $i = 1, \ldots,n$, where $(Z_i)$ is a mean zero non-stationary process.
We want to test for the existence of change point in means, that is we want to test for $H_0: \mu_i=\mu_0$ for all $i$ versus the alternative hypothesis 
\begin{equation}
  \label{eq:trend}
H_1: \mu_i = \mu_0+\delta \mathbb{I}\{i >\tau\} \text{ holds for some $1 < \tau<n$ and $\delta \neq 0$.}
\end{equation}
\noindent
We propose a CUSUM-based testing procedure with test statistic 
\begin{equation}
  \label{eq:test_stat}
U_n := \max_{t \in (0, 1) } |\sum_{i \le n t} (X_i - \bar{X})| / \sqrt{n},
\end{equation}
where we reject our null hypothesis if $U_n$ is larger than some suitable cut-off value. Under the null hypothesis, we can write $U_n = \max_{t \in (0, 1)} |U_{n, t}|$, where $U_{n,t} := \sum_{i=1}^n w_i(t) Z_i$ and the weights $w_i(t)=((1-1/n)\mathbb{I}\{i \leq n t\}- (1/n) \mathbb{I}\{i > n t\})/\sqrt{n} $. Let
\begin{equation*}
 V_n = \max_{t \in (0, 1)} V_{n,t}, \mbox{ where }
 V_{n,t} := \sum_{i=1}^n w_i(t) \left(\mathbb{B}(\IE(S_i^2))- \mathbb{B}(\IE(S_{i-1}^2)) \right).
\end{equation*}
By (\ref{generalbound}), we have $|U_n- V_n | = o_{\IP}(1)$ since $\Omega_n= (2-1/n) /\sqrt{n}$ and $\Omega_n n^{1/p}\sqrt{\log n} \to 0$. \qed

\ignore{
The following result establishes a asymptotic theory for this testing procedure under the null hypothesis, thereby showing its validity. 
\begin{thm}
\label{thm:size_test}
Let $\{S_i\}_{i=1}^n$ be formed based on $Z_i=X_i - \mu_i$.
  Further assume the conditions of Theorem \ref{thm:improvedGA} hold for the noise sequence $(Z_t)_{t \geq 1}$. \color{black}Then, under $H_0$, we have
  \begin{equation}
  \label{eq:asym_U_n}
    \bigg |U_n- \max_{1 \leq i \leq n} n^{-1/2} | \IB(\IE(S_i^2)) - \frac{i}{n} \IB(\IE(S_n^2))| \bigg | =o_{\IP}(1),
  \end{equation}
	as $n \rightarrow \infty$.
\end{thm}
\begin{proof}
  Observe that under $H_0$, $U_n$ can be written as $\max_{1 \leq t \leq n} |\sum_{i=1}^n w(t, i) Z_i|/\sqrt{n}$, 
  where $w(t,i)=(1-1/n)\mathbb{I}\{i \leq t\}-1/n \mathbb{I}\{i>t\}$. Let 
  \begin{align}\label{eq:BM}
    U_{n,t} := \sum_{i=1}^n w(t, i) Z_i/\sqrt{n} \nonumber, \ \ V_{n,t} := \sum_{i=1}^n w(t, i) \left(\mathbb{B}(\IE(S_i^2))- \mathbb{B}(\IE(S_{i-1}^2)) \right)/\sqrt{n},
  \end{align}
  and $V_n = \max_{1 \leq t \leq n} V_{n,t}$. Note that, $\Omega_n= 2-1/n = O(1)$. Hence, 
  \begin{align}
    |U_n - V_n| \leq \max_{1 \leq t \leq n} |U_{n,t} -V_{n,t}|=o_{\IP}(n^{1/p - 1/2}\log n)=o_{\IP}(1),
  \end{align}
  which completes the proof. 
\end{proof}

 This result shows that one can approximate the quantiles of the CUSUM statistic $U_n$ by the quantiles of suitably centered Brownian motion. This result enables us to define an approximately valid level-$\alpha$ test $\psi_{n}$, which we identify as an oracle test, as follows:
  \begin{equation} \label{test_cp}
    \psi_{n}:= \mathbb{I}\{U_n > c_{\alpha})\},
  \end{equation}
  where
     $c_{\alpha}:= \inf_{r}\bigg\{\IP\left(\max_{i\leq n} n^{-1/2}\big| \mathbb{B}(\IE(S_i^2)) - \frac{i}{n} \mathbb{B}(\IE(S_n^2))\big| >r \right)\leq \alpha\bigg\}.$
\color{black} From a practitioner's point of view, we can employ arguments from Section \ref{ssc:estvar} to estimate $c_{\alpha}$ by its bootstrap estimate $c_{\alpha}^T:=\inf_{r}\bigg\{\IP^{*}\left(\max_{i\leq n} n^{-1/2}\big| \mathbb{W}(\mathcal{T}_i) - \frac{i}{n} \mathbb{W}(\mathcal{T}_n)\big| >r \right)\leq \alpha\bigg\},$ where $\IP^{*}$ denotes the conditional probability given $X_1, \ldots, X_n$, and $\mathcal{T}_i$ is formed as in \eqref{eq:T_k} but based on $X_i - \hat{\mu}_i$, where $\hat{\mu}_i=\frac{1}{\tau}\sum_{j=1}^{\tau} X_j \mathbb{I}\{i \leq \tau\} + \frac{1}{n-\tau}\sum_{j=\tau+1}^{n} X_j \mathbb{I}\{i > \tau\}$, with $\tau = {\rm argmax}_t |\sum_{i = 1}^{t} (X_t - \bar{X})| / \sqrt{n} $.
 For future references we denote this bootstrap-based test by 
 \begin{equation} \label{test_cp_new}
    \psi_{n1}:= \mathbb{I}\{U_n > c_{\alpha}^T)\},
  \end{equation}
\color{black}
}

\subsection{Simultaneous confidence band}
\label{subsec:scb}
In this section, we discuss construction of simultaneous confidence band for a time-varying signal-plus-noise model with possibly irregularly spaced observed data and possibly non-stationary noise. Let $0=t_0 < t_1< t_2<\ldots<t_{n-1}<t_n<t_{n+1}=1$ be an $n$-length grid on $[0,1]$. Consider
\begin{equation}
  \label{eq:scb_model}
  X_i = \mu(t_i) + Z_i, \quad \quad i = 1,\ldots,n,
\end{equation}
where $\mu(\cdot) \in \mathcal{C}^3[0,1]$. The case $t_i=i/n$ has been thoroughly analyzed in the literature for stationary and i.i.d.set-up, such as \cite{10.2307/2291269}, \cite{10.1214/aos/1030563978} and \cite{zhibia}. Here we let $t_i=F^{-1}(i/n)$, where $F(t)=\int_{0}^t f(u) \text{d}u$ for some density $f \in C^3[0,1]$. We will estimate the trend function from observed data $(X_i)$ using the local linear estimate, and denote the result by $\hat\mu_{h_n} (\cdot)$, where $h_n$ is the bandwidth parameter. Define 
\begin{equation}\label{eq:Sj}
  S_j(t)=\sum_{i=1}^n(t-t_i)^j K( (t-t_i)/h_n).
\end{equation}
Theorem \ref{thm:scb} below provides a Gaussian approximation for the local linear estimate
\begin{align}\label{eq:hat_mu}
\hat\mu_{h_n} (t) := \sum_{i=1}^n w_{h_n}(t, i) X_i, \text{ where }w_{h_n}(t, i)=K\left(\frac{t-t_i}{h_n}\right) \frac{S_2(t)-(t-t_i) S_1(t)}{S_2(t) S_0(t)-S_1^2(t)}.\color{black}
\end{align}
Assume that $K$ is a smooth symmetric kernel with bounded support $[-\omega, \omega]$, satisfying:
\begin{equation} \label{condn:regularity}
  \int_{\mathbb{R}} \Psi_K(u ; \delta) \mathrm{d} u=O(\delta) \text { as } \delta \rightarrow 0, \mbox{ where } \Psi_K(u ; \delta)=\sup \left\{\left|K(y)-K(u)\right|: |y-u| \le \delta \right\}.
\end{equation}
\begin{thm}
\label{thm:scb}
Assume $\mu,f \in \mathcal{C}^3[0,1]$ and, for some constants $C_1, C_2 > 0$, $C_1 \leq f(t) \leq C_2$ for all $t\in[0,1]$. Then under the assumptions of Theorem \ref{thm:GA} for $Z_i$, there exists Brownian motion $\IB(\cdot)$ such that with $Q_{h_n}(t) = \sum_{i=1}^n w_{h_n}(t, i) \mathbf{Y}_i $, where $\mathbf{Y}_i = \mathbb{B}(\IE(S_i^2)) - \mathbb{B}(\IE(S_{i-1}^2))$, the following is true: 
\begin{equation}
  \label{eq:result_scb}
  \sup_{t \in [\omega h_n,1-\omega h_n]} \left|\hat\mu_{h_n} (t) - \mu(t) - h_n^2 \beta \mu{''}(t) - Q_{h_n} (t)\right| = o_{\IP} (h_n^{-1} n^{1/p-1} \sqrt{\log n}),
\end{equation} for any $h_n \to 0$ satisfying ${h_n^4} = O({n^{1/p-1}})$ and ${nh_n} \to \infty$ with $\beta=\int u^2 K(u) \text{d}u /2$.
\end{thm}
\begin{proof}
We apply Theorem \ref{thm:GA} to $(Z_i)_{i=1}^n$. Note that $Q_{h_n} (t)$ is obtained by fitting the same local linear regression with bandwidth $h_n$ to $(\mathbf{Y}_i)_{i=1}^n$. By the argument in Theorem 3.1 in \cite{fan1996local}, $\IE[\hat\mu_{h_n} (t)]- \mu(t) = h_n^2 \mu{''}(t) \beta+ O( h_n^3 + n^{-1}h_n^{-1})$. Then (\ref{eq:result_scb}) follows by applying \eqref{generalbound} to $\hat\mu_{h_n} (t) - \IE[\hat\mu_{h_n} (t)] - Q_{h_n} (t)$ and noting that $\Omega_n=O(1/{(nh_n)})$ using Lemma \ref{lemma3zhibiao} and $C_1 \leq f(\cdot) \leq C_2$.
\end{proof}
{\ignore{  
  We apply our Gaussian approximation to $(Z_i)_{i=1}^n$ and obtain $\mathbf{Y}_i := \mathbb{B}(\IE(S_i^2)) - \mathbb{B}(\IE(S_{i-1}^2))$. To $(\mathbf{Y}_i)_{i=1}^n$, we can fit the same model with bandwidth $h_n$ to obtain $Q_{h_n} (t) := \sum_{i=1}^n w_{h_n}(t, i) \mathbf{Y}_i.$
 Using our Theorem \ref{thm:GA} along with \eqref{generalbound}, one can obtain the following result as the justification for our bootstrap procedure. 
\begin{thm}
\label{thm:scb}
Assume $\mu \in \mathcal{C}^3[0,1]$ and, for some constants $C_1, C_2 > 0$, $C_1 \leq f(t) \leq C_2$ for all $t\in[0,1]$. Then under assumptions of Theorem \ref{thm:GA2}, there exists Brownian motion $\IB(\cdot)$ such that with $Q_{h_n}(t)$ as above, it holds that
\begin{equation}
  \label{eq:result_scb}
  \sup_{t \in [\omega h_n,1-\omega h_n]} \left|\hat\mu_{h_n} (t) - \mu(t) - h_n^2 \mu{''}(t) \beta - Q_{h_n} (t)\right| = \operatorname{o}_{\IP} (h_n^{-1} n^{1/p-1} \log n)
\end{equation}
for any $h_n \to 0$ satisfying ${h_n^4} = O({n^{1/p-1}})$ and ${nh_n} \to \infty$. Here $\beta=\frac{1}{2}\int u^2 K(u)\ du$.
\end{thm}
\begin{proof}
Following the argument in Theorem 3.1 in \cite{fan1996local} and the estimates of $S_j(t)$ from Lemma \ref{lemma3zhibiao}, the bias part $\IE[\hat\mu_{h_n} (t)]- \mu(t) = h_n^2 \mu{''}(t) \beta+ \operatorname{O}_{\IP}\left( h_n^3 + {h_n}/{n} \right)$. Then (\ref{eq:result_scb}) follows by applying \eqref{generalbound} to $\hat\mu_{h_n} (t) - \IE[\hat\mu_{h_n} (t)] - Q_{h_n} (t)$ and noting that $\Omega_n=O(1/{(nh_n)})$ using Lemma \ref{lemma3zhibiao} and $C_1 \leq f(t) \leq C_2$ for all $t$.
\end{proof}
\ignore{
\begin{proof}
  Note that
\begin{equation}
  \label{eq:err_nu}
  \left|\hat\mu_{h_n} (t) - \mu(t) - h_n^2 \mu^{''}(t) \beta - {Q}_{h_n} (t)\right| \leq \underbrace{\left|\hat\mu_{h_n} (t) - \IE[\hat\mu_{h_n} (t)] - Q_{h_n} (t) \right |}_{\text{error}} + \underbrace{\left|\IE[\hat\mu_{h_n} (t)]- \mu(t) - h_n^2 \mu^{''}(t) \beta \right |}_{\text{bias}}
\end{equation}

For the bias part of \eqref{eq:err_nu}, using Lemma \ref{lemma3zhibiao} and $\frac{h_n^4}{n^{1/p}\Delta_n} \to 0$,
\begin{equation}
  \label{eq:mu-nu}
  \left|\IE[\hat\mu_{h_n} (t)]- \mu(t) - h_n^2 \mu^{''}(t) \beta \right | = \operatorname{O} \left( h_n^3 + \frac{\Delta_n}{h_n} \right) = o\left(\frac{n^{1/p}\Delta_n}{h_n}\right).
\end{equation}
For the error part, invoking \eqref{generalbound} and noting that $\Omega_n=O(\frac{\Delta_n}{h_n})$, completes the proof. 
\end{proof}
}}
\subsubsection{Bias correction}
 Using \eqref{eq:result_scb} to construct simultaneous confidence band requires estimation of $\mu''(t)$. Following \cite{hardle}, we use the jackknife-based bias corrected estimator 
  \begin{equation}\label{debias}
    \Tilde{\mu}_{h_n}(t)=2\hat{\mu}_{h_n}(t) - \hat{\mu}_{h_n \sqrt{2}}(t).
  \end{equation}
  Using \eqref{debias} is asymptotically equivalent to using the kernel
  $ K^{*}(x)=2 K(x) - {K(x/\sqrt{2})}/{\sqrt{2 }}$;
  see \cite{zhouwu10}, \cite{zhibia} and \cite{karmakar2022simultaneous} among others.
  Based on \eqref{debias} one can observe $\IE[ \Tilde{\mu}_{h_n}(t)] - \mu(t)= O(h_n^3 + n^{-1}h_n^{-1}\color{black})$. Thus one can get rid of the $h_n^2\mu''(t)$ term from the left-hand side of the \eqref{debias} to obtain
  \begin{equation}\label{debias_rate}
    \sup_{t \in [\omega h_n,1-\omega h_n]}|\Tilde{\mu}_{h_n}(t_i)- \mu(t_i)- \Tilde{Q}_{h_n}(t_i) |=o_{\IP}(h_n^{-1} n^{1/p-1} \sqrt{\log n}).
  \end{equation}

\subsubsection{Choice of bandwidth $h_n$}
Since our Gaussian approximation Theorem \ref{thm:GA} holds with $n^{1/4}$ rate for $p\geq4$, $A>A_0$, for this subsection, assume $p=4$. Ignoring the log factors, we obtain a rate of $O_{\IP}(n^{-3/4}/h_n)$ from \eqref{eq:result_scb}, which readily allows a large range of $h_n$:
  \begin{equation}\label{condition_p4}
    n^{-3/4} \leq h_n \leq n^{-3/16}.
  \end{equation}
  In particular, \eqref{condition_p4} allows for $h_n \asymp n^{-1/5}$, which is the mean-square error optimal bandwidth. 
As equation \eqref{debias_rate} suggests, $\Tilde{Q}_{h_n} $ is a good simultaneous approximation for $\Tilde{\mu}_{h_n}-\mu $ in distribution. Therefore, for our bootstrap algorithm, $\Tilde{Q}_{h_n} $ is generated based on $(\mathbf{Y}_i)$, which is simulated from our Gaussian approximation where we estimate $\IE(S_i^2)$ by $\mathcal{T}_i$ 's formed by $Z_i$ as in \eqref{eq:T_k}. Using this, for $0<\alpha<1$, we can calculate $q_{1-\alpha}$, the empirical $(1-\alpha)$-th quantile of $\max_{1\leq i \leq n} | \Tilde{Q}_{h_n}(i / n) |$. Thus, given significance level $\alpha$, the simultaneous confidence level for $\mu(\cdot)$ can be constructed as 
\begin{equation}
  \label{eq:cons_scb}
    [\Tilde{\mu}_{h_n} (t) - q_{1-\alpha}, \Tilde{\mu}_{h_n} (t) +q_{1-\alpha}], \quad t \in [0,1].
\end{equation}
\subsection{Wavelet coefficient process}\label{sec:wavelet}
Wavelet transform is a way of representing a time series locally both in time and frequency windows. Mathematically speaking, wavelength coefficients are simply the coefficients when the signal $(X_i)_{1 \leq i \leq n}$ is decomposed in terms of some orthonormal basis of $L^2(\R)$. The simplest discrete wavelet transform used is called the Haar Transform \cite{Haar1910ZurTD}. Assume the signal length is $n=2^k$. Then the $j$-th level Haar Wavelet coefficients with $j \le k$ are
\begin{equation} \label{haarcoeff}
  W_{j,t}=\sum_{l=1}^{2^j} h_{j,l} X_{2^j t -l+1} \ , \ t=1,\ldots, 2^{k-j}, \ \text{where} \ h_{j,l}= \begin{cases}
    -2^{-j/2} & \text{ if } 1 \leq l \leq 2^{j-1},\\
    2^{-j/2} & \text{ if } 2^{j-1} < l \leq 2^{j}.
    \end{cases}
\end{equation}
\citet{donoho} used wavelet methods to perform non-parametric signal estimation via soft thresholding; however their threshold value crucially depends on the assumptions of the noise process being i.i.d. Gaussian. \citet{johnstone} and \citet{vonsachs} extended the results for correlated Gaussian and locally stationary noise processes respectively. Recently, \cite{McGonigle} considered locally stationary wavelet processes as the noise processes for estimation of signal. Stationarity assumption also features crucially in the wavelet variance estimation mechanism of \citet{PERCIVAL2012623}. Here we allow the signal $(X_i)_{1 \leq i \leq n}$ to be possibly non-stationary, and focus on applying our Theorem \ref{thm:GA} to provide a Gaussian approximation result for the wavelet coefficient process $W_{j,t}$. 
Note that $W_{j,t}$ can be written as $\sum_{i=1}^n w_i(j,t)X_i$, where $w_i(j,t)=h_{j, 2^j t - i +1}$. Let 
\[ W^{\diamond}_{j,t}=\sum_{i=1}^n w_i(j,t)(\mathbb{B}(\IE(S_i^2)) - \mathbb{B}(\IE(S_{i-1}^2))). \]
With $\Omega_n$ as defined as in \eqref{eq:omega}, it can be easily seen that $\Omega_n =O(2^{-j/2})$. Thus, using \eqref{generalbound}, we get,
\begin{equation}
\label{eq:unifwave}
\max_{ j_* \le j \le k} \max_{1 \leq t \leq n/2^j} |W_{j,t} - W^{\diamond}_{j,t}|= o_{\IP}(2^{-j_*/2}n^{1/p} \sqrt{\log n}). 
\end{equation}
To ensure a uniform Gaussian approximation, we require the highest resolution level $j_*$ to satisfy:
\begin{equation}\label{waveletsuff}
j_* - \frac{2}{\log 2}\left(\frac{1}{p} \log n + \frac{1}{2}\log \log n \right) \rightarrow \infty. 
\end{equation}
In particular, it holds if $j_* \ge c \log n$ for some constant $c > 2/(p \log 2)$. Similar analysis can be performed for the more general Daubechies wavelet filters (\citet{debauches}), with better smoothness properties. The uniform Gaussian approximation (\ref{eq:unifwave}) allows an asymptotic distributional theory for statistics based on wavelet transforms of non-stationary processes.

\section{Simulation}\label{sec:simu}
This section presents a simulation study for some of our results in Sections \ref{sec:main_thm}, \ref{ssc:estvar} and \ref{sec:application} while some more are postponed to the Appendix Section \ref{appendix:simu}. Our aims are as follows. In Section \ref{subsec:theosimu}, we start off by investigating the accuracy of the two kinds of theoretical Gaussian approximations in Sections \ref{subsec:thm:new_GA} and \ref{subsec:thm:GA}. \color{black} We postpone inspecting the accuracy of our bootstrap Gaussian approximations for finite sample to appendix Section \ref{subsec:simu1},  In particular, in Section \ref{sec:nocp}, having argued that excluding the cross-product terms results in a worse rate and a less accurate approximation compared to \eqref{eq:BMbootstrap}, we compare their finite sample accuracy for some simple cases. Moving on to showing simulation-based evidences for our applications, in Section \ref{sim:scb}, we explore the empirical coverage of our simultaneous confidence band procedure discussed in Section \ref{subsec:scb} under different settings. We again defer analysing the performance of the CUSUM-based testing procedure for existence of change-point, as discussed in Section \ref{subsec:cpt} to Appendix Section \ref{subsec:simu2}.
\color{black}
\vspace{-0.1 in}
\subsection{Empirical accuracy of theoretical Gaussian approximations}\label{subsec:theosimu}\color{black}
Consider two models: 
\begin{enumerate}
 \item\label{ARtheosim} Model 5.1: $X_t = \theta X_{t-1} + \varepsilon_t, \ \theta \in \{0.9, -0.9\}.$ \\
 \vspace{-0.1 in}
\item\label{ARtheononsim}Model 5.2: $X_t = \theta_t X_{t-1} + \varepsilon_t$, $\theta_t = \theta$ if $t \le n/2$, $\theta_t = - \theta$ if $t > n/2$, $\theta \in \{0.9, -0.9\}$.
\end{enumerate}
\vspace{-0.1 in}
We will start off by letting $\varepsilon_t \overset{\text{i.i.d.}}{\sim} t_4/\sqrt{2}$ for both the Models. Observe that, with $N(0,1)$ innovations, $(X_t)_{t=1}^n$ is already a Gaussian process for both Models \ref{ARtheosim} and \ref{ARtheononsim}, and therefore the approximation error is trivially zero. This motivates the use of some other mean-zero error for this model. 
\noindent We will initially consider a small sample of size $n=100$. For each of the set-up, we will compare the quantiles of the following three random variables:
\begin{equation*}
    U_X:= \max_{1 \leq i \leq n} S_i, \ U_1 = \max_{1 \leq i \leq n} \IB(\IE(S_i^2)), \ U_2 = \max_{1 \leq i \leq n} \sum_{j=1}^i Y_i,
\end{equation*}
where $(Y_t)_{t=1}^n$ is a centered Gaussian process with same covariance structure as $(X_t)_{t=1}^n$. The true quantiles are estimated by sample quantiles based on $10^3$ repetitions. Figures \ref{Fig:ARtheosim} and \ref{Fig:ARtheononsim} depicts the QQ-plots of $U_1$ and $U_2$ against $U_X$. Clearly, when compared with $U_1$ which involves Brownian motion, our Gaussian approximation of Section \ref{subsec:thm:new_GA} maintaining covariance structure, performs much better for such a small sample size $n=100$. 
\begin{figure}[!htbp]
\centering
\includegraphics[height=6cm, width=13cm]{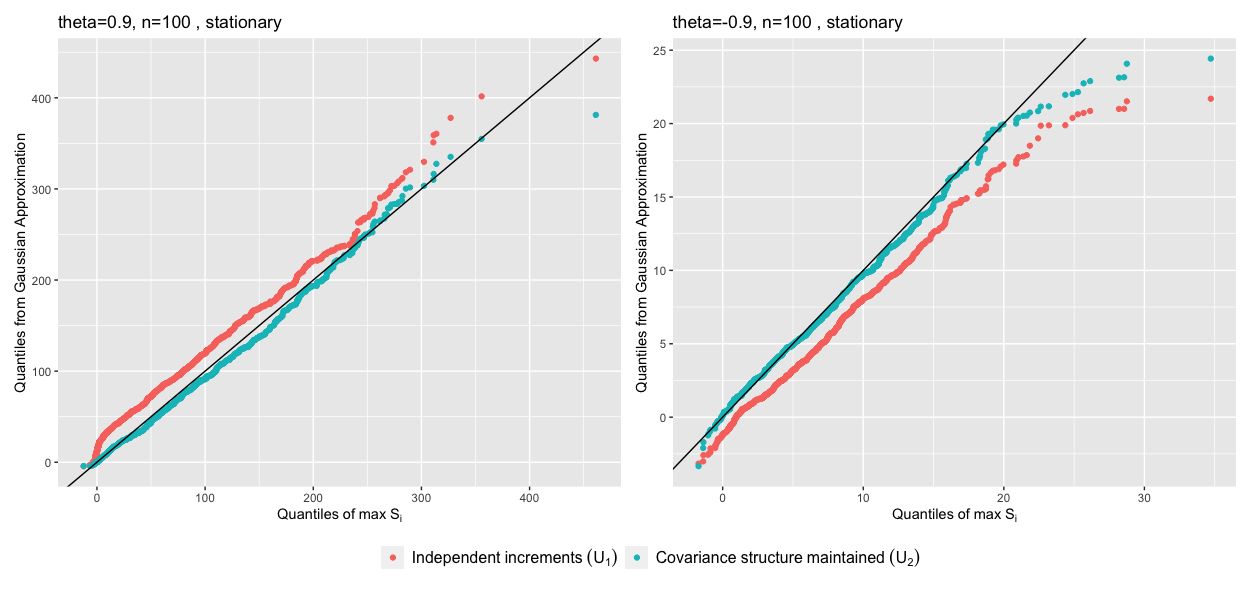} 
\caption{Comparison of theoretical quantiles with the two kinds of Gaussian approximation $X_1, \ldots, X_n \sim$ Model \ref{ARtheosim} with $t_4$ innovations: with independent increments, and with the approximation maintaining covariance structure.}
\label{Fig:ARtheosim}  
\end{figure}
\begin{figure}[!htbp]
\centering
\includegraphics[height=6cm, width=13cm]{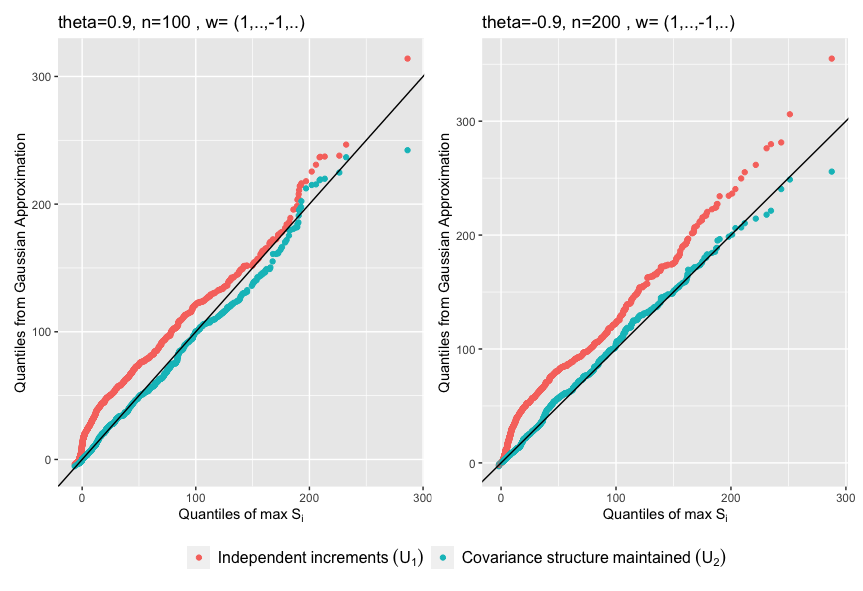} 
\caption{Comparison of theoretical quantiles with the two kinds of Gaussian approximation $X_1, \ldots, X_n \sim$ Model \ref{ARtheononsim} with $t_4$ innovations: with independent increments, and with the approximation maintaining covariance structure.}
\label{Fig:ARtheononsim}  
\end{figure}\color{black}
However, as we increase $n$, both the approximations being theoretically valid with optimal rate of convergence, their performances become comparable. To show this empirically, we consider two more complicated non-stationary models.
\vspace{0.05 in}
\saveenum
\begin{enumerate}\resetenum
\item \label{ARnonsim} Let $w_{1} = \underbrace{0.75, \ldots}_{n/4}, \underbrace{-0.75, \ldots}_{n/4}, \underbrace{0.75, \ldots}_{n/4}, \underbrace{-0.75,\ldots}_{n/4} $, $w_{2} = (\sin(8\pi t/n))_{t=1}^n$, and
\begin{equation*}
  X_t=\theta_t X_{t-1} + \varepsilon_t, \ \theta_t=\theta w_{it},  \  X_0=0 \ , i \in \{1,2 \}, \ \theta \in \{-0.8, 0.8\}.
\end{equation*}

 \item \label{ARnonsimsq} $X_t= \sin(Y_t), \ \text{where $Y_t \sim$ Model \ref{ARnonsim}}$. 
\end{enumerate}
To further show the efficacy of our approximation, we consider a skewed error for Model \ref{ARnonsim} with i.i.d. $\chi^2_1-1$ errors. We consider i.i.d. $N(0,1)$ innovations for Model \ref{ARnonsimsq}. Note that due to the $\sin$ transformation, Model \ref{ARnonsimsq} is no longer Gaussian. The corresponding QQ-plots are shown in Figures \ref{Fig:chisq_model1_theo} and \ref{Fig:ar_nonsim_sine}. It can be seen that both Gaussian approximations show excellent accuracy for a somewhat increased sample size $n=200$. In fact, in some of the set-ups, the more natural Gaussian approximation retains an advantage over the Gaussian approximation involving the Brownian motion. 
\begin{figure}[!htbp]
\centering
\includegraphics[height=6cm, width=13cm]{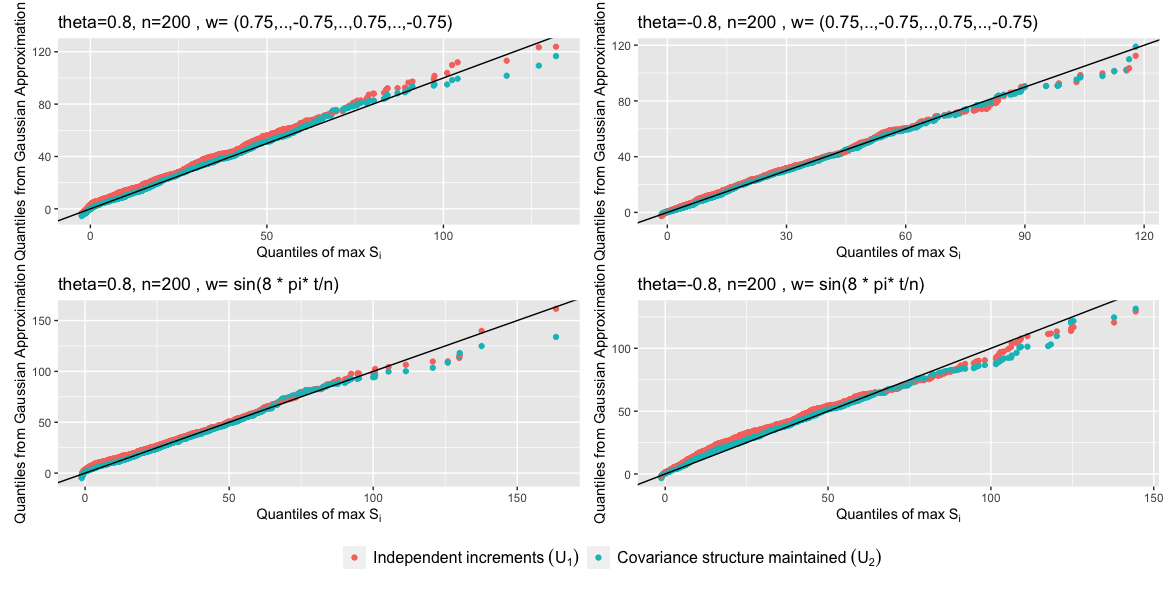} 
\caption{Comparison of theoretical quantiles with the two kinds of Gaussian approximation $X_1, \ldots, X_n \sim$ Model \ref{ARnonsim} with $\chi^2_1-1$ innovations: with independent increments, and with the approximation maintaining covariance structure.}
\label{Fig:chisq_model1_theo}  
\end{figure}
\begin{figure}[!htbp]
\centering
\includegraphics[height=6cm, width=13cm]{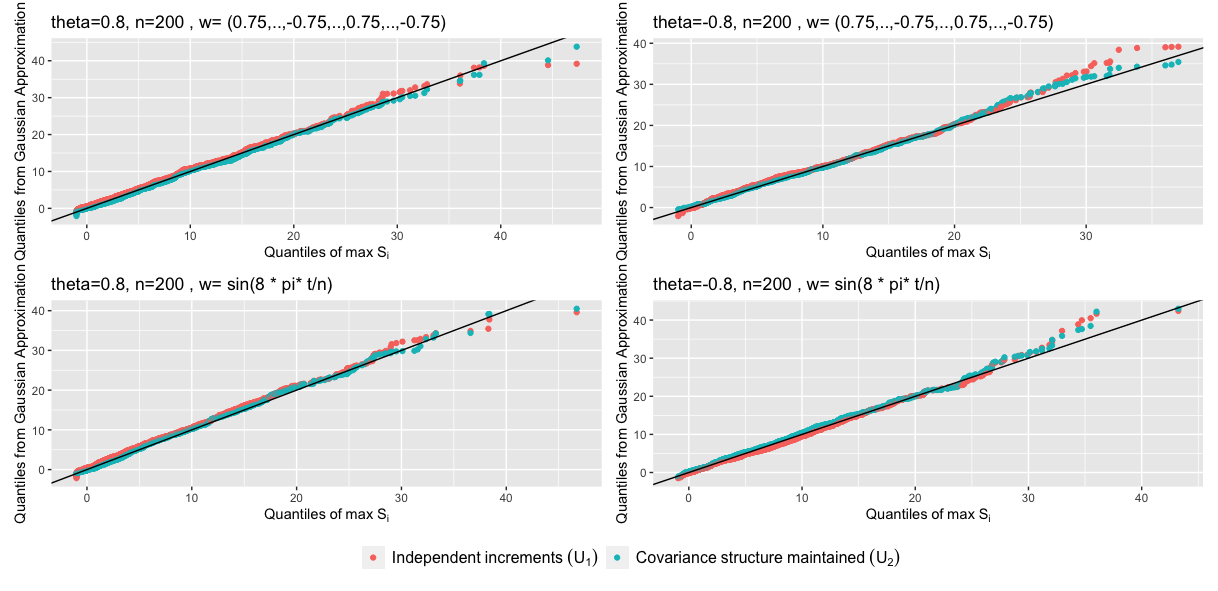} 
\caption{Comparison of theoretical quantiles with the two kinds of Gaussian approximation $X_1, \ldots, X_n \sim$ Model \ref{ARnonsimsq} with $N(0,1)$ innovations: with independent increments, and with the approximation maintaining covariance structure.}
\label{Fig:ar_nonsim_sine}  
\end{figure}
\vspace{-0.13 in}
\subsection{Simulation for simultaneous confidence bands} \label{sim:scb}
In this subsection, we will explore the empirical coverage probabilities for our $95\%$ SCBs constructed as in \eqref{eq:cons_scb}. 
We will use the Jackknife-based bias corrected version of the local linear estimate, as in \eqref{debias}. 
\color{black} We generate data from the model \eqref{eq:scb_model} with $\mu(t)=0.5\cos(2\pi t-0.7)+0.3\exp(-t),$ with $t_i=i/n$ for $i=1, \ldots,n$.
We consider the two models \eqref{ARnonsim} and \eqref{ARnonsimsq} with innovations $\varepsilon_t \sim t_6\sqrt{2/3}$ for our error generating process $Z_t$, and consider the two weighing schemes for each model with $\theta \in \{-0.8, -0.4, 0.4, 0.8\}$ in \eqref{ARnonsim}. We will estimate the mean curve using the Epanechnikov kernel $K(x)=\frac{3}{4}(1-x^2)\mathbb{I}\{|x|\leq 1\}$. For each of these model, we consider data of sizes $n=600$ and $800$, and bandwidths $h_n=0.11, 0.13$ and $0.15$. For each such setting, we perform $1000$ replications each with $500$ bootstrap samples of size $n$ each. Following our theoretical result in Theorem \ref{thm:scb} as well as the discussion at Section 3.2.5 of \cite{fan1996local}, the variance of local linear estimator is comparatively high on the boundary points, which affects coverage. Thus, we report as empirical coverage the percentage of times the estimated SCB contains the true $\mu(t)$ curve in the interval $[0.05, 0.95]$. Generally speaking, the coverage probabilities in Tables \ref{tab:scb_t_new} and \ref{tab:scb_sin_t_new} are reasonably close to the nominal level $0.95$. Moreover, the bandwidths do not seem to have too large an effect on the coverage probability. 

 \begin{table}[H]
  \centering
  \resizebox{\columnwidth}{0.1\textwidth}{%
\begin{tabular}{|l|c|c|c|c|c||c|c|c|c||} 
\hline & & \multicolumn{4}{c||}{Weights : $w=(0.75, \ldots, -0.75, \ldots ,0.75, \ldots, -0.75, \ldots)$} & \multicolumn{4}{c||}{Weights: $w=\sin(8\pi t/n)$} \\ \hline 
$n$ & $h_n$ & \hspace*{0.05cm} $\theta=-0.8$ \hspace*{0.05cm} & \hspace*{0.05cm} $\theta=-0.4$ \hspace*{0.05cm} & \hspace*{0.05cm} $\theta=0.4$ \hspace*{0.05cm} & $\theta=0.8$ & $\theta=-0.8$& $\theta=-0.4$ & $\theta=0.4$ & $\theta=0.8$ \\
\hline \hline $600$ & $0.11$ & $0.922$ & $0.949$ & $0.929$ & $0.913$ & $0.930$ & $0.951$ & $0.959$ & $0.916$\\
\hline & $0.13$ & $0.946$ & $0.952$ & $0.951$ & $0.938$ & $0.951$ & $0.956$ & $0.963$ & $0.950$ \\
\hline & $0.15$ & $0.950$& $0.963$ & $0.951$ & $0.950$ & $0.956$ & $0.964$ & $0.964$ & $0.959$ \\
\hline \hline $800$ & $0.11$& $0.948$ & $0.963$ & $0.954$ & $0.932$& $0.952$ & $0.962$ & $0.951$ & $0.952$ \\
\hline & $0.13$ & $0.954$ & $0.963$ & $0.960$ & $0.956$ & $0.958$ & $0.966$ & $0.958$ & $0.962$ \\
\hline & $0.15$ & $0.955$ & $0.965$ & $0.965$ & $0.953$ & $0.959$ & $0.966$ & $0.971$ & $0.970$\\
\hline
\end{tabular}
}
\caption{Empirical coverage probabilities of SCB of $X_t$ from Model \eqref{eq:scb_model} where $Z_t \sim$ Model \ref{ARnonsim} with normalized $t_6$ error.}
  \label{tab:scb_t_new}
\end{table}

\begin{table}[H]
  \centering
  \resizebox{\columnwidth}{0.12\textwidth}{%
\begin{tabular}{|l|l|c|c|c|c||c|c|c|c||} 
\hline & & \multicolumn{4}{c||}{Weights : $w=(0.75, \ldots, -0.75, \ldots ,0.75, \ldots, -0.75, \ldots)$} & \multicolumn{4}{c||}{Weights: $w=\sin(8\pi t/n)$} \\ \hline 
$n$ & $h_n$ & \hspace*{0.05cm} $\theta=-0.8$ \hspace*{0.05cm} & \hspace*{0.05cm} $\theta=-0.4$ \hspace*{0.05cm} & \hspace*{0.05cm} $\theta=0.4$ \hspace*{0.05cm} & $\theta=0.8$ & $\theta=-0.8$& $\theta=-0.4$ & $\theta=0.4$ & $\theta=0.8$ \\
\hline \hline $600$ & $0.11$& $0.940$ & $0.951$ & $0.943$ & $0.946$ & $0.941$ & $0.954$ & $0.958$ & $0.938$ \\
\hline & $0.13$ & $0.957$ & $0.951$ & $0.947$ & $0.951$ & $0.953$ & $0.951$ & $0.962$ & $0.950$ \\
\hline & $0.15$ & $0.950$ & $0.962$ & $0.954$ & $0.942$ & $0.959$ & $0.959$ & $0.958$ & $0.957$ \\
\hline \hline $800$ & $0.11$& $0.943$ & $0.967$ & $0.956$ & $0.941$ & $0.953$ & $0.959$ & $0.971$ & $0.938$ \\
\hline & $0.13$ & $0.953$ & $0.961$ & $0.967$ & $0.953$ & $0.956$ & $0.958$ & $0.961$ & $0.952$ \\
\hline & $0.15$ & $0.946$ & $0.965$ & $0.968$ & $0.949$ & $0.966$ & $0.958$ & $0.959$ & $0.963$ \\
\hline
\end{tabular}
}
\caption{Empirical coverage probabilities of SCB of $X_t$ from Model \eqref{eq:scb_model} where $Z_t \sim$ Model \ref{ARnonsimsq} with $t_6$ error.}
  \label{tab:scb_sin_t_new}
\end{table}
\ignore{
 \begin{table}[H]
  \centering
  \resizebox{\columnwidth}{0.165\textwidth}{%
\begin{tabular}{|l|c|c|c|c|c||c|c|c|c||} 
\hline & & \multicolumn{4}{c||}{Weights : $w=(0.75, \ldots, -0.75, \ldots ,0.75, \ldots, -0.75, \ldots)$} & \multicolumn{4}{c||}{Weights: $w=\sin(8\pi t/n)$} \\ \hline 
$n$ & $h_n$ & \hspace*{0.05cm} $\theta=-0.8$ \hspace*{0.05cm} & \hspace*{0.05cm} $\theta=-0.4$ \hspace*{0.05cm} & \hspace*{0.05cm} $\theta=0.4$ \hspace*{0.05cm} & $\theta=0.8$ & $\theta=-0.8$& $\theta=-0.4$ & $\theta=0.4$ & $\theta=0.8$ \\
\hline \hline $600$ &$0.10$ & $0.924$ & $0.944$ & $0.920$ & $0.904$ & $0.926$ & $0.946$ & $0.954$ &$0.908$\\
\hline & $0.11$& $0.948$ & $0.964$ & $0.944$ & $0.906$ & $0.940$ & $0.938$ & $0.954$ & $0.938$ \\
\hline & $0.12$& $0.952$& $0.968$ & $0.944$ & $0.916$ & $0.960$ & $0.970$ & $0.962$ & $0.948$ \\
\hline & $0.13$ & $0.956$ & $0.966$ & $0.950$ & $0.924$ & $0.962$ & $0.966$ & $0.960$ & $0.950$ \\
\hline & $0.14$ & $0.954$ & $0.972$ & $0.936$ & $0.926$ & $0.966$ & $0.968$ & $0.966$ & $0.952$ \\
\hline & $0.15$ & $0.956$& $0.968$ & $0.940$ & $0.934$ & $0.958$ & $0.966$ & $0.972$ & $0.962$ \\
\hline \hline $800$ &$0.10$ & $0.922$ & $0.960$ & $0.944$ & $0.914$ & $0.948$ & $0.958$ & $0.946$ & $0.934$ \\
\hline & $0.11$& $0.944$ & $0.966$ & $0.948$ & $0.932$& $0.962$   & $0.964$ & $0.956$ & $0.946$ \\
\hline & $0.12$& $0.944$ & $0.964$ & $0.952$ & $0.934$ & $0.970$ & $0.972$ & $0.966$ & $0.954$ \\
\hline & $0.13$ & $0.948$ & $0.972$ & $0.962$ & $0.938$ & $0.962$ & $0.966$ & $0.958$ & $0.960$ \\
\hline & $0.14$ & $0.950$ & $0.966$ & $0.962$ & $0.940$ & $0.964$ & $0.966$ & $0.968$ & $0.968$ \\
\hline & $0.15$ & $0.956$ & $0.968$ & $0.954$ & $0.936$ & $0.956$ & $0.968$ & $0.966$ & $0.970$\\
\hline
\end{tabular}
}
\caption{Empirical coverage probabilities of SCB of $X_t$ from Model \eqref{eq:scb_model} where $Z_t \sim$ Model \ref{ARnonsim} with $N(0,1)$ error.}
  \label{tab:scb_t}
\end{table}
\begin{table}[H]
  \centering
  \resizebox{\columnwidth}{0.165\textwidth}{%
\begin{tabular}{|l|l|c|c|c|c||c|c|c|c||} 
\hline & & \multicolumn{4}{c||}{Weights : $w=(0.75, \ldots, -0.75, \ldots ,0.75, \ldots, -0.75, \ldots)$} & \multicolumn{4}{c||}{Weights: $w=\sin(8\pi t/n)$} \\ \hline 
$n$ & $h_n$ & \hspace*{0.05cm} $\theta=-0.8$ \hspace*{0.05cm} & \hspace*{0.05cm} $\theta=-0.4$ \hspace*{0.05cm} & \hspace*{0.05cm} $\theta=0.4$ \hspace*{0.05cm} & $\theta=0.8$ & $\theta=-0.8$& $\theta=-0.4$ & $\theta=0.4$ & $\theta=0.8$ \\
\hline \hline $600$ &$0.10$ & $0.930$ & $0.942$ & $0.938$ & $0.906$ & $0.958$ & $0.964$ & $0.924$ & $0.912$\\
\hline & $0.11$& $0.934$ & $0.956$ & $0.938$ & $0.930$ & $0.970$ & $0.968$ & $0.940$ & $0.938$ \\
\hline & $0.12$& $0.950$ & $0.968$ & $0.942$ & $0.938$ & $0.970$ & $0.966$ & $0.950$ & $0.948$ \\
\hline & $0.13$ & $0.956$ & $0.964$ & $0.948$ & $0.940$ & $0.968$ & $0.968$ & $0.952$ & $0.956$ \\
\hline & $0.14$ & $0.944$ & $0.968$ & $0.956$ & $0.950$ & $0.966$ & $0.974$ & $0.962$ & $0.962$\\
\hline & $0.15$ & $0.950$ & $0.966$ & $0.950$ & $0.948$ & $0.968$ & $0.970$ & $0.964$ & $0.962$ \\
\hline \hline $800$ &$0.10$ & $0.950$ & $0.972$ & $0.964$ & $0.942$ & $0.954$ & $0.966$ & $0.962$ & $0.944$ \\
\hline & $0.11$& $0.952$ & $0.960$ & $0.968$ & $0.946$ & $0.966$ & $0.968$ & $0.974$ & $0.946$ \\
\hline & $0.12$& $0.960$ & $0.966$ & $0.960$ & $0.962$ & $0.970$ & $0.972$ & $0.974$ & $0.954$  \\
\hline & $0.13$ & $0.958$ & $0.968$ & $0.970$ & $0.964$ & $0.970$ & $0.972$ & $0.976$ & $0.956$ \\
\hline & $0.14$ & $0.958$ & $0.960$ & $0.964$ & $0.960$ & $0.968$ & $0.966$ & $0.970$ & $0.954$ \\
\hline & $0.15$ & $0.958$ & $0.962$ & $0.972$ & $0.958$ &$0.960$ & $0.962$ & $0.970$ & $0.958$ \\
\hline
\end{tabular}
}
\caption{Empirical coverage probabilities of SCB of $X_t$ from Model \eqref{eq:scb_model} where $Z_t \sim$ Model \ref{ARnonsimsq} with $N(0,1)$ error.}
  \label{tab:scb_sin_t}
\end{table}}

\section{Real data application: analysis of Lake Chichancanab sediment density data}\label{sec:datanalysis}
The Maya civilization, arguably one of the most important pre-Columbian mesoamerican civilizations, underwent a collapse during the last classical period of their history, circa 900-1100 AD \cite{10.2307/26307982, demarest2004ancient, 10.2307/26309152, renfrew_1988}. A severe drought has been hinted at as a primary reason behind this collapse \cite{FAUST2001153,10.2307/26308045, doi:10.1073/pnas.1210106109}, despite the Mayans primarily inhabiting a seasonally dry tropical forest (\cite{golden2004continuities}). Drought has also been explored as a possible cause of a comparatively less-studied preclassical Maya collapse in 150-200 AD (\cite{gill2000great}). \cite{Hodell1995, doi:10.1126/science.1057759, HODELL20051413} analyzed the sediment core density dataset from the Lake Chichancanab in the Yucatan peninsula to analyze the onset pattern of droughts during the Maya civilization. An age-depth model of radiocarbon dating is used to estimate the calendar age of depth of each sediment. The total number of data points is $n=564$, and the corresponding years range from 858 BC to 1994 AD. 

We first test the existence of a change-point for this dataset as described in subsection \ref{subsec:cpt}. For this we choose $m=20$. The $p$-value of our test $\psi_{n1}$ comes out to be $0.09$, and thus we fail to reject non-existence of a change-point. \cite{gill2000great} posited that between 800 and 1000 AD, the Yucatan peninsula was hit by a massive drought, triggering the Mayan collapse. However in light of our findings, such a hypothesis seems unlikely. Next we move on to building a simultaneous confidence band as in \eqref{eq:cons_scb}, which we will subsequently use to test the existence of certain trend. For the local linear estimates (Figures \ref{fig:sfig8}), we select $h=0.1$. The residual plots \ref{fig:sfig7} of $X_i - \hat{\mu}_L(t_i)$ where $\hat{\mu}_L$ is the locally linear estimate, suggest that the error process is indeed non-stationary. \cite{HODELL20051413} concluded that the Yucatan peninsula experienced two drought cycles of period $208$ and $50$ years. This hypothesis has been very influential in shaping academic discussion not only around classical Mayan collapse (\cite{w3020479}, \cite{doi:10.1073/pnas.1210106109}) but also in dialogues involving climate change (\cite{https://doi.org/10.1111/hic3.12140}). In order to test this hypothesis, we fit the following trend function to our data:
\begin{equation} \label{trend_maya}
  \mu(t)= \alpha_0 t + \boldsymbol{\alpha_1}^Tf_S(2\pi t \theta_1) + \boldsymbol{\alpha_2}^Tf_S(2\pi t \theta_2),
\end{equation}
where $\theta_1=208/N$ and $\theta_2=50/N$ with $N$=range of the years in observation, and $f_S(x)=(\sin(x), \cos(x))^T$. Figure \ref{fig:sfig8} shows that based on our $95\%$ SCB, we cannot accept the trend of \eqref{trend_maya}. \cite{carleton2017archaeological} argued that \cite{doi:10.1126/science.1057759,HODELL20051413} used interpolation to turn the irregularly spaced data-points into a regularly spaced one before applying their methods, and the obtained periodicity might have been the superficial result of such method. 
\begin{figure}[!htbp]
\begin{subfigure}{.5\textwidth}
 \centering
 \includegraphics[width=\linewidth, height=0.65\linewidth]{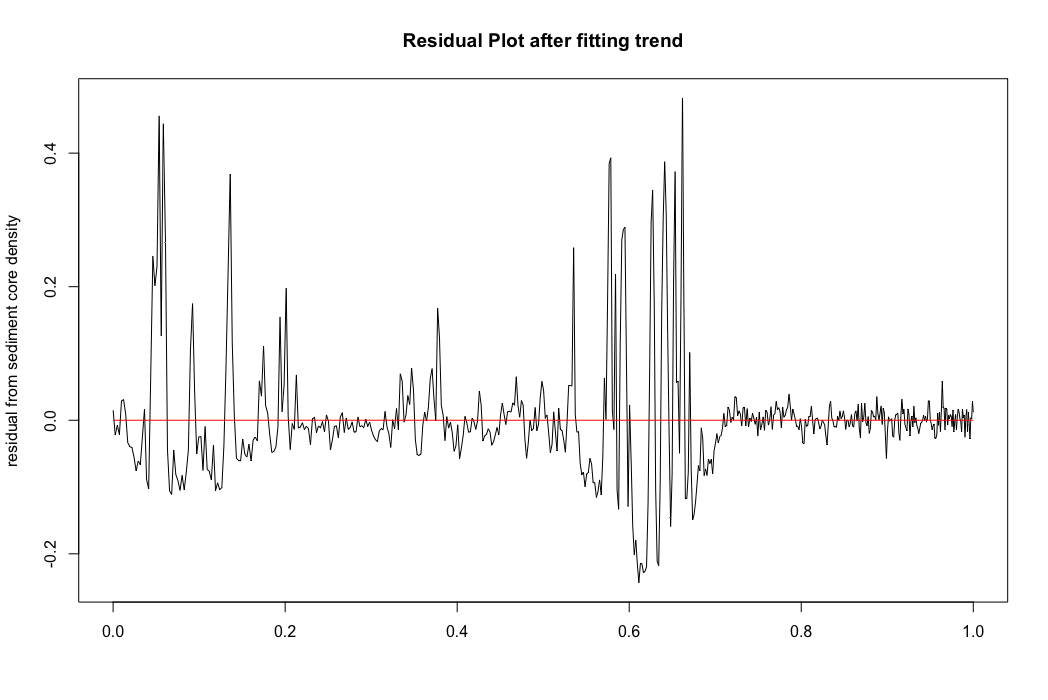}
 \caption{}
 \label{fig:sfig7}
\end{subfigure}%
\begin{subfigure}{.5\textwidth}
 \centering
 \includegraphics[width=\linewidth,height=0.65\linewidth]{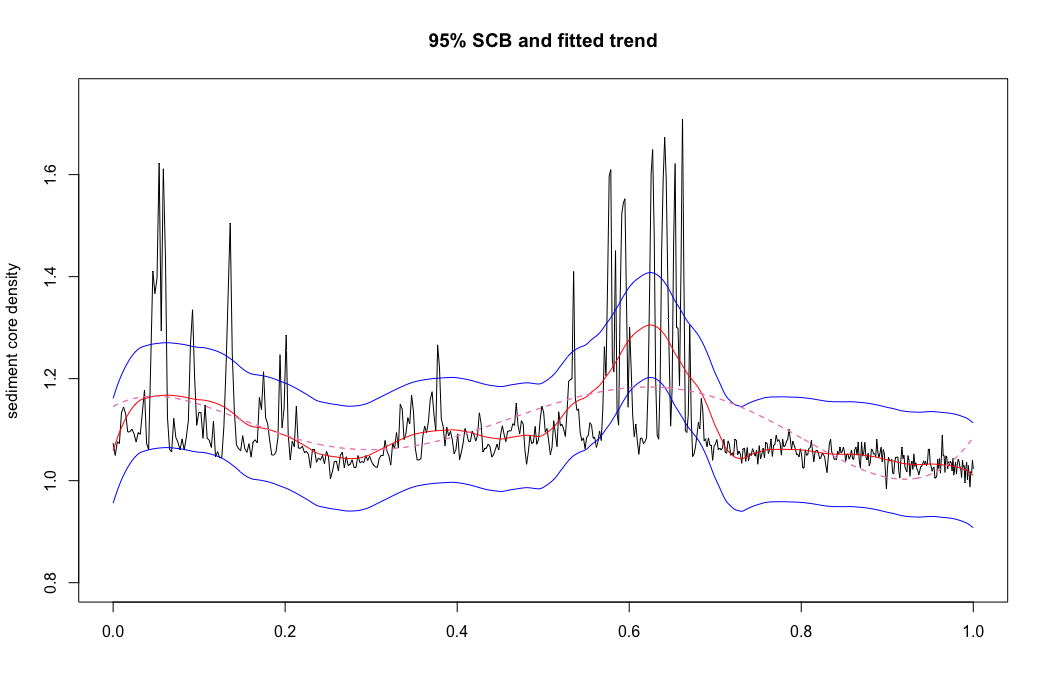}
 \caption{}
 \label{fig:sfig8}
\end{subfigure}
\caption{(a) Plot of the residual $X_i- \hat{\mu}_i$. (b) $95\%$ SCB in blue and the fitted local linear estimate in red. The fitted line \eqref{trend_maya} is in dashed green. }
\label{fig:scb_maya}
\end{figure}

\vspace{-0.1 in}
\section{Discussion}\label{sec:discussion}This paper develops an optimal Gaussian approximation for non-stationary univariate time series, that besides being optimal, also provides a clear instructive way as to how one can construct such approximations for practical applications. Our results match the best possible rates from other literature on non-stationary time series \cite{ MR0375412,MR0402883,kmt,KarmakarWu2020} etc. with relaxed assumptions. 

Our first result is an approximation result that preserves the population second order properties in the approximating Gaussian analogue. Our second, and probably more practically usable result states that the approximating Gaussian process can be embedded in a Brownian motion with evolving variances. A major difficulty in constructing approximating Gaussian processes was the non-availability of the notion of a long-run covariance, and our paper settles this question while maintaining the sharp rate. This work lays out an asymptotic framework which can be used in many areas of non-stationary time series, such as complex non-linear and non-stationary econometric models with smooth or abrupt changes. Moreover, one can further explore beyond just temporal dependence and wish to obtain similar results for complex spatial, spatio-temporal or tensor processes where non-stationarity is quite intrinsic.

{\bf Acknowledgments.} The authors would like to thank the Associate Editor and the reviewers for their constructive feedbacks that helped improving the paper significantly. The second author thanks NSF DMS grant 2124222 for supporting their research. The third author's research is partially supported by NSF DMS grants 1916351, 2027723 and 2311249.

\begin{center}
    SUPPLEMENTARY MATERIAL
\end{center}

\noindent \textbf{Supplement to ``Gaussian approximation for non-stationary time series with optimal rate and explicit construction ''} (; .pdf) contains all proofs in Sections \ref{sec:section2proofs}, \ref{sec2proofsnew}, \ref{sec:section3proofs} and \ref{appendix:scb}, and some additional simulation results in Section \ref{appendix:simu}.

\ignore{
For this paper, we also obtain the rate of Gaussian approximation when the variance is estimated. One could see that we could not better the $n^{1/4}$ rate although with the assumption of higher moments, one could reach the optimal $n^{1/p}$ bound for the theoretical ones. One could possibly show an information-theoretic block in this phenomena that shows this gap whenever such sequence of variance is estimated. Since this was slightly out of scope of our paper, we wish to pursue this in a different paper in near future. 
}

\bibliographystyle{imsart-nameyear}
\bibliography{references}
\newpage

\setcounter{page}{1}
\begin{center}
\begin{large}
   ONLINE SUPPLEMENTARY MATERIAL
\end{large}
\end{center}
 \textbf{Supplement to ``Gaussian approximation for non-stationary time series with optimal rate and explicit construction ''} (; .pdf) contains all proofs in Sections \ref{sec:section2proofs}, \ref{sec2proofsnew}, \ref{sec:section3proofs} and \ref{appendix:scb}, and some additional simulation results in Section \ref{appendix:simu}.
\section{Appendix A: Proofs of Theorems \ref{corol:new} and \ref{thm:GA}} 
\label{sec:section2proofs}
Since both Theorems \ref{corol:new} and \ref{thm:GA} require similar sets of assumptions, we will prove them together. Further, Theorem \ref{thm:GA} does not require the non-singularity Condition \ref{cond:regularity_new} for $(X_t)_{t \geq 1}$. Therefore, we begin by proving this result. 
\subsection{Proof of Theorem \ref{thm:GA}} \label{sec:proofGA}
Recall $A_0$ from \eqref{eq:Alowerbound}. Define, in the light of the form of $\Theta_{i,p}$ in Condition \ref{eq:Thetaip} with $A>A_0$, 
\begin{align}
  L&=\frac{2-f_1+f_2 + \sqrt{f_3+ f_2^2}}{2pf_4}, \label{L}\\
  \alpha &= \frac{2+f_1+f_2 + \sqrt{f_3+f_2^2}}{2+2p+2A}, \nonumber
\end{align}
with 
\begin{align*}
  f_1 (p, A)&= p(3+A), \\
  f_2 (p, A) &= p^2 (1+A), \\
  f_3(p,A) &= 4- 4p(A-1) -p^2 (7A^2+6A+3) + 2p^3(A^2 -1),\\
  f_4(p, A) &= p(A+1)^2-2.
\end{align*}
Specifically, with $A> A_0$, our choice of $L$ and $\alpha$ satisfies the following relations, which will be used in our proofs:
\begin{align}
  &\frac{1}{2}-\frac{1}{p} - \frac{LA}{2} < 0, \label{eq1n}\\
  & L\left(\frac{\alpha}{2}-1\right)+ 1 - \frac{\alpha}{p} <0, \label{eq2n} \\
  &p<\alpha < 2(1+p+pA)/3, \label{eq3n}\\
  & 1/p - 1/\alpha + L - L(A+1)p/\alpha =0. \label{eq4n}
\end{align}

These relations feature crucially in our proof, enabling us to read off certain terms as $o(1)$. In particular, they are important in proving the following three results.
We will employ the following lemma, which uses the uniform integrability condition to control the $p$-th moment of the truncated process. 
\begin{lemma}\label{lem:truncation}
  Assume Conditions \ref{cond:fdm3} and \ref{cond:ui} for the sequence $(X_i)$. Then, 
  \[\sup_{i}\IE(|T_{n^{1/p}}(X_i)|^{\alpha})=o(n^{\alpha/p-1}). \]
\end{lemma}
\begin{proof}
  Note that, for a fixed $a>0$, an application of Condition \ref{cond:ui} entails
  \begin{align}
    \limsup _{n \to \infty} n^{1- \alpha/p} \sup_{i}\IE(|T_{n^{1/p}}(X_i)|^{\alpha}) &= \limsup _{n \to \infty} n^{1- \alpha/p} \sup_{i}\IE(|T_{n^{1/p}}(X_i)|^{\alpha}( \mathbb{I}\{|X_i|^p\leq an \} \nonumber\\ & \hspace*{3cm}+ \mathbb{I}\{|X_i|^p> an \} )) \nonumber\\
    & \leq a^{\alpha / p -1}\sup_{i} \IE(|X_i|^p) + \limsup _{n \to \infty} \sup_{i} n \IP(|X_i|^p > an) \nonumber \\
    & \leq a^{\alpha / p -1} \sup_{i} \IE(|X_i|^p) \label{limsuparg}.
  \end{align}
  Since $\sup_{i} \IE(|X_i|^p)=O(1)$ by Condition \ref{cond:fdm3}, and $a$ can be chosen arbitrarily small, \eqref{limsuparg} completes the proof. 
\end{proof}

\subsubsection{Key lemma}
In this section, first we provide a bound on the $p$-th moment of maximal partial sums. 
{\color{black}
\begin{lemma} \label{pmoment}
Consider Condition \ref{cond:fdm3} for $X_t$ from \eqref{eq:model_Z}. Let $p \ge 2$. Then for any $m \ge 1$, it holds that
  \begin{equation}\label{eq:pmoment}
 \sup_a \| \max_{1 \leq k\leq m} |X_{a+1} + \ldots +X_{a + k}| \|_p \le \frac{p}{\sqrt{p-1}} m^{1/2} \Theta_{0, p}.
  \end{equation} 
\end{lemma}
\begin{proof}
 Let us denote the projection operator $P_{k}(X)=\IE(X|\mathcal{F}_k)-\IE(X|\mathcal{F}_{k-1})$, $R_k= \sum_{i=1}^k X_{a+i}$ and $R_{k,s}=\sum_{i=1}^{k} P_{a+i-s}X_{a+i}, s \ge 0$. Note that $R_k=\sum_{s=0}^{\infty} R_{k,s}$. For fixed $s\geq0$, $(P_{a+i-s}X_{a+i})_{1\leq i \leq m}$ form martingale differences, and therefore, Burkholder's inequality (\cite{Rio2009MomentIF}, Theorem 2.1) entails that
 \begin{align*}
   \|R_{m,s}\|_p^2 &= \|\sum_{i=1}^m P_{a+i-s} X_{a+i} \|_p^2 
   \leq (p-1) \sum_{i=1}^m \|P_{a+i-s} X_{a+i}\|_p^2
   \leq (p-1) m \delta_{p}(s)^2.
 \end{align*}
 where the last assertion follows from Theorem 1 of \cite{Wu2005} and the uniform definition of our functional dependence measure. Finally, Doob's maximal inequality implies that
 \begin{align}\label{doobfinal}
   \|\max_{1 \leq k \leq m} |R_k| \|_p \leq \sum_{s=0}^{\infty} \|\max_{1 \leq k \leq m} |R_{k,s}| \|_p \leq \sum_{s=0}^{\infty} \frac{p}{p-1} \|R_{m,s}\|_p \leq \frac{p}{\sqrt{p-1}} m^{1/2} \Theta_{0,p},
 \end{align}
 which completes the proof of \eqref{eq:pmoment}.
\end{proof}
}
Next, we present a lemma which is one of the main ingredients of our proof. In this result, we raise the partial sums of the truncated, $m$-dependent processes to a power $\alpha>p$, and our specific choice of $\alpha$ allows us to provide a sharp upper bound. We will use this lemma throughout our proof to infer certain quantities are $o(1)$. 
\begin{lemma}\label{Rosenthal}
  Assume Conditions \ref{cond:fdm3} and \ref{cond:ui}, along with \eqref{eq1n}, \eqref{eq2n}, \eqref{eq3n} and \eqref{eq4n} for $A$, $L$ and $\alpha$. Let $m=\lfloor n^L \rfloor$ and let \[\Tilde{R}_{s,t}=\Tilde{X}_{s+1} + \ldots +\Tilde{X}_{s+t}, \] 
  where $\Tilde{X}_i$ is as defined in \eqref{m-dep}. Then 
  \begin{align}
   \sup_s \IE \left[\max_{1 \leq t \leq m} |\tilde{R}_{s,t}|^{\alpha} \right] = o(mn^{\alpha/p-1}).
  \end{align}
  \end{lemma}
  \begin{remark}
     Lemma \ref{Rosenthal} and its proof should be contrasted with Lemma 7.3 of \cite{KarmakarWu2020}, where one requires a sequence $t_n$ converging slowly to zero in both the definition of $m$ and the truncated process $X_i^{\oplus}$. In contrast, our Condition \ref{cond:ui} with the help of Lemma \ref{lem:truncation} enables us to circumvent the need of such sequences.  
  \end{remark}
  \begin{proof}
In the following, $\lesssim$ includes constants depending on $p$ and $A$, emanating from $\mu_{p,A}=O(1)$ from Condition \ref{cond:fdm3}. Let ${\delta}^{\oplus}_p(\cdot)$ and $\Tilde{\delta}_p(\cdot)$ denote the functional dependence measure defined for the truncated and the $m$-dependent processes, respectively. 
    Since the functional dependence measure \eqref{eq:fdm} is defined in a uniform manner, we can ignore the $\sup_s$ term and apply the Rosenthal-type bound in \cite{LiuHanWu} to obtain
    \begin{align}
      \| \max_{1 \leq t \leq m} |\Tilde{R}_{s,t}|\|_{\alpha} & \lesssim m^{1/2} \left[\sum_{j=1}^m \Tilde{\delta}_2(j) + \sum_{m+1}^{\infty} \Tilde{\delta}_{\alpha}(j) + \sup_{i} \| T_{n^{1/p}}(X_i)\| \right] \nonumber \\
      & \hspace*{1cm} + m^{1/\alpha} \left[\sum_{j=1}^m j^{1/2 - 1/\alpha} \Tilde{\delta}_{\alpha}(j) + \sup_{i} \|T_{n^{1/p}}(X_i)\|_{\alpha} \right] \nonumber\\
      & \lesssim I +II +III+IV, \label{liudecomp}
    \end{align}
    where 
    \begin{align*}
      I &= m^{1/2}\bigg(\sum_{j=1}^m \Tilde{\delta}_2(j) + \sup_{i} \|X_i\|\bigg),\\
      II &= m^{1/2} \sum_{j=m+1}^{\infty} \Tilde{\delta}_{\alpha}(j), \\
      III &= m^{1/\alpha} \sum_{j=1}^{m} j^{1/2 - 1/\alpha} \Tilde{\delta}_{\alpha}(j),\\
      IV &= m^{1/\alpha} \sup_{i} \|T_{n^{1/p}}(X_i)\|_{\alpha}.
    \end{align*}
    For $I$, we note that $\Tilde{\delta}_2(j) \leq \delta^{\oplus}_2(j) \leq \delta_2(j)$, and $\sup_{i} \|X_i\| \leq \Theta_{0,2}$. Thus $I=O(m^{1/2})$,
    which yields, using \eqref{eq2n}, 
    \begin{equation}\label{I}
    \frac{n^{1-\alpha/p}}{m}I^{\alpha}=o(1).
    \end{equation}
    
    \noindent For both $II$ and $III$, we start by observing that
    \allowdisplaybreaks
    \begin{align}
      (\Tilde{\delta}_{\alpha}(j))^\alpha \leq \left(\delta^{\oplus}_{\alpha}(j) \right)^{\alpha}
      &= \sup_{i } \IE\left(|T_{n^{1/p}}(X_i) - T_{n^{1/p}}(X_{i, \{i-j\}})|^{\alpha} \right)\nonumber\\
      &\leq n^{\alpha/p} \sup_{i } \IE\left(\left|\min\left(2, \frac{|X_i -X_{i, i-j}|}{n^{1/p} } \right)\right|^{\alpha} \right) \nonumber\\
      & \leq 2^{\alpha} n^{\alpha/p -1} \delta_p(j)^{p} \label{lemma7.2},
    \end{align}
    since $\min(2^{\alpha}, |x|^{\alpha})\leq 2^{\alpha}|x|^p$. Hence, for $II$ we use \eqref{lemma7.2} to get,
    \begin{align*}
       II \leq 2 m^{1/2}n^{1/p -1/\alpha}\sum_{j=m+1}^{\infty}\delta_p(j)^{p/\alpha} 
       & \leq 2m^{1/2}n^{1/p -1/\alpha}\sum_{l=\lfloor\log_2 m\rfloor}^{\infty} \sum_{j=2^l}^{2^{l+1}-1} \delta_p(j)^{p/\alpha} \\
       & \lesssim m^{1/2}n^{1/p -1/\alpha} \sum_{l=\lfloor\log_2 m\rfloor}^{\infty} 2^{l(1-p/\alpha)} \Theta_{2^l, p}^{p/\alpha} \\
       & \lesssim m^{1/2}n^{1/p -1/\alpha} m^{1-p/\alpha - pA / \alpha} = O(m^{1/2})
    \end{align*}
    using \eqref{eq4n}. Thus \eqref{eq2n} leads to 
    \begin{equation} \label{II}
    \frac{n^{1-\alpha/p}}{m}II^{\alpha}=o(1).
    \end{equation}
    \allowdisplaybreaks
    In light of \eqref{lemma7.2}, for $III$, we proceed as following 
    \begin{align}
      \sum_{j=1}^{m} j^{1/2 - 1/ \alpha} \delta_p(j)^{p/ \alpha} &\leq \sum_{l=1}^{\lceil\log_2(m)\rceil} \sum_{j=2^l}^{2^{l+1}-1} j^{1/2 - 1/ \alpha}\delta_p(j)^{p/\alpha} \nonumber \\
      & \lesssim \sum_{l=1}^{\lceil\log_2(m)\rceil} 2^{l(3/2 -1/\alpha-p/\alpha)} \Theta_{2^l, p}^{p/\alpha} = O(1) \ \ \text{(Using \eqref{eq3n})}. \label{eq2.3kmt}
    \end{align}
    Fix $J$. Using $\Tilde{\delta}_{\alpha}(j)^{\alpha} \leq \delta^{\oplus}_{\alpha}(j)^{\alpha} \leq C\sup_{i }\IE(|T_{n^{1/p}}(X_i)|^{\alpha})$ in conjunction with Lemma \ref{lem:truncation} yields
 \begin{equation} \label{1stsum}
   n^{1-\alpha/p}( j^{1/2 -1/\alpha} \Tilde{\delta}_{\alpha}(j))^{\alpha} =o(1),
 \end{equation}
for each $1\leq j \leq J$. 
 Thus, using \eqref{lemma7.2} along with \eqref{eq2.3kmt} and \eqref{1stsum}, 
\begin{align}
  \limsup_{n \to \infty}\frac{n^{1-\alpha/p}}{m} III^{\alpha} &= \limsup_{n \to \infty}n^{1 - \alpha/p} \left(\sum_{j=1}^{\infty} j^{1/2 - 1/\alpha} \Tilde{\delta}_{\alpha}(j) \right)^{\alpha}\nonumber \\& \leq\limsup_{n \to \infty} c_{\alpha}n^{1-\alpha/p}\left(\sum_{j=1}^J j^{1/2 -1/\alpha} \Tilde{\delta}_{\alpha}(j)\right)^{\alpha} + c_{\alpha} \left( \sum_{j=J+1}^{\infty} j^{1/2 - 1/\alpha} \delta_p(j)^{p/ \alpha} \right)^{\alpha}\nonumber \\ &= c_{\alpha}\left( \sum_{j=J+1}^{\infty} j^{1/2 - 1/\alpha} \delta_p(j)^{p/ \alpha} \right)^{\alpha} \nonumber,
\end{align}
which, in view of \eqref{eq2.3kmt}, implies 
\begin{equation}\label{III}
   \limsup_{n \to \infty}\frac{n^{1/\alpha-1/p}}{m} III^{\alpha} \leq c_{\alpha} \limsup_{J \to \infty} \left( \sum_{j=J+1}^{\infty} j^{1/2 - 1/\alpha} \delta_p(j)^{p/ \alpha} \right)^{\alpha} =0 .
\end{equation}
Finally, for $IV$, using Lemma \ref{lem:truncation}, one obtains
  \begin{align}
    \frac{n^{1-\alpha/p}}{m} IV^{\alpha} = n^{1-\alpha/p} \sup_{i} \|T_{n^{1/p}}(X_i)\|^{\alpha}_{\alpha}= o(1). \label{IV}
  \end{align}
The proof is now completed combining  \eqref{I}, \eqref{II}, \eqref{III} and \eqref{IV}.
  \end{proof}
\color{black}

Now we are ready to prove our first Gaussian approximation result.

\begin{proof}[Proof of Theorem \ref{thm:GA}]
The proof can be divided broadly in seven steps, which we discuss in the following sections. In the following, we will use $C$ to denote a generic positive constant. The value of this constant can change from line to line.

\subsubsection{Truncation}
Recall $S_i^{\oplus}$ from \eqref{trunc}. In this section we derive a result showing the effectiveness of the truncated partial sum process in optimally approximating the original partial sum process $(S_i)_{i \geq 1}$. 
\begin{proposition} \label{prop1}
Under the conditions of Theorem \ref{thm:GA}, $\max_{1\leq i \leq n}| S_i - S_i^{\oplus}| = o_{\IP}(n^{1/p})$. 
\end{proposition}
\begin{proof}
 Note that by the truncated uniform integrability Condition \ref{cond:ui},
\begin{equation}\label{eq:1st}
\max_{1\leq j \leq n}\IP(|X_j| > n^{1/p}) \leq \frac{1}{n}\max_{1\leq j \leq n} \IE\left(|X_j|^p \mathbb{I}_{|X_j| \geq n^{1/p}} \right) = o(n^{-1}). \color{black}
\end{equation}
Thus,
\begin{equation*}
  \IP(\max_{1\leq i \leq n} |S_i - \sum_{j=1}^i X_i^{\oplus}|> 0)
  \leq \IP(\max_{1\leq j \leq n} |X_j| > n^{1/p})
  \leq n \max_{1\leq j \leq n}\IP(|X_j| > n^{1/p}) \to 0.
\end{equation*}
 Hence, 
 \begin{equation}\label{eq2}
   \max_{1\leq i \leq n} |S_i - \sum_{j=1}^i X_i^{\oplus}| = o_{\IP}(1).
 \end{equation}
 \color{black}
Next, note that by \eqref{eq:1st},
 \[ X_i - X_i^{\oplus}= (X_i -n^{1/p})\mathbb{I}\{X_i > n^{1/p} \}+ (X_i+n^{1/p})\mathbb{I}\{X_i < -n^{1/p} \}. \]
  This immediately implies
  \[ |X_i - X_i^{\oplus}| \leq |X_i| \mathbb{I}\{|X_i|>n^{1/p} \}\leq n^{1/p -1} |X_i|^p\mathbb{I}\{|X_i|>n^{1/p}\}, \]
  which, upon invoking Condition \ref{cond:ui}, yields
\begin{equation}\label{eq3.11}
  \max_{1\leq i \leq n} |\IE(X_i - X_i^{\oplus})| =o(n^{{1/p}-1}).
\end{equation}\color{black}
Therefore,
\begin{equation}\label{eq3}
  \max_{1\leq i \leq n} | \IE(S_i - \sum_{j=1}^i X_j^{\oplus})|\leq \sum_{j=1}^n | \IE(X_j - X_j^{\oplus})| \leq n \max_{1\leq j \leq n} | \IE(X_j - X_j^{\oplus})| = o(n^{1/p}),
\end{equation}
which by (\ref{eq2}) and (\ref{eq3}) completes the proof. 
\end{proof}
\subsubsection{$m$-dependence}\label{mdep}
$m$-dependence approximation is a useful tool which allows us to handle the truncated process in terms of the innovations $\varepsilon_i$. This technique has been studied extensively in the literature; see for example \cite{LiuLin2009} and \cite{kmt}. For a suitably chosen $m$, one looks at the conditional mean $\IE(X_i|\varepsilon_i, \ldots, \varepsilon_{i-m})$. More formally, let $m=\lfloor n^L \rfloor$. Define the $m$-dependent partial sum process
\begin{align}
\tilde{S}_i=\sum_{j=1}^i \tilde{X}_j, \text{ where }\tilde{X_j}=\IE(X_j^{\oplus} | \varepsilon_{j}, \ldots, \varepsilon_{j-m})- \IE(X_j^{\oplus}) \label{m-dep}.
\end{align}
We will need the following proposition.
\begin{proposition} \label{prop:prop2}
Under the conditions of Theorem \ref{thm:GA}, $\max_{1 \leq i \leq n} |S_i^{\oplus} - \tilde{S_i}| = o_{\IP}(n^{1/p})$. \label{prop2}
\end{proposition}
\begin{proof}
Using Lemma A1 of \cite{LiuLin2009} (\color{black} although the lemma holds for stationary random variables, however the proof can be verified to be readily applicable to the non-stationary case\color{black}) and \eqref{eq1n}, we have
\begin{equation}
  \| \max_{1 \leq i \leq n} | S_i^{\oplus}- \tilde{S_i} | \|_p \leq C n^{1/2} \Theta_{1+m, p} = o(n^{1/p}),
\end{equation}
which completes the proof. 
\end{proof}
\subsubsection{Blocking} \label{blocksec}
We will form blocks of sums of $m$-dependent process obtained above. Such blocking will make it easier to control the dependency structure of our process, resulting in optimal error bounds. For the same choice of $m$ as in Section \ref{mdep} , and for $1 \leq i \leq n$, denote
\begin{equation}\label{notation}
  l_i= \bigg\lceil\frac{\lceil i/m \rceil}{3}\bigg \rceil.
\end{equation} 
For $k=1, \ldots, \lceil n/m \rceil$, consider blocks of $m$-dependent processes
\begin{equation*}
  \tilde{B}_k=\sum_{j=(k-1)m+1}^{(km) \wedge n} \tilde{X_j}.
\end{equation*}
 Similarly for our original process we will define blocks 
\begin{equation}\label{blockdefn}
  {B}_k=\sum_{j=(k-1)m+1}^{(km) \wedge n} {X_j} .
\end{equation}
For the blocking approximation we approximate the partial sum process $\Tilde{S}_i$ by \[ S_{i}^{\diamond}=\sum_{l=1}^{l_i}\sum_{k=3l-2}^{3l \wedge \lceil n/m \rceil} \tilde{B}_k.\] The following proposition justifies the blocking approximation.
\begin{proposition} \label{prop3}
Under conditions of Theorem \ref{thm:GA}, $\max_{1\leq i \leq n}| \tilde{S}_i - S_i^{\diamond}| = o_{\IP}(n^{1/p})$. 
\end{proposition}
\begin{proof}
\color{black}For $1\leq k \leq n$ and $l > k$, \color{black} let $\tilde{S}_{k,l}=\sum_{j=k+1}^l \tilde{X}_j$ with $S_{k,k}=0$. Note that, for $\delta>0$, 
\begin{align*}
\IP(\max_{1\leq i \leq n} |\tilde{S}_i - S_i^{\diamond}| > n^{1/p} \delta) &\leq l_n \max_{1\leq l \leq l_n} \IP\left(\max_{3lm \leq j \leq 3(l+1)m} |\tilde{S}_{3lm,j}| > n^{1/p}\delta\right)\\
& \leq C\frac{n}{3m} \max_{1\leq l \leq l_n} \frac{\IE(\max_{3lm \leq j \leq 3(l+1)m}|\tilde{S}_{3lm, j}|^{\alpha} )}{\delta^\alpha n^{\alpha/p}} \\
& =o(1). \ \text{(By Lemma \ref{Rosenthal})} 
\end{align*}
Therefore, $\max_{1\leq i \leq n}| \tilde{S}_i - S_i^{\diamond}| = o_{\IP}(n^{1/p})$. 
\end{proof}

\subsubsection{Conditional Gaussian approximation}
The blocking step in Section \ref{blocksec} yields us $m$-dependent blocks. The dependence between these blocks is induced by the shared innovations $\varepsilon_i$'s along the border. In this subsection, we condition on these shared $\varepsilon_i$'s and apply Theorem 1 of \cite{Sakhanenko2006} to obtain a conditional Gaussian approximation. 

In order to properly explain the conditioning argument, we require some notation. Let $\boldsymbol{\eta}=(\ldots,\boldsymbol{\eta}_{-3}, \boldsymbol{\eta}_0, \boldsymbol{\eta}_3, \ldots)$, where $\boldsymbol{\eta}_{k}=(\varepsilon_{(k-1)m+1}, \ldots, \varepsilon_{km})$. We will use an argument conditioning via $\boldsymbol{\eta}$. To facilitate such arguments, denote by $\boldsymbol{a}$ an arbitrary deterministic sequence $(\ldots, \boldsymbol{a}_{-3}, \boldsymbol{a}_0, \boldsymbol{a}_3, \ldots)$ with $\boldsymbol{a}_{k}=(a_{(k-1)m+1}, \ldots, a_{km})$.

Let $\Tilde{g}_i$ be measurable functions such that we can write $\Tilde{X}_i=\Tilde{g}_i(\varepsilon_{i-m}, \ldots, \varepsilon_i).$
Recall $l_n$ from \eqref{notation}. For $1\leq k \leq l_n$, define the random functions, 
\begin{align*}
  \Tilde{B}_{3k-2}(\boldsymbol{a}_{3k-3}) &= \sum_{i=(3k-3)m+1}^{(3k-2)m} \Tilde{g}_i(a_{i-m}, \ldots, a_{(3k-3)m}, \varepsilon_{(3k-3)m+1}, \ldots, \varepsilon_i),\\
  \Tilde{B}_{3k-1} &= \sum_{i=(3k-2)m+1}^{(3k-1)m} \Tilde{g}_i(\varepsilon_{i-m}, \ldots, \varepsilon_{(3k-2)m}, \ldots, \varepsilon_i),\\
  \Tilde{B}_{3k}(\boldsymbol{a}_{3k})&=\sum_{i=(3k-1)m+1}^{3km} \Tilde{g}_i(\varepsilon_{i-m}, \ldots, \varepsilon_{(3k-1)m}, a_{(3k-1)m+1}, \ldots, a_i) .
\end{align*}

\noindent \ignore{Let $\IE^{*}$ denote the conditional expectation given $\boldsymbol{a}$, that is $\IE^{*}(\cdot)=\IE(\cdot | \boldsymbol{\eta}=\boldsymbol{a})$.} For $1 \leq l \leq l_n$, let
\begin{align}
  M_{3l}(\boldsymbol{a}_{3l})&=\IE^{}(\tilde{B}_{3l}(\boldsymbol{a}_{3l})), \nonumber\\
  M_{3l-2}(\boldsymbol{a}_{3l-3}) &= \IE^{}(\tilde{B}_{3l-2}(\boldsymbol{a}_{3l-3})), \ \text{and} \nonumber \\
  Y_l^{\boldsymbol{a}}:= Y_l(\boldsymbol{a}_{3l-3}, \boldsymbol{a}_{3l})&= \tilde{B}_{3l-2}(\boldsymbol{a}_{3l-3}) - M_{3l-2}(\boldsymbol{a}_{3l-3}) + \tilde{B}_{3l-1}+ \tilde{B}_{3l}(\boldsymbol{a}_{3l})-M_{3l}(\boldsymbol{a}_{3l}). 
\end{align}


In the rest of this sub-section, unless otherwise specified, we will treat $\boldsymbol{a}$ as fixed. 
Note that in $Y_l^{\boldsymbol{a}}$'s, we have combined three blocks together to combine an "outer" layer of blocks. Further, due to our conditioning (by $\boldsymbol{\eta}=\boldsymbol{a}$), $Y_l^{\boldsymbol{a}}$'s are independent. The corresponding mean and variance functionals, for $1\leq k \leq l_n$, are respectively,
\begin{align}\label{meanvarY}
  M_k(\boldsymbol{a}) &= \sum_{l=1}^{k} [M_{3l-2}(\boldsymbol{a}_{3l-3}) + M_{3l}(\boldsymbol{a}_{3l})],\\
  Q_k(\boldsymbol{a}) &= \sum_{l=1}^{k} V_l(\boldsymbol{a}_{3l-3}, \boldsymbol{a}_{3l}),
\end{align}
where $V_l(\boldsymbol{a}_{3l-3}, \boldsymbol{a}_{3l})$ is the variance of $Y_l^{\boldsymbol{a}}$. Note that $\IE(Y_l^{\boldsymbol{a}})=0$. Let $C_{3l-1}(\boldsymbol{\eta}_{3l-2})=\IE[\tilde{B}_{3l-1}| \boldsymbol{\eta}_{3l-2}]$. We will decompose $V_l$ as follows:
\begin{align*}
  &V_l(\boldsymbol{a}_{3l-3}, \boldsymbol{a}_{3l})\\ &= \IE^{}(Y_l(\boldsymbol{a}_{3l-3}, \boldsymbol{a}_{3l})^2)\\
  &= \IE^{}\left[\left(\IE^{}[Y_l(\boldsymbol{a}_{3l-3}, \boldsymbol{a}_{3l}) | \boldsymbol{\eta}_{3l-2}, \boldsymbol{\eta}_{3l-1}] - \IE^{}[Y_l(\boldsymbol{a}_{3l-3}, \boldsymbol{a}_{3l}) | \boldsymbol{\eta}_{3l-2}] \right)^2\right] + \IE^{}\left[ \left(\IE^{}[Y_l(\boldsymbol{a}_{3l-3}, \boldsymbol{a}_{3l}) | \boldsymbol{\eta}_{3l-2}]\right)^2\right]\\
  &= \IE^{}\left[\left(\tilde{B}_{3l-1}- C_{3l-1}(\boldsymbol{\eta}_{3l-2}) + \tilde{B}_{3l}(\boldsymbol{a}_{3l})-M_{3l}(\boldsymbol{a}_{3l})\right)^2\right] \\
  & \hspace*{2cm}+ \IE^{}\left[\left(\tilde{B}_{3l-2}(\boldsymbol{a}_{3l-3}) - M_{3l-2}(\boldsymbol{a}_{3l-3})+ C_{3l-1}(\boldsymbol{\eta}_{3l-2}) \right)^2 \right]\\
  &:= \tilde{V}_{2l}(\boldsymbol{a}_{3l}) + \tilde{V}_{2l-1}(\boldsymbol{a}_{3l-3}).
\end{align*}
 Let 
\begin{equation}\label{eq:vl0}
  V^{0}_l(\boldsymbol{a}_{3l})= \tilde{V}_{2l}(\boldsymbol{a}_{3l}) + \tilde{V}_{2l+1}(\boldsymbol{a}_{3l}).
\end{equation}
Then, for all $t\in \mathbb{N}$, 
\begin{equation} \label{eq:3.21}
  \sum_{l=1}^t V_l(\boldsymbol{a}_{3l-3}, \boldsymbol{a}_{3l}) = \tilde{V}_{1}(\boldsymbol{a}_0) + \sum_{l=1}^{t-1} V^{0}_l(\boldsymbol{a}_{3l})+ \tilde{V}_{2t}(\boldsymbol{a}_{3t}).
\end{equation}
Let, for $\alpha$ satisfying \eqref{eq1n}-\eqref{eq4n}, 
\allowdisplaybreaks
\begin{align*}
  L_{\alpha}(\boldsymbol{a},x) &= \sum_{l=1}^{l_n} \IE \min \{|Y_l(\boldsymbol{a}_{3l-3}, \boldsymbol{a}_{3l})/x|^{\alpha}, |Y_l(\boldsymbol{a}_{3l-3}, \boldsymbol{a}_{3l})/x|^2\}\\
  &\leq \sum_{l=1}^{l_n} \IE \left[|Y_l(\boldsymbol{a}_{3l-3}, \boldsymbol{a}_{3l})/x|^{\alpha} \right].
\end{align*}
Then by Theorem 1 of \cite{Sakhanenko2006}, there exists a probability space $(\Omega_{\boldsymbol{a}}, \mathcal{A}_{\boldsymbol{a}}, \mathbf{P}_{\boldsymbol{a}})$ where we can define random variables $(\mathcal{R}_l(\boldsymbol{a}))_{1 \leq l \leq l_n} =_{\mathbb{D}} (Y_l(\boldsymbol{a}_{3l-3}, \boldsymbol{a}_{3l}))_{1 \leq l \leq l_n}$, and Brownian motion $\mathbf{B}_{\boldsymbol{a}}$, such that 
\begin{equation} \label{condnGA}
  \IP^{} \left(\max_{1\leq i \leq n} | \Gamma_i(\boldsymbol{a})- \mathbf{B}_a(Q_{l_i}(\boldsymbol{a}))| \geq c \alpha x \right) \leq L_{\alpha}(\boldsymbol{a},x), \mbox{ where } \Gamma_i(\boldsymbol{a})=\sum_{j=1}^{ l_i } \mathcal{R}_j(\boldsymbol{a}).
\end{equation}
Here $c>0$ is an absolute constant. Now we will incorporate the randomness coming from $\boldsymbol{\eta}$ in our conditional Gaussian approximation \eqref{condnGA}.
Using $x=n^{1/p}$ and Lemma \ref{Rosenthal},
\begin{align*}
\IE(L_{\alpha}(\boldsymbol{\eta},x)) \leq \frac{n^{1-\alpha/p}}{m} C \max_{1\leq l\leq n-3m}\IE[|\Tilde{S}_{l, 3m+l}|^{\alpha}] = o(1).
\end{align*}
Thus, we have, 
\begin{equation}\label{eq:condn1}
  \max_{1\leq i \leq n} | \Gamma_i(\boldsymbol{\eta})- \mathbf{B}_{\boldsymbol{\eta}}(Q_{l_i}(\boldsymbol{\eta}))|= o_{\IP}(n^{1/p}).
\end{equation}\color{black}
Similar to \cite{kmt}, the probability space for the above in-probability convergence is 
\[
\left(\Omega_*, \mathcal{A}_*, \mathrm{P}_*\right)=(\Omega, \mathcal{A}, \mathbb{P}) \times \prod_{\tau \in \Omega}\left(\Omega_{\boldsymbol{\eta}(\tau)}, \mathcal{A}_{\boldsymbol{\eta}(\tau)}, \mathbf{P}_{\boldsymbol{\eta}(\tau)}\right),
\]
where $(\Omega, \mathcal{A}, \mathbb{P})$ is the probability space on which the random variables $\left(\varepsilon_i\right)_{i \in \mathbb{Z}}$ are defined and, for a set $A \subset \Omega_*$ with $A \in \mathcal{A}_*$, the probability measure $\mathrm{P}_*$ is defined as
$$
\mathrm{P}_*(A)=\int_{\Omega} \mathbf{P}_{\boldsymbol{\eta}(\omega)}\left(A_\omega\right) \mathbb{P}(d \omega),
$$
where $A_\omega$ is the $\omega$-section of $A$. Here we recall that, for each $\boldsymbol{a},\left(\Omega_{\boldsymbol{a}}, \mathcal{A}_{\boldsymbol{a}}, \mathbf{P}_{\boldsymbol{a}}\right)$ is the probability space carrying $\mathbf{B}_{\boldsymbol{a}}$ and $\mathcal{R}_l(\boldsymbol{a})$ given $\boldsymbol{\eta}=\boldsymbol{a}$. On the probability space $\left(\Omega_*, \mathcal{A}_*, \mathrm{P}_*\right)$, the random variable $\mathcal{R}_{l}(\boldsymbol{\eta}) $ is defined as $\mathcal{R}_{l}(\boldsymbol{\eta})(\omega, \theta(\cdot))=$ $\mathcal{R}_{l}(\boldsymbol{\eta}(\omega))(\theta(\omega))$, where $(\omega, \theta(\cdot)) \in \Omega_*, \theta(\cdot)$ is an element in $\prod_{\tau \in \Omega} \Omega_{\boldsymbol{\eta}(\tau)}$ and $\theta(\tau) \in \Omega_{\boldsymbol{\eta}(\tau)}, \tau \in \Omega$. The other random processes $M_{l_i}({\boldsymbol{\eta}})$ and $\mathbb{B}_{\boldsymbol{\eta}}\left(Q_{l_i}(\boldsymbol{\eta})\right)$ can be similarly defined. 
\color{black}
\subsubsection{Unconditional Gaussian approximation}
In this subsection, we lift the condition on the shared innovations and work with $\Gamma_i(\boldsymbol{\eta})$ and $Q_{l_i}(\boldsymbol{\eta})$. 
Again, given $\boldsymbol{a}$, using \eqref{eq:3.21}, for i.i.d. standard normal random variables $(Z_{k}^{\boldsymbol{a}})_{k=1}^{l_n}$, let \[\omega_k(\boldsymbol{a}):=\sum_{l=1}^{k-1} \sqrt{V_l^{0}(\boldsymbol{a}_{3l})} Z_{l}^{\boldsymbol{a}},\text{ and }R_k(\boldsymbol{a})=\sqrt{\tilde{V}_1(\boldsymbol{a}_0)} Z_0^{\boldsymbol{a}}+\sqrt{\tilde{V}_{2k}(\boldsymbol{a}_{3k})} Z_k^a,\]
such that we can write
\[\mathbf{B}_a(Q_k(\boldsymbol{a}))= \omega_k(\boldsymbol{a})+R_k(\boldsymbol{a}).\]
For $1 \leq k \leq l_n$, denote $Z_k^{\eta}:=Z_k$ to be the i.i.d. standard normal random variables independent of $\varepsilon$'s, and define
\begin{equation}
  \Phi_i = \sum_{k=1}^{l_i-1} \sqrt{V_k^{0}(\boldsymbol{\eta}_{3k})} Z_k.
\end{equation}
Note that $(\Phi_i | \{\boldsymbol{\eta}=\boldsymbol{a}\})_{i\geq 1} =_{\mathbb{D}} (\omega_{l_i}(\boldsymbol{a}))_{i \geq 1}$. Hence it holds
$(\Phi_i)_{i \geq 1}=_{\mathbb{D}} (\omega_{l_i}(\boldsymbol{\eta}))_{i \geq 1 }$.
By Jensen's inequality and Lemma \ref{Rosenthal}, we have
\begin{equation} \label{eq:small1}
  \max_{1 \leq k \leq l_n} \|\tilde{V}_{2k}(\boldsymbol{\eta}_{3k})\|^{\alpha/2} = o(mn^{\alpha/p-1}),
\end{equation}
which implies, for $C>0$,  
\begin{align*}
  \IP(\max_{1\leq k \leq l_n} |\tilde{V}_{2k}(\boldsymbol{\eta}_{3k})| \geq Cn^{2/p}) \leq \sum_{k=1}^{l_n} \IP(|\tilde{V}_{2k}(\boldsymbol{\eta}_{3k})| \geq Cn^{2/p})&\leq C^{-\alpha/2}\frac{n}{3m} n^{-\alpha/p}\|\tilde{V}_{2k}(\boldsymbol{\eta}_{3k})\|^{\alpha/2} \\ &=o(1).
\end{align*}
Therefore, 
\begin{equation}\label{eq:small2}
  \max_{1 \leq k \leq l_n} |\tilde{V}_{2k}(\boldsymbol{\eta}_{3k})|=o_{\IP}(n^{2/p}).
\end{equation}
Similarly, $|\tilde{V}_1(\boldsymbol{\eta}_0)|=o_{\IP}(n^{2/p})$. Thus, 
\begin{equation} \label{eqn:uncondn1}
  \max_{1\leq i \leq n} |\mathbf{B}_{\boldsymbol{\eta}}(Q_{l_i}(\boldsymbol{\eta})) - \omega_{l_i}(\boldsymbol{\eta}) | = o_{\IP}(n^{1/p}). 
\end{equation}
From (\ref{eq:condn1}) and (\ref{eqn:uncondn1}), we have
\begin{equation} \label{3.30}
  \max_{1 \leq i \leq n} | \Gamma_i(\boldsymbol{\eta})- \omega_{l_i}(\boldsymbol{\eta})| = o_{\IP}(n^{1/p}).
\end{equation}
Recall \eqref{meanvarY}. We have the following distributional equality
\begin{equation} \label{3.31}
  (\Gamma_i(\boldsymbol{\eta}) + M_{l_i}(\boldsymbol{\eta}))_{1 \leq i \leq n} =_{\mathbb{D}} (S_i^{\diamond})_{1 \leq i \leq n}.
\end{equation}
In view of \eqref{3.30} and \eqref{3.31}, we need to prove Gaussian approximation for $\Phi_i +M_{l_i}(\boldsymbol{\eta})$. For $1\leq i \leq n$ let
\begin{equation} \label{A_l}
  S_{i}^{\natural}=\sum_{l=1}^{l_i} \tilde{A}_l, \text{ where }\tilde{A}_l = \sqrt{V_l^{0}(\boldsymbol{\eta}_{3l})} Z_l + M_{3l}(\boldsymbol{\eta}_{3l}) + M_{3l+1}(\boldsymbol{\eta}_{3l}), 1 \leq l \leq l_n.
\end{equation}
 Note that by the same argument as in (\ref{eq:small1}) and (\ref{eq:small2}), we have
\begin{equation} \label{3.33}
  \max_{1\leq i \leq n} | \Phi_i + M_{l_i}(\boldsymbol{\eta}) - S_i^{\natural} |= \max_{1 \leq i \leq n} |\sqrt{V_{l_i}^0(\boldsymbol{\eta}_{3l})}Z_{l_i}-M_{3l_i+1}(\boldsymbol{\eta}_{3l}) + M_{1}(\boldsymbol{\eta}_{0}) | = o_{\IP}(n^{1/p}).
\end{equation}
Hence, by Theorem 1 of \cite{Sakhanenko2006} (ignoring the technicalities of enriched space), on the same probability space on which $\Tilde{A}_l$'s are defined, we have a standard Brownian motion $\mathbb{B}(\cdot)$ such that
\begin{equation} \label{3.34}
  \max_{1\leq i \leq n} |S_{i}^{\natural} - \mathbb{B}(\sigma_i^2)| = o_{\IP}(n^{1/p}), \text{ where $\sigma_i^2=\sum_{l=1}^{l_i} \|\tilde{A}_l\|^2$.}
\end{equation}
\subsubsection{Approximation of the variance}
In this final step of our proof, we aim to provide an approximation to $\sigma_i^2$ in terms of the variance of the truncated random process $X_i^{\oplus}$. To that end, we start from the expression of $V_l^0(\boldsymbol{a}_{3l})$ in equation (\ref{eq:vl0}):
\begin{align}
  V_l^0(\boldsymbol{a}_{3l})&=\| \tilde{B}_{3l-1} \|^2 + \| \tilde{B}_{3l}(\boldsymbol{a}_{3l}) \|^2 - M_{3l}^2(\boldsymbol{a}_{3l}) + \|\Tilde{B}_{3l+1}(\boldsymbol{a}_{3l}) \|^2 \nonumber\\ & \quad\quad\quad - M_{3l+1}^2(\boldsymbol{a}_{3l}) - \|C_{3l-1}(\boldsymbol{\eta}_{3l-2})\|^2 + \|C_{3l+2}(\boldsymbol{\eta}_{3l+1})\|^2 \nonumber\\ & \quad\quad\quad + 2\IE(\Tilde{B}_{3l-1}\Tilde{B}_{3l}(\boldsymbol{a}_{3l}))+ 2\IE(\Tilde{B}_{3l+1}(\boldsymbol{a}_{3l})C_{3l+2}(\boldsymbol{\eta}_{3l+1})). \label{v_l0}
\end{align}
Using \eqref{v_l0} in view of the definition of $A_l$ in \eqref{A_l} yields, 
\begin{align}\label{eq:tildeA}
  \Tilde{v}_l:= \| \tilde{A}_l\|^2= \| \tilde{B}_{3l-1} \|^2 &+ \| \tilde{B}_{3l} \|^2 + \| \tilde{B}_{3l+1} \|^2 \nonumber\\ & \quad\quad\quad + 2\IE(\tilde{B}_{3l-1} \tilde{B}_{3l}) + 2\IE(\tilde{B}_{3l} \tilde{B}_{3l+1}) + 2 \IE(\tilde{B}_{3l+1} \tilde{B}_{3l+2}) \nonumber\\ & \quad\quad\quad  - \|C_{3l-1}(\boldsymbol{\eta}_{3l-2})\|^2 + \|C_{3l+2}(\boldsymbol{\eta}_{3l+1})\|^2.
\end{align}
Hence, 
\begin{align*}
  \sigma_i^2 &= \sum_{l=1}^{l_i} (\|\Tilde{B}_{3l-1}\|^2 + \|\Tilde{B}_{3l} \|^2 + \|\Tilde{B}_{3l+1}^2\| + 2 \IE[\Tilde{B}_{3l-1}\Tilde{B}_{3l} + \Tilde{B}_{3l}\Tilde{B}_{3l+1}+\Tilde{B}_{3l+1}\Tilde{B}_{3l+2}] )
  \\ & \hspace*{2cm}- \|C_{2}(\boldsymbol{\eta}_1)\|^2 + \|C_{3l_i+2}(\boldsymbol{\eta}_{3l_i+1})\|^2.
\end{align*}
We define the functional dependence measure for the process $\Tilde{X}_i$ as
\[\Tilde{\delta}_{p}(k)=\sup_{i}\|\Tilde{X}_i-\Tilde{X}_{i, {i-k}}\|_p, \]
where $p > 2$ is same as in the statement of Theorem \ref{thm:GA}. We can easily have the following simple relation:
 \begin{equation}\label{ineqfdm}
 \quad\Tilde{\delta}_{p}(k) \leq\delta_{p}(k).
 \end{equation}
Lemma \ref{lemmacp}, Lemma \ref{pmoment} and \eqref{ineqfdm} along with observing that $\max_{1 \leq k \leq \lceil n/m \rceil} \IE(B_k^2)= O(m)$ yields
\[ \max_{1 \leq i \leq n} |\|\Tilde{S}_i \|^2 - \sigma_i^2 | =O(m) = o(n^{(\alpha/p-1)/(\alpha/2-1)}),
\]
in view of \eqref{eq2n}. This implies 
\begin{equation} \label{eq:tildeapprox}
  \max_{1\leq i \leq n} |\mathbb{B}(\sigma_i^2) - \mathbb{B}(\|\Tilde{S}_i\|^2)|=o_{\IP}(n^{(\alpha/p-1)/(\alpha-2)}\sqrt{\log n})=o_{\IP}(n^{1/p}).
\end{equation}
Now, using $\|S_n^{\oplus} - \Tilde{S}_n \| \leq \sqrt{n}\Theta_{m,2} \leq \sqrt{n}\Theta_{m,p}$ and \eqref{eq1n}, we obtain, 
\[ |\|S_i^{\oplus}\|^2 - \|\Tilde{S}_i \|^2 | \leq \|S_i^{\oplus} - \Tilde{S}_i\|\|S_i^{\oplus} + \Tilde{S}_i\| \leq n\Theta_{m,p}\Theta_{0,p}=O(nm^{-A})=o(n^{2/p}/\log n). \]
Therefore,
\begin{equation} \label{eq:finalapprox}
\max_{1 \leq i \leq n}| \mathbb{B}(\|\Tilde{S}_i\|^2) - \mathbb{B}(\|{S}^{\oplus}_i\|^2) | = o_{\IP}(n^{1/p}),
\end{equation}
\color{black}
which completes the proof of \eqref{eq:GA} in view of Propositions \ref{prop1}, \ref{prop:prop2}, \ref{prop3}, and equations \eqref{eq:condn1}, \eqref{3.30}, \eqref{3.31}, \eqref{3.33}, \eqref{3.34}, \eqref{eq:tildeapprox} and \eqref{eq:finalapprox}.
\subsubsection{Connecting $\|S_i^{\oplus}\|^2$ to $\|S_i\|^2$}
The crux of the proof in this section relies on the following fundamental lemma connecting the variance of the truncated process to that in terms of the original process.
\begin{lemma}\label{lem:trunc-originalvariance}
  Let $T_{n^{1/p}}(X_i)$ and $S_i^{\oplus}$ be defined as in \eqref{trunc}. Also assume Conditions \ref{cond:fdm3} and \ref{cond:ui} for the process $(X_t)_{t \geq 1}$. Then, 
  \begin{equation}\label{trunc-originalvariance}
    \max_{1 \leq i \leq n}|\IE(S_i^2)-\IE( (S_i^{\oplus})^2)|=o(n^{2/p}).
  \end{equation}
  \end{lemma}
  \begin{proof}
    In view of our truncated uniform integrability Condition \ref{cond:ui}, one obtains,
    \begin{equation}
      \max_{1 \leq i \leq n}(\sum_{t=1}^i \IE[X_t^{\oplus}])^2 =n^2 (n^{1/p -1})^2 o(1) = o(n^{2/p}).
    \end{equation}
    Thus, it is enough to show that $\max_{1 \leq i \leq n}|\IE(S_i^2 - (\sum_{t=1}^i X_t^{\oplus})^2)|=o(n^{2/p})$. Writing $\IE(S_i^2 - (\sum_{t=1}^i X_t^{\oplus})^2)=\sum_{s=1}^i\sum_{t=1}^i \IE[X_sX_t -X_s^{\oplus}X_t^{\oplus}]$, observe that 
 \begin{equation} \label{maindecomp}
 |\IE(X_sX_t)- \IE(X_s^{\oplus}X_t^{\oplus})| \leq |\IE[(X_t-X_t^{\oplus})X_s]|+ |\IE[X_t^{\oplus}(X_s - X_s^{\oplus})]|:= I + II. 
 \end{equation}
Since \eqref{maindecomp} is symmetric in $s$ and $t$, we can assume without loss of generality that $s\geq t$. Recall the causal representation \eqref{eq:model_Z} for $X_s$. Denote by $X_{s, \{t\}}=g_s(\varepsilon_s, \varepsilon_{s-1}, \ldots, \varepsilon_{t+1}, \varepsilon'_{t}, \varepsilon'_{t-1}, \ldots)$, where $\varepsilon'_l, \varepsilon_i$ are i.i.d. for all $l, i \in \Z$. Such coupling has also been used in \cite{CUNY2018233} to obtain weaker conditions for \cite{kmt}'s result. Since $X_{s, \{t\}}$ is independent of $X_t$, hence for the term $I$ in \eqref{maindecomp}, H\"older's inequality along with Condition \ref{cond:ui} yields,
 \begin{align}
   \IE[(X_t-X_t^{\oplus})X_s] = \IE[(X_t-X_t^{\oplus})(X_s -X_{s, \{t\}})] &\leq \| X_t - X_t^{\oplus} \|_{\frac{p}{p-1}} \|X_s -X_{s, \{t\}} \|_p \nonumber\\&\leq C n^{2/p-1}\Theta_{s-t,p}o(1),\label{term1_new}
 \end{align}
 where the last inequality follows from an application of Theorem 1(iii) of \cite{Wu2005}. Note that here and also in the subsequent equations, the $o(1)$ term is independent of $s$, $t$ and $i$, since our Condition \ref{cond:ui} is defined uniformly for all $t \geq 1$.

 Now we will tackle the term $II$ in \eqref{maindecomp}. For simplicity, let us denote the remainder term $X_s-X_s^{\oplus}$ by $r_{n^{1/p}}(X_s)$. Again via H\"older inequality we obtain,
 \begin{align}
  | \IE[X_t^{\oplus}(X_s - X_s^{\oplus})]| &= | \IE[X_t^{\oplus}(r_{n^{1/p}}(X_s) - r_{n^{1/p}}(X_{s, \{t\}}))] | \nonumber \\ &\leq \|X_t^{\oplus} \|_p \|r_{n^{1/p}}(X_s) - r_{n^{1/p}}(X_{s, \{t\}})\|_{\frac{p}{p-1}}. \label{term2_new}
 \end{align}
 Now, in the light of \text{H\"older's inequality} and \text{Condition \ref{cond:ui}}, we have,
 \begin{align}
   \IE[|r_{n^{1/p}}(X_s) - r_{n^{1/p}}(X_{s, \{t\}})|^{\frac{p}{p-1}}] & \leq \IE\bigg[|X_s -X_{s, \{t\}}|^{\frac{p}{p-1}}\bigg(\mathbb{I}\{|X_s|\geq n^{1/p}\}+ \nonumber \\ & \hspace*{5cm}\mathbb{I}\{|X_{s, \{t\}}|\geq n^{1/p}\} \bigg)\bigg] \nonumber\\ \hspace*{2cm}
   & \leq \left(\IE[\|X_s -X_{s, \{t\}}\|^p]\right)^{\frac{1}{p-1}} \ \|\mathbb{I}\{|X_s|\geq n^{1/p}\} \nonumber\\ & \hspace*{5cm}+\mathbb{I}\{|X_{s, \{t\}}|\geq n^{1/p}\} \|_{\frac{p-1}{p-2}} \nonumber\\
   & \leq C \Theta_{s-t,p}^{\frac{p}{p-1}}n^{\frac{2-p}{p-1}} o(1). \label{term3_new}
 \end{align}
 Therefore, using $\sup_{t\geq 1} \|X_t^{\oplus}\|_p \leq \Theta_{0,p}$, \eqref{maindecomp}, \eqref{term1_new} and \eqref{term2_new} in conjunction with \eqref{term3_new} yields, for a fixed $1\leq i \leq n$,
 \begin{align*}
  \sum_{s=1}^i \sum_{t=1}^i |\IE(X_sX_t)- \IE(X_s^{\oplus}X_t^{\oplus})|\leq C o(1)\Theta_{0,p}\sum_{s=1}^i\sum_{t=1}^i \Theta_{|s-t|, p} n^{2/p-1}.
 \end{align*}
 This, in view of $\Theta_{i,p}=O(i^{-A})$ for $A>1$, immediately implies that,
 \[ \max_{1 \leq i \leq n}|\IE(S_i^2)-\IE((S_i^{\oplus})^2)| \leq C \Theta_{0,p}O(1) n^{2/p}o(1)=o(n^{2/p}),\]
which completes the proof of the lemma.\end{proof}

\noindent The proof of \eqref{eq:improvedGA} is immediate using Theorem \ref{thm:GA}, Lemma \ref{lem:trunc-originalvariance} and increment property of Brownian motion. 
 \end{proof}

\color{black}
\subsection{Proof of Theorem \ref{corol:new}}
Before we state the proof of Theorem \ref{corol:new}, we state and prove a couple of lemmas which are heavily used in the subsequent proofs. 

Lemma \ref{lemmaY} concerns approximating a Gaussian process $(Y_t)_{t \geq 1}$ by its orthogonal projections on sum of consecutive blocks of size $m$. On the other hand, given a Brownian motion $\IB(\cdot)$ and a process $(X_t)_{t\geq 1}$, Lemma \ref{lemmaYconstruct} constructs a Gaussian process $(Y_t)_{t\geq 1}$, such that the $(Y_t)_{t \geq 1}$ has the same covariance structure as $(X_t)_{t \geq 1}$, and partial sums of $(Y_t)_{t \geq 1}$ are approximated by the Brownian motion $\IB(\cdot)$ computed at variances of certain projections of $(X_t)$. 

\begin{lemma} \label{lemmaY}
   Let $(X_t)_{t=1}^n$ satisfy Conditions \ref{cond:fdm3}, \ref{cond:ui} and \ref{cond:regularity_new}; consider $m=m_n\in \N$ satisfying $m/n \to 0$ and $m \to \infty$ as $n \to \infty$. Further let $(Y_t)$ be a mean-zero Gaussian process with $\operatorname{Cov}(Y_s, Y_t)=\operatorname{Cov}(X_s, X_t)$. Denote $S_i^Y = \sum_{j=1}^i Y_j$ with $S_0^Y:=0$. Let $\Xi_k= S_{(km) \wedge n}^Y -S_{(k-1)m}^Y$ be the block sums,  $1 \leq k \leq \lceil n/m \rceil$, and let the $k$-th order projection be defined as
   \[ \xi_k= \IE[S_n^Y | \Xi_k, \ldots, \Xi_1]- \IE[S_n^Y|\Xi_{k-1}, \ldots, \Xi_1], k \geq 2, \text{ and $\xi_1:=\IE[S_n^Y |\Xi_1]$.}\] Then it holds that 
   \begin{align}\label{lemmaY1} 
     \max_{1 \leq i \leq n}|S_i^Y - \sum_{k=1}^{\lceil i/m \rceil} \xi_k |=O_{\IP}(\sqrt{m\log n}).
   \end{align}
Further, recall $B_k$ from \eqref{blockdefn}, defined using $m$ considered in this lemma. Let $\mathcal{q}_k=\IE[\xi_k^2]$, $1 \leq k \leq \lceil n/m \rceil$. Then, it holds uniformly in $k$ that,
\begin{align}\label{matrixvar}
  \mathcal{q}_k=u_{k,k}^{}+ 2 u_{k+1, k}^{} + \boldsymbol{u}_{k}^{^T} U_{k, k}^{-1}\boldsymbol{u}_{k}^{} - \boldsymbol{u}_{k-1}^{^T} U_{k-1,k-1}^{-1} \boldsymbol{u}_{k-1}^{} + O(m^{1-A}), \mbox{ where } U_{k,k}= \operatorname{Var}\begin{pmatrix}
   B_1 \\ B_2 \\ \vdots\\ B_k
 \end{pmatrix}
\end{align}
and $u_{k,l} = \IE[B_k B_l],$ $\boldsymbol{u}_k =(u_{k+1,1}, \ldots, u_{k+1, k})^T$ with $\boldsymbol{u}_0=\boldsymbol{u}_{\lceil n/m \rceil}=\boldsymbol{0}$. 
 \end{lemma} 
\begin{proof}
We will require repeated use of results from conditional mean of normal distributions. For $1 \leq k \leq \lceil n/m \rceil$,  we can write
\begin{equation}\label{condnexp}
\operatorname{Var}\begin{pmatrix}
   B_1 \\ B_2 \\ \vdots\\ B_k \\ B_{k+1} + \ldots + B_{\lceil n/m \rceil}
 \end{pmatrix}= \begin{pmatrix}
   U_{k,k} & \boldsymbol{u}_k + \boldsymbol{r}_k\\
   \boldsymbol{u}_k^T + \boldsymbol{r}_k^T & \| B_{k+1} + \ldots + B_{\lceil n/m \rceil}\|^2
 \end{pmatrix},
 \end{equation}
 where $\boldsymbol{r}_k=\sum_{i=k+2}^{\lceil n/ m \rceil} (u_{i,1}, \ldots , u_{i, k})^T$, $k<\lceil n/m\rceil-1$ with $\boldsymbol{r}_{\lceil n/m\rceil-1}=\boldsymbol{r}_{\lceil n/m\rceil}=\boldsymbol{0}$. Observe that, by Lemma \ref{lemmacp}, it holds uniformly in $k$ that
 \begin{align}
   |\sum_{i:|i-k|\geq 2}u_{i,k}|&=O(m^{1-A}), \
   |u_{k+1,k}|= O(1),  \text{ and } 
   |\boldsymbol{r}_k^T \boldsymbol{u}_k|= O(m^{1-A}) \label{rateu}.
 \end{align}\color{black}
 Define the additional remainder part of the partial sum: $\boldsymbol{R}_{(km) \wedge n}=0$ for $1 \leq k \leq \lceil n/m \rceil$, and $\boldsymbol{R}_i=\sum_{j=i+1}^{(m\lceil i/m \rceil) \wedge n} Y_j$ for $1 \leq i < n$ such that $i/m \notin \N$. Note that, by the telescopic nature of the definition of $\xi_k$, we have
    \[ \sum_{k=1}^{\lceil i/m \rceil} \xi_k=\IE[S_n^Y|\Xi_{\lceil i/m \rceil}, \ldots, \Xi_1] = \sum_{k=1}^{\lceil i/m \rceil} \Xi_k +  \IE\bigg[\sum_{k=\lceil i/m \rceil+1}^{\lceil n/m \rceil} \Xi_k| \Xi_{\lceil i/m \rceil}, \ldots, \Xi_1\bigg].\]
Moreover, define the $k$-dimensional column vector $\Gamma_k = (\Xi_1, \ldots, \Xi_k)^T$, then
    \[ \IE\bigg[\sum_{k=\lceil i/m \rceil+1}^{\lceil n/m \rceil} \Xi_k| \Xi_{\lceil i/m \rceil}, \ldots, \Xi_1\bigg]=(\boldsymbol{u}_{\lceil i/m \rceil}^T + \boldsymbol{r}_{\lceil i/m \rceil}^T) U_{\lceil i/m \rceil, \lceil i/m \rceil}^{-1} \Gamma_{\lceil i/m \rceil}
    \]
    follows from noting that the vector $(\Xi_{1}, \ldots, \Xi_{\lceil i/m \rceil}, \sum_{k=\lceil i/m \rceil+1}^{\lceil n/m \rceil} \Xi_k)^T$ follows a  $\lceil i/m \rceil+1$ dimensional centered multivariate normal distribution with covariance matrix given by \eqref{condnexp} (Theorem 3.2.4. of \cite{mardiamva}). 
Note that we can write $S_i^Y=\sum_{k=1}^{\lceil i/m \rceil} \Xi_k - \boldsymbol{R}_i$ for all $1\leq i\leq n$. Therefore, using \[ \operatorname{Var}\big(\boldsymbol{u}_k^T U_{k,k}^{-1}  \Gamma_{k} \big) 
    = \boldsymbol{u}_k^T U_{k,k}^{-1}\boldsymbol{u}_k,\]
    we have
\begin{align} \label{decompu}
  \|\sum_{k=1}^{\lceil i/m \rceil} \xi_k - S_i^Y \|^2 \leq 2\IE[\boldsymbol{R}_i^2] + 2(\boldsymbol{u}_{\lceil i/m \rceil}^T + \boldsymbol{r}_{\lceil i/m \rceil}^T) U_{\lceil i/m \rceil, \lceil i/m \rceil}^{-1}(\boldsymbol{u}_{\lceil i/m \rceil} + \boldsymbol{r}_{\lceil i/m \rceil}).
\end{align}
 Observe that, there exists constant $c>0$ such that uniformly in $k$,
\begin{equation}\label{eigencalc}
\lambda_{\min}(U_{k,k}) \geq \min_{1 \leq l \leq k}[\operatorname{Var}(B_l)- \sum_{j \neq l, 1\leq j \leq k} |\IE(B_l B_j)|] \geq \min_{1 \leq l \leq k}\operatorname{Var}(B_l) - c \to \infty, 
\end{equation}
where the first inequality is by Gershgorin circle theorem and the second inequality follows by invoking Lemma \ref{lemmacp} and noting that $\sum_{j,l, |j-l|=1}^{k}|\IE[B_l B_j]|=O(1)$, and $\sum_{j,l: |j-l|\geq 2}^{k}|\IE[B_l B_j]|=O(m^{1-A})$. The limit assertion is due to Condition \ref{cond:regularity_new}. \color{black} Finally, via \eqref{rateu}, we note that
\[\max_{1 \leq k \leq \lceil n/m \rceil}(\boldsymbol{u}_k + \boldsymbol{r}_k)^T(\boldsymbol{u}_k + \boldsymbol{r}_k) = O(1).\] Thus, \eqref{eigencalc} implies 
\begin{align}\label{U_calc}
&(\boldsymbol{u}_{\lceil i/m \rceil}+\boldsymbol{r}_{\lceil i/m \rceil})^T U_{\lceil i/m \rceil, \lceil i/m \rceil}^{-1}(\boldsymbol{u}_{\lceil i/m \rceil}+\boldsymbol{r}_{\lceil i/m \rceil})\nonumber\\\leq &\lambda_{\max}(U_{\lceil i/m \rceil, \lceil i/m \rceil}^{-1})(\boldsymbol{u}_{\lceil i/m \rceil} + \boldsymbol{r}_{\lceil i/m \rceil})^T(\boldsymbol{u}_{\lceil i/m \rceil} + \boldsymbol{r}_{\lceil i/m \rceil}) = O(1). 
\end{align}
This directly yields, in view of \eqref{decompu}, \eqref{U_calc} and $ \max_{1\leq i\leq n}\IE[\boldsymbol{R}_i^{2}] \leq C m \Theta_{0,2}^2$,
\begin{align}\label{finalunblocking}
 \max_{1\leq i \leq n}\| \sum_{k=1}^{\lceil i/m \rceil} \xi_k - S_i^Y\|^2 =O(m),
\end{align}
which implies \eqref{lemmaY1} via the elementary inequality $\IP(W>t)\leq 2e^{-t^2/2\sigma^2}$ for $W \sim N(0, \sigma^2)$. For proving \eqref{matrixvar}, \eqref{condnexp} implies for $1 \leq k \leq \lceil n/m \rceil$,
 \begin{align} 
  \xi_k&=\Xi_k + (\boldsymbol{u}_k^T + \boldsymbol{r}_k^T) U_{k,k}^{-1} \Gamma_{k}
   - (\boldsymbol{u}_{k-1}^T + \boldsymbol{r}_{k-1}^T) U_{k-1,k-1}^{-1}   \Gamma_{k-1}
  \nonumber \\
  &=  \big( (\boldsymbol{u}_k^T+\boldsymbol{r}_k^T) U_{k,k}^{-1} + (- (\boldsymbol{u}_{k-1}^T + \boldsymbol{r}_{k-1}^T) U_{k-1,k-1}^{-1} \ | \ 1) \big)  \Gamma_{k}, \label{xikcalc}
 \end{align}
where $(- (\boldsymbol{u}_{k-1}^T + \boldsymbol{r}_{k-1}^T) U_{k-1,k-1}^{-1} \ | \ 1)$ is a $k$-dimensional row vector with last entry $1$. Observe that \allowdisplaybreaks
\begin{align*}
  \operatorname{Cov}\big(\Xi_k, \ \boldsymbol{u}_{k-1}^T U_{k-1, k-1}^{-1} \Gamma_{k-1}
    \big) = \boldsymbol{u}_{k-1}^T U_{k-1, k-1}^{-1} \boldsymbol{u}_{k-1}, &\text{ and }
  U_{k,k}=\begin{pmatrix}
      U_{k-1,k-1} & \boldsymbol{u}_{k-1}\\
      \boldsymbol{u}_{k-1}^T & u_{k,k}
  \end{pmatrix}.
\end{align*}
 Therefore, in light of \eqref{rateu} and \eqref{eigencalc}, \eqref{xikcalc} directly yields
 \begin{align} \label{finalU}
   \mathcal{q}_k= u_{k,k} + \boldsymbol{u}_{k}^T U_{k,k}^{-1} \boldsymbol{u}_{k} - \boldsymbol{u}_{k-1}^T U_{k-1,k-1}^{-1}\boldsymbol{u}_{k-1} + 2 (-\boldsymbol{u}_{k-1}^T U_{k-1,k-1}^{-1} \ | \ 1 ) \boldsymbol{u}_{k} + O(m^{1-A}) .
 \end{align}
Write $\boldsymbol{u}_{k}= (\boldsymbol{s}_{k}^T \ | \ u_{k+1, k})^T$, where $\boldsymbol{s}_{k}=(u_{k+1,1}, \ldots, u_{k+1, k-1})^T$. Then similar to \eqref{rateu}, $|\boldsymbol{u}_{k-1}^T\boldsymbol{s}_k|=O(m^{1-A})$ holds uniformly in $k$. So, in light of \eqref{eigencalc}, \eqref{finalU} can be re-written as \eqref{matrixvar}, which completes the proof.
\color{black}
\end{proof}

Before moving onto the construction of $Y^c$-processes, we need to introduce a few notation. For $w_1, \ldots, w_n \overset{\text{i.i.d.}}{\sim} N(0,1)$, define
\begin{equation}\label{defnY}
  Y_t^{w}:= Y_t(w_1, \ldots, w_t)= \|X_t\| \bigg(\sum_{i=1}^t x_i^{(t)} w_i \bigg),
\end{equation}
where $x_1^{(1)}=1$, and for $i\leq t$, $x_i^{(t)}$ is obtained by solving the equation system $\sum_{k=1}^i x_k^{(i)} x_k^{(t)}=\operatorname{Corr}(X_i, X_t)$. Observe that by construction, $\operatorname{Cov}(Y_s^w, Y_t^w)=\operatorname{Cov}(X_s, X_t)$ for all $1\leq s,t \leq n$. \color{black} For $m = m_n \to \infty$ with $m/n \to 0$, let us define $$\xi_k^w:=\IE\bigg[\sum_{i=1}^n Y_i^w \bigg| \bigg(\sum_{i=(j-1)m+1}^{(jm) \wedge n}Y_i^w\bigg)_{j=1}^k \bigg] - \IE\bigg[\sum_{i=1}^n Y_i^w \bigg| \bigg(\sum_{i=(j-1)m+1}^{(jm) \wedge n}Y_i^w\bigg)_{j=1}^{k-1} \bigg],$$ for $2\leq k \leq \lceil n/m \rceil$ with $\xi_1^w:=\IE[\sum_{i=1}^n Y_i^w | \sum_{i=1}^{m}Y_i^w]$. Further define real-valued $\alpha_i^{(k)}, 1\leq k \leq \lceil n/m \rceil, 1 \leq i \leq km,$ such that it holds
\begin{equation} \label{alphaeps}
   \sum_{j=1}^k \xi_j^w= \sum_{i=1}^{(km) \wedge n} \alpha_i^{(k)} w_i, \text{ for $1 \leq k \leq \lceil n/m \rceil$.}
\end{equation} 
Note that the sequence $(\alpha_i^{(k)})$ satisfying \eqref{alphaeps} exists, since, by property of conditional expectation of multivariate normal distributions, $\xi_k^{w}$ can be written as linear combinations of $Y_i^w$'s which, in turn, are linear combinations of $w_i$'s themselves. Recall the definition of $\mathcal{q}_k$ from Lemma \ref{lemmaY}. Observe that due to property of projection, $\xi_k$'s are uncorrelated and thus independent as well since they are jointly normally distributed. Moreover, $\IE[\xi_k]=0$. Therefore, by definition of $\mathcal{q}_k$ and \eqref{alphaeps}, it follows that 
\begin{align*}
    \sum_{j=1}^k \mathcal{q}_j=\operatorname{Var}(\sum_{j=1}^k \xi_j)&= \operatorname{Var}(\IE\bigg[\sum_{i=1}^n Y_i^w \bigg| \bigg(\sum_{i=(j-1)m+1}^{(jm) \wedge n} Y_i^w\bigg)_{j=1}^k \bigg])\\ &=\operatorname{Var}(\sum_{i=1}^{(km) \wedge n}\alpha_i^{(k)} w_i)=\sum_{i=1}^{(km) \wedge n}(\alpha_i^{(k)})^2.
\end{align*}
We emphasize that at the present level of construction, $(Y_t^c)_{t=1}^n$ are not defined. However, what is known is their distribution, and therefore $(\mathcal{q}_k)$ and the sequence $\alpha_i^{(k)}$ are also known. Now $Y_t^c$ will be defined through this quantities in the following lemma. 
 \begin{lemma}\label{lemmaYconstruct}
 Let $\IB(\cdot)$ be a Brownian motion, and $(X_i)_{i=1}^n$ be a stochastic process satisfying Conditions \ref{cond:fdm3}, \ref{cond:ui} and \ref{cond:regularity_new}. Let $\mathcal{q}_k$ be defined as in Lemma \ref{lemmaY} with some $m \in \N$ such that $m\to \infty$, $m/n \to 0$ as $n \to \infty$. Consider the following algorithm of constructing a new Gaussian process $Y_t^c$:
 \begin{itemize}
   \item For $1 \leq k \leq \lceil n/m \rceil$, write $\IB(\sum_{j=1}^k \mathcal{q}_j):=\sum_{j=1}^{km} \alpha_i^{(k)} \eta_i$, where $\eta_{1}, \ldots, \eta_n$ are i.i.d. $N(0,1)$. 
 \item For $1\leq i \leq n$, define $Y_i^c:= Y_i(\eta_1, \ldots, \eta_i)$ as in \eqref{defnY}.
\end{itemize}
Then it holds that 
\begin{itemize}
  \item For $1 \leq i \leq n$, $Y_i^c \sim N(0, \operatorname{Var}(X_t))$, and for $1 \leq i\neq j \leq n$, $\operatorname{Cov}(Y_i^c, Y_j^c)=\operatorname{Cov}(X_i, X_j)$. 
  \item $\max_{1 \leq i \leq n} |\IB(\sum_{j=1}^{\lceil i/m 
  \rceil} \mathcal{q}_j) - \sum_{j=1}^i Y_j^c|= O_{\IP}(\sqrt{m\log n}). $
\end{itemize}

 \end{lemma}
 \begin{proof}
 The first assertion can be verified directly from our construction. For the second assertion, let $\Xi_l^c$ and $\xi_l^c$ be obtained from $Y_t^c$ as in Lemma \ref{lemmaY}. Then as in \eqref{alphaeps}, $\mathbb{B}(\sum_{l=1}^{\lceil i/m \rceil} \mathcal{q}_l)$ can also be represented as $\sum_{l=1}^{\lceil i/m \rceil} \xi_l^c$. Then the result follows from Lemma \ref{lemmaY}.
 \end{proof}

 \begin{proof}[Proof of Theorem \ref{corol:new}]
   Consider the Brownian motion $\IB(\cdot)$ from the conclusion of Theorem \ref{thm:GA}. \color{black} We will use the same $m$ as in the proof of Theorem \ref{thm:GA}. Let $\mathcal{q}_k$ be defined as in Lemma \ref{lemmaY} with the original process $(X_t)_{t=1}^n$. Denote $\boldsymbol{D}_i:=\sum_{j=m\lfloor i/m \rfloor+1}^{i} X_j$ if $i/m \notin \N$, and $\boldsymbol{D}_i=0$ otherwise. Observe that,
   \[ \IE(S_i^2) = \sum_{k=1}^{\lfloor i/m \rfloor} u_{k,k} + 2 \sum_{k=1}^{{\lfloor i/m \rfloor}-1} u_{k, k+1} + \IE[\boldsymbol{D}_i^2] + 2 \IE[B_{\lfloor i/m \rfloor}\boldsymbol{D}_i] + \Psi_i, \]
   where $\Psi_i:= 2\sum_{k,l \leq \lfloor i/m \rfloor: |k-l|\geq 2 } \IE[B_kB_l] + 2\sum_{k=1}^{\lfloor i/m \rfloor-1}\IE[B_k \boldsymbol{D}_i]$ is the sum of all the higher order cross-products.
   Observe that by Lemma \ref{lemmacp}, 
   \[ |\Psi_i| \leq 2\lceil \frac{n}{m} \rceil \max_{1 \leq k \leq \lceil n/m \rceil} \left| \sum_{l: |k-l| \geq 2} \IE[B_lB_k]\right|=O(nm^{-A}) \text{ uniformly in $i$}. \]
   On the other hand, from \eqref{matrixvar} in Lemma \ref{lemmaY}, it follows that,
   \[\sum_{j=1}^{\lceil i/m \rceil}\mathcal{q}_j = \sum_{k=1}^{\lceil i/m \rceil} u_{k,k} + 2 \sum_{k=1}^{\lceil i/m \rceil} u_{k,k+1}+ \boldsymbol{u}_{\lceil i/m \rceil}^T U_{\lceil i/m \rceil, \lceil i/m \rceil}^{-1}\boldsymbol{u}_{\lceil i/m \rceil}+ O(nm^{-A}) \text{ uniformly in $i$}.\]
  Therefore, an argument similar to \eqref{U_calc} along with an application of Lemma \ref{lemmacp} and Lemma \ref{pmoment} yields that
 \begin{align}\label{8.65} 
 &\max_{1 \leq i \leq n} |\sum_{j=1}^{\lceil i/m \rceil}\mathcal{q}_{j} - \IE(S_i^2) |\nonumber\\ \leq &\max_{1\leq i \leq n}(u_{\lceil i/m \rceil, \lceil i/m \rceil} + \IE[D_i^2] + 2 |u_{\lceil i/m \rceil, \lceil i/m \rceil+1}| + 2|u_{\lceil i/m \rceil-1, \lceil i/m \rceil+1}|+ 2 |\IE[B_{\lfloor i/m \rfloor}\boldsymbol{D}_i] |  \nonumber \\ & \hspace*{2cm}+ |\boldsymbol{u}_{\lceil i/m \rceil}^T U_{\lceil i/m \rceil, \lceil i/m \rceil}^{-1}\boldsymbol{u}_{\lceil i/m \rceil}| ) + O(nm^{-A})\nonumber \\ =&O(m+ nm^{-A})=o(n^{(\alpha/p-1)/(\alpha/2-1)}+ n^{2/p}/\log n )=o(n^{2/p}/\log n),\end{align}
where the second equality is due to \eqref{eq2n} and \eqref{eq1n} respectively. This in turn implies, by the increment of the Brownian motion, that
\begin{align*}
  \max_{1 \leq i \leq n} |\IB(\IE(S_i^2)) - \IB(\sum_{j=1}^{\lceil i/m \rceil} \mathcal{q}_j) |=o_{\IP}(n^{1/p}).
\end{align*}
\color{black}
Lemma \ref{lemmaYconstruct} and Theorem \ref{thm:GA}, along with another application of \eqref{eq2n} complete the proof of \eqref{corol_new}.

For \eqref{corol_new_trunc}, define $\mathcal{q}_k^{\oplus}$ based on $X_t^{\oplus}$. Invoking \eqref{ineqfdm}, Lemma \ref{lemmaY} and \ref{lemmaYconstruct} hold when the original process is $X_t^{\oplus}$. Therefore, following exactly the above argument, along with Theorem \ref{thm:GA} completes the proof. 
 \end{proof}
\color{black}
\section{Appendix B: Proofs of Theorems \ref{thm:multGA} and \ref{thm:GA2}}\label{sec2proofsnew}
Much like Theorems \ref{thm:GA} and \ref{corol:new}, we will prove Theorem \ref{thm:GA2} first, and the proof of Theorem \ref{thm:multGA} will follow from it. 
\subsection{Proof of Theorem \ref{thm:GA2}}
Recall $A_0'$ from \eqref{eq:Alowerbound2}. Define, in the light of the form of $\Theta_{i,p}$ in \eqref{eq:Thetaip} with $A'>A_{0}'$,
\begin{align}
  L'&=\frac{f_1-f_2 + A'\sqrt{(p-2)(f_3 -3p)}}{A'f_4} \label{L'},\\
  \alpha' &= \frac{(2p + p^2)A' + p^2+ 3p+2+\sqrt{f_5}}{2+2p+4A'}, \nonumber
\end{align}
with $f_1(p,A)=Ap^2(A+1)$, $f_2(p,A)=A(2pA+3p-2)$, $f_3(p,A)=p^3(1+A)^2+6f_1+4pA-2$, $f_4(p,A)=2p(2pA^2 + 3pA + p-2)$ and $f_5(p,A)=p^2(p^2+4p-12)A^2 + 2p(p^3 + p^2 -4p-4)A + (p^2-p-2)^2$. Our choice of $L'$ and $\alpha'$ satisfies the following relations which we use in our proof:
\begin{align}
  &\frac{1}{2}-\frac{1}{p} - L'A' < 0, \label{eq1n_ver2}\\
  & L'\left(\frac{\alpha'}{2}-1\right)+ 1 - \frac{\alpha'}{p} <0, \label{eq2n_ver2} \\
  &p<\alpha' < 2(1+p+pA')/3, \label{eq3n_ver2}\\
  & 1/p - 1/\alpha' + L' - L'(A'+1)p/\alpha' =0. \label{eq4n_ver2}
\end{align}
Note that with this new $L'$ and $\alpha'$ along with choosing a new $m=\lfloor n^{L'} \rfloor$, proof of Lemma \ref{lem:truncation} and Lemma \ref{Rosenthal} goes through. 
We will also use the following lemma in a crucial step of our proof. 
\begin{lemma}\label{lemma6.3}
Under the assumption of Theorem \ref{thm:GA2}, 
\begin{equation} \label{kmtassumption}
  \min_{l\geq 1} \{\Theta_{l,p} + ln^{2/p-1}\}=o({n^{1/p-1/2}} ). 
\end{equation}
\end{lemma}
\begin{proof}
Let $A'> 1$. Define $B:=(A'+1)/2$ and choose $l_0=n^{(1/2-1/p)/B}$. In light of $1 <B < A'$,
\begin{align*}
 \min_{l\geq 0} \{\Theta_{l,p} + ln^{2/p-1}\} \leq \Theta_{l_0,p} + l_0n^{2/p-1}&=(n^{1/2-1/p})^{-A'/B} +(n^{1/2-1/p})^{1/B-2}=o({n^{1/p-1/2}}). 
\end{align*}
\end{proof}
\begin{proof}[Proof of Theorem \ref{thm:GA2}]
For the proof of Theorem \ref{thm:GA2}, we proceed exactly as in Theorem \ref{thm:GA}, with $L$ and $\alpha$ therein replaced by $L'$ and $\alpha'$, and equations \eqref{eq1n}-\eqref{eq4n} replaced by equations \eqref{eq1n_ver2}-\eqref{eq4n_ver2}. The arguments up until equation \eqref{3.34} go through verbatim with $m=\lfloor n^{L'} \rfloor$. Thus the only part which requires our attention is the \textit{approximation of variance} step. 
\subsubsection{Variance regularization}
Based on equation \eqref{eq:tildeA}, define, for $1 \leq l \leq l_n$, 
\begin{align}
  v_l &:= \|B_{3l-1}\|^2 + \|B_{3l}\|^2 + \|B_{3l+1}\|^2 + 2 \IE[B_{3l-1}B_{3l}] + 2 \IE[B_{3l}B_{3l+1}] + 2 \IE[B_{3l+1}B_{3l+2}] \nonumber \\ & - \|C_{3l-1}(\boldsymbol{\eta}_{3l-2})\|^2 + \|C_{3l+2}(\boldsymbol{\eta}_{3l+1})\|^2 + 2\IE[B_1B_{3l} + \ldots +B_{3l-2}B_{3l}] \nonumber \\
  & + 2\IE[B_1B_{3l+1} + \ldots + B_{3l-1}B_{3l+1}] + 2 \IE[B_1 B_{3l+2} + \ldots + B_{3l}B_{3l+2}]. \label{v_l}
\end{align}

Let $B^{\oplus}_k=\sum_{j=(k-1)m+1}^{(km) \wedge n} (X_j^{\oplus} - \IE(X_j^{\oplus}))$, $1\leq k \leq \lceil n/m \rceil$. Note that uniformly over $k$, \[\|\Tilde{B}_k - {B}^{\oplus}_k \| \leq \sqrt{m} \Theta_{m,2} \leq \sqrt{m} \Theta_{m,p}.\] Let the projection operator $P_j$ be defined as $P_j
(\cdot)=\IE[\cdot | \mathcal{F}_j] - \IE[\cdot | \mathcal{F}_{j-1}]$. Using the Cauchy-Schwarz inequality, and ${B}^{\oplus}_k=\sum_{j=-\infty}^k P_j{B}^{\oplus}_k$, for $k \geq 1$,
\begin{equation} \label{sq_1}
  | \IE(\Tilde{B}_k^2)- \IE({B}^{\oplus^2}_k) | \leq \| \Tilde{B}_k - {B}^{\oplus}_k\| \| \Tilde{B}_k + {B}^{\oplus}_k\| \leq 
  4m \Theta_{m,p} \Theta_{0,p}.
\end{equation}
Similarly,
\begin{equation}\label{cr1}
  |\IE(\Tilde{B}_k\Tilde{B}_{k+1})- \IE({B}^{\oplus}_k{B}^{\oplus}_{k+1}) | \leq \|{B}^{\oplus}_k\|\| \Tilde{B}_{k+1} - {B}^{\oplus}_{k+1}\| + \|\Tilde{B}_{k+1}\| \| \Tilde{B}_{k}- {B}^{\oplus}_k\| \leq 4m \Theta_{m,p} \Theta_{0,p},
\end{equation}
for $1 \leq k \leq \lceil n/m \rceil$. Further, note that uniformly for all $k, l\geq 1$ using uniform integrability condition (\ref{cond:ui}), we obtain
\allowdisplaybreaks
\begin{align*}
  |\IE(X_{k}X_{l} - X^{\oplus}_{k}X^{\oplus}_{l})| 
  =& |\IE(X_{k}X_{l} \mathbb{I}_{\max\{|X_{k}|,|X_{l}|\}\leq n^{1/p}}) - \IE(X^{\oplus}_{k}X^{\oplus}_{l}) \\ & \hspace*{4cm} + \IE(X_{k}X_{l} \mathbb{I}_{\max\{|X_{k}|,|X_{l}|\}> n^{1/p}})|\\
  =& |-\IE(X^{\oplus}_{k}X^{\oplus}_{l} \mathbb{I}_{\max\{|X_{k}|,|X_{l}|\}> n^{1/p}}) + \IE(X_{k}X_{l} \mathbb{I}_{\max\{|X_{k}|,|X_{l}|\}> n^{1/p}})|\\
  \leq & |\IE(X^{\oplus}_{k}X^{\oplus}_{l} \mathbb{I}_{\max\{|X_{k}|,|X_{l}|\}> n^{1/p}}) | + | \IE(X_{k}X_{l} \mathbb{I}_{\max\{|X_{ k}|,|X_{l}|\}> n^{1/p}})|\\
  \leq & \IE\left((|X_{k}|^2 + |X_{l}|^2)\mathbb{I}_{\max\{|X_{ k}|,|X_{l}|\}> n^{1/p}} \right)
  = o(n^{2/p-1}).
\end{align*}
In view of \eqref{ineqfdm}, \eqref{eq:product} also holds for $X_k^{\oplus}X_l^{\oplus}$. Noting that Condition \ref{cond:ui} implies $\sup_{i} |\IE(X_i^{\oplus})|=o(n^{1/p-1})$, we have, for a fixed $0\leq j \leq \lceil n/m \rceil-1$ and $l \geq 0$, 
\begin{align*}
  |\IE(B_{j+1}^2 - (B_{j+1}^\oplus)^2)| &= \bigg|\sum_{k=1}^m \IE(X_{jm+k}^2-(X_{jm +k}^\oplus)^2) \\ & \hspace*{1cm}+ \sum_{s \neq t}^m \IE(X_{jm+s}X_{jm+t}- X_{jm+s}^{\oplus}X_{jm+t}^{\oplus}) - \left(\IE[\sum_{k=1}^{m} X_{jm+k}^{\oplus}]
\right)^2\bigg|\\& \leq o(mn^{2/p -1}) + O(lmn^{2/p -1} + m \sum_{s=l+1}^{\infty} \sum_{i=0}^{\infty} \delta_p(i) \delta_p(i + s)),
\end{align*}
where the last line follows from using the fact that there are $\leq m$ terms of the form $\IE(X_kX_{k+s}-X_k^{\oplus} X_{k+s}^{\oplus})$ for a fixed $s \leq m$ and applying \eqref{eq:product} in the proof of Lemma \ref{lemmacp}. Note that \[\sum_{j=l}^{\infty} \sum_{i=0}^{\infty} \delta_p(i) \delta_p(i + j) \leq \Theta_{0,p}\Theta_{l,p}.\] Hence, 
\begin{align}
  |\IE(B_j^2) - \IE( (B_j^\oplus)^2)| &= O(mn^{2/p-1} + m\min_{l \geq 0} \{ln^{2/p-1} + \Theta_{l+1,p} \})\nonumber \\ &= O(m\min_{l \geq 1} \{ln^{2/p-1} + \Theta_{l,p} \}).\label{sq2}
\end{align}
Similarly, 
\begin{equation}\label{cr2}
  |\IE(B_j B_{j+1}) - \IE(B_j^{\oplus} B_{j+1}^{\oplus}) |=O(m\min_{l \geq 1} \{ln^{2/p-1} + \Theta_{l,p} \}).
\end{equation}
Therefore, \eqref{sq_1}, \eqref{cr1}, \eqref{sq2} and \eqref{cr2} together with Lemma \ref{lemmacp} yields
\begin{equation} \label{v_lcomp}
  |\Tilde{v}_l - v_l| = O(m \Theta_{m,p} + m\min_{l \geq 1} \{ln^{2/p-1} + \Theta_{l,p} \} + m^{1-A'}) = O(m^{1-A'} + m\min_{l \geq 1} \{ln^{2/p-1} + \Theta_{l,p} \}).
\end{equation}
Applying Lemma \ref{lemma6.3} along with \eqref{v_lcomp} leads to the following assertion:
\begin{equation} \label{limit}
  \max_l \frac{|\Tilde{v}_l - |v_l||}{m} \leq \max_l |\Tilde{v}_l -v_l|/m \to 0 \text{ as } n \to \infty .
\end{equation}
Recall $l_n$ from \eqref{notation}; \eqref{eq:blockproduct} implies that 
\begin{align} \label{crossprodrem}
\max_{1 \leq l \leq l_n}&\bigg|\IE[B_{3l-1}B_{3l}] + \IE[B_{3l}B_{3l+1}] + \IE[B_{3l+1}B_{3l+2}] + \IE[B_1B_{3l} + \ldots +B_{3l-2}B_{3l}] \nonumber \\ & + \IE[B_1B_{3l+1} + \ldots + B_{3l-1}B_{3l+1}] + \IE[B_1 B_{3l+2} + \ldots + B_{3l}B_{3l+2}]\bigg|/m \to 0 \text{ as } n \to \infty.
\end{align}
Hence, for large enough $n$, it follows from the regularity Condition \ref{cond:regularity} along with \eqref{crossprodrem} that 
\begin{align}\label{large.n.var}
  \inf_{l} \frac{|v_l|}{3m} &\geq \inf_{l} \bigg( \| B_{3l-1}\|^2 +\|B_{3l}\|^2+\|B_{3l+1}\|^2+\|C_{3l+2}(\boldsymbol{\eta}_{3l+1})\|^2- \|C_{3l-1}(\boldsymbol{\eta}_{3l-1})\|^2 \nonumber \\& - 2 |\IE[B_{3l-1}B_{3l}] + \IE[B_{3l}B_{3l+1}] + \IE[B_{3l+1}B_{3l+2}] + \IE[B_1B_{3l} + \ldots +B_{3l-2}B_{3l}] \nonumber \\& + \IE[B_1B_{3l+1} + \ldots + B_{3l-1}B_{3l+1}]   + \IE[B_1 B_{3l+2} + \ldots + B_{3l}B_{3l+2}] |  \bigg)/(3m) >\frac{c}{3},
\end{align}
where we have used $\|C_{3l-1}(\boldsymbol{\eta}_{3l-1})\|^2 \leq \| B_{3l-1}\|^2$ by Jensen's inequality.
Observe that $\mathbb{B}(\sigma_n^2)$ can be represented as $\sum_{l=1}^{l_n} \sqrt{\Tilde{v}_l}Z_l^{*}$ for i.i.d. standard Gaussian random variables $Z_1^{*}, \ldots, Z_{l_n}^{*}$. We define the following Brownian motion:
\begin{align}\label{boldB_n}
  \mathbf{{B}}_n = \sum_{l=1}^{l_n}\sqrt{|v_l|}Z_l^{*}.
\end{align}
Let 
\begin{equation}
  \Psi_n^2 = \| \mathbb{B}(\sigma_n^2)) - \mathbf{B}_n \|^2= \sum_{l=1}^{l_n} (\sqrt{\Tilde{v}_l}- \sqrt{|v_l|})^2.
\end{equation}
Using \eqref{large.n.var} and \eqref{limit}, for large enough $n$ we have, 
\begin{equation}\label{large.n.var2}
  \inf_l \frac{ (\sqrt{\Tilde{v}_l}+ \sqrt{|v_l|})^2}{3m} \geq \inf_l\frac{\Tilde{v}_l+ |v_l|}{3m}\geq \inf_l \frac{2v_l - |\Tilde{v}_l - v_l|}{3m} > \frac{c}{3}.
\end{equation}
Therefore, from \eqref{v_lcomp} and \eqref{large.n.var2} one obtains for $1\leq l \leq l_n$, 
\begin{equation}\label{9.19new}
  \sup_{l}(\sqrt{\Tilde{v}_l}- \sqrt{|v_l|})^2=\frac{(\tilde{v}_l-|v_l|)^2/ 3m}{ (\sqrt{\Tilde{v}_l}+ \sqrt{|v_l|})^2/ 3m}= O\left(m^{1-2A'} + m\left(\min_{l \geq 1} \{ln^{2/p-1} + \Theta_{l,p} \}\right)^2\right).
\end{equation}
Thus, Lemma \ref{lemma6.3} together with \eqref{eq1n_ver2} implies
\begin{equation} \label{eq:psin}
  \Psi_n^2 =O\left(nm^{-2A'}+ n\left(\min_{l \geq 1} \{ln^{2/p-1} + \Theta_{l,p} \}\right)^2\right)=o(n^{2/p}).
\end{equation}
Note that ${\mathbb{B}}(\sigma_k^2) - \mathbf{B}_k$ is a Gaussian process with independent increments. Therefore, using Doob's maximal inequality we have
\begin{equation} \label{eq:Bn}
  \max_{1\leq i \leq n} |\mathbb{B}(\sigma_i^2) - \mathbf{B}_i | =o_{\IP}(n^{1/p}),
\end{equation}
in view of \eqref{eq:psin}.
Next, definition of $v_l$ in \eqref{v_l} along with \eqref{eq:blockproduct} implies 
\begin{equation}
  | \IE(S_i^2) - \sum_{l=1}^{l_i} |v_l| | \leq | \IE(S_i^2) - \sum_{l=1}^{l_i} v_l | =O(m) \text{ uniformly over $1\leq i \leq n$},
\end{equation}
which readily leads to 
\begin{equation} \label{finalstep}
  \max_{1 \leq i \leq n} |\mathbb{B}(\IE(S_i^2)) - \mathbf{B}_i | = o_{\IP}(n^{1/p}).
\end{equation}
This completes the proof of \eqref{result:GA_ver2} in view of Propositions \ref{prop1}, \ref{prop2}, \ref{prop3}, and equations \eqref{eq:condn1}, \eqref{3.30}, \eqref{3.31}, \eqref{3.33}, \eqref{3.34}, \eqref{eq:Bn} and \eqref{finalstep}. 
\color{black}
\end{proof}

\color{black}
\subsection{Proof of Theorem \ref{thm:multGA}}
\begin{proof}
Now, we will construct a Gaussian process $(Y_t^c)_{t=1}^n$ with the same covariance structure $(X_t)_{t=1}^n$ such that \eqref{eq:multGA_nocondn} holds. We will use the same $m$ as in the proof of Theorem \ref{thm:GA2}. Recall $\mathcal{q}_k$ as defined in Lemma \ref{lemmaY}, and $v_l$ from \eqref{v_l}. We will use the same notation as in the proof of Theorem \ref{corol:new}. First we define
\begin{align*}
    \mathbf{v}_l&= \|B_{3l-1}\|^2 + \|B_{3l}\|^2 + \|B_{3l+1}\|^2 + 2 \IE[B_{3l-1}B_{3l}] + 2 \IE[B_{3l}B_{3l+1}] + 2 \IE[B_{3l+1}B_{3l+2}] \nonumber \\ & - \|C_{3l-1}(\boldsymbol{\eta}_{3l-2})\|^2 + \|C_{3l+2}(\boldsymbol{\eta}_{3l+1})\|^2,\ 1 \leq l \leq l_n.
\end{align*}
 By the argument similar to \eqref{large.n.var} it follows that $\inf_l |\mathbf{v}_l|/m >c$ for all sufficiently large $n$. We define a new Brownian motion $\mathcal{H}_n:= \sum_{l=1}^{l_n} |\mathbf{v}_l|^{1/2} Z_l^{\star}$ with the same $Z_l^{*}$'s as in definition of $\mathbf{B}_n$ in \eqref{boldB_n}. It follows similar to \eqref{large.n.var2}-\eqref{eq:Bn} that 
\begin{equation}
  \| \mathbf{B}_n - \mathcal{H}_n \|^2= \sum_{l=1}^{l_n}(|v_l|^{1/2} - |\mathbf{v}_l|^{1/2})^2= O(nm^{-2A'})= o(n^{2/p}),
\end{equation}
which yields $\max_{1 \leq i \leq n} | \mathbf{B}_i - \mathcal{H}_i| =o_{\IP}(n^{1/p}).$ Next we define 
\[ \tau_k:=u_{k,k}^{}+ 2 u_{k+1, k}^{} + \boldsymbol{u}_{k}^{^T} U_{k, k}^{-1}\boldsymbol{u}_{k}^{} - \boldsymbol{u}_{k-1}^{^T} U_{k-1,k-1}^{-1} \boldsymbol{u}_{k-1}^{} \text{ for $1 \leq k \leq \lceil n/m \rceil$.} \]
 In light of Lemma \ref{lemmaY} and similar to the argument leading to \eqref{8.65}, it can be observed that 
 \begin{align} \label{newGn}
   \max_{1 \leq i \leq n}|\sum_{l=1}^{l_i} \mathbf{v}_{l} - \sum_{k=1}^{\lceil i/m \rceil }\tau_k|= O(m).
 \end{align}
 Observe that due to Condition \ref{cond:regularity}, $\min_{1 \leq l \leq l_n}\mathbf{v}_l>0$ and $\min_{1 \leq k\leq \lceil n/m \rceil} \tau_k>0$ for all sufficiently large $n$. This motivates us to define the following Gaussian process
 \[ \mathcal{G}_n:= \sum_{l=1}^{\lceil n/m \rceil} |\tau_l|^{1/2} Z_{l}^{*}. \]
In light of \eqref{newGn} and \eqref{eq2n_ver2}, we obtain 
\begin{equation} \label{newtauapprox}
  \max_{1 \leq i \leq n} | \mathcal{H}_i - \mathcal{G}_i| =O_{\IP}(\sqrt{m\log n})=o_{\IP}(n^{1/p}).
\end{equation}
Denote $$\mathcal{B}_n:=\sum_{k=1}^{\lceil n/m \rceil} \mathcal{q}_k^{1/2} Z_k^{*}.$$ 

\noindent We now show that $\mathcal{B}_n$ is a good enough approximation of $\mathcal{G}_n$. By definition of $\tau_l$ it holds, $\mathcal{q}_l=\tau_l+O(m^{1-A'})$ uniformly for $1 \leq l \leq \lceil n/m \rceil$; moreover, an argument similar to \eqref{large.n.var} yields that $\inf_l |\tau_l| /m >c>0$ for all sufficiently large $n$. Thus an application of the argument same as \eqref{large.n.var2} and \eqref{9.19new} implies that 
\begin{equation}
  \| \mathcal{B}_n - \mathcal{G}_n \|^2= \sum_{l=1}^{\lceil n/m \rceil}(\mathcal{q}_l^{1/2} - |\tau_l|^{1/2})^2= O(nm^{-2A'})= o(n^{2/p}),
\end{equation}
which in turn yields
\begin{equation}\label{newY}
  \max_{1 \leq k \leq \lceil n/m \rceil}| \mathcal{B}_k - \mathcal{G}_k|=o_{\IP}(n^{1/p}),
\end{equation}
using Doob's maximal inequality. 

We will construct our Gaussian approximation $(Y_t^c)_{t=1}^n$ from $\mathcal{B}_k$ exactly as in Lemma \ref{lemmaYconstruct}. In light of \eqref{newY}, \eqref{newtauapprox} and \eqref{eq:Bn}, Lemma \ref{lemmaYconstruct} completes the proof of \eqref{eq:multGA_nocondn}.
\end{proof}
\color{black}
\section{Appendix C: Proofs of Section \ref{ssc:estvar}} \label{sec:section3proofs}
In this section we provide the proofs of the two main results of the estimation step. Firstly, we prove the maximal quadratic deviation inequality in Theorem \ref{thm:maxpartialquad}. 
\subsection{Proof of Theorem \ref{thm:maxpartialquad}}\label{proofmaxquad}
In order to prove this theorem, we require the following lemmas. These results are well-known in the literature for $p>4$ case; however we weaken the condition on moments to allow $p>2$. For the sake of completeness we state and prove the results as we use it. In the following, we use $C_p$ to denote a generic positive constant depending solely on $p$, and $C$ to denote an universal constant. Their values are subject to change from line to line. \color{black}

Our first lemma is the (one-dimensional) general version of Lemma D.6 of Supplement to \cite{zhang&wu2021}. 
\begin{lemma}\label{lemma8.1}
   Assume that the process \eqref{eq:model_Z} has $\IE(X_t)=0$ and $\Theta_{0,p} < \infty$ for some $p >2$. Let $\mathcal{D}_n \leq n$ and $\boldsymbol{\eta}_k=(\varepsilon_{(k-1)\mathcal{D}_n+1}, \ldots, \varepsilon_{k\mathcal{D}_n})$. For $1 \leq k \leq \lceil n/\mathcal{D}_n \rceil$, define $V_k$ as in \eqref{Vj} and let $V_{k, h}=\IE(V_k | \boldsymbol{\eta}_{k-h}, \boldsymbol{\eta}_{k-h+1}, \ldots, \boldsymbol{\eta}_k)$. Then for $h \geq 2$, 
   \begin{equation} \label{eq8.1}
     \| V_{k,h} - V_{k,h-1} \|_{p/2} \leq 
     \begin{cases}
       C_p\mathcal{D}_n^{1/2 + 2/p} \Theta_{0,p} \sum_{d=(h-2)\mathcal{D}_n +1}^{(h+1)\mathcal{D}_n}\delta_p(d), & \ 2 < p \leq 4,\\
       C_p\mathcal{D}_n \Theta_{0,p} \sum_{d=(h-2)\mathcal{D}_n +1}^{(h+1)\mathcal{D}_n}\delta_p(d), & \ p>4.
     \end{cases}
   \end{equation}
\end{lemma}
\begin{proof}
For $a<b$, let $\mathcal{F}_a^b=(\varepsilon_a, \varepsilon_{a+1}, \ldots, \varepsilon_b)$. Define the backward projection operator as $\mathcal{P}^b_a X = \IE(X|\mathcal{F}_a^b) - \IE(X|\mathcal{F}_{a+1}^b)$ with $\mathcal{P}_a^a(X)=\IE[X|\varepsilon_a]$. Observe that for $h \geq 2$,
  \begin{align}
    V_{k,h}- V_{k,h-1}&= \IE(V_k | \boldsymbol{\eta}_{k-h}, \ldots, \boldsymbol{\eta}_{k}) - \IE(V_k | \boldsymbol{\eta}_{k-h+1}, \ldots, \boldsymbol{\eta}_{k})= \sum_{l=1}^{\mathcal{D}_n} \mathcal{P}_{(k-h-1)\mathcal{D}_n + l}^{k \mathcal{D}_n} V_k\nonumber.
  \end{align}
  Using Jensen's inequality (see \cite{Wu2005}; Theorem 1.(i)), one obtains
  \begin{align}
    \|\mathcal{P}_{(k-h-1)\mathcal{D}_n + l}^{k \mathcal{D}_n} V_k\|_{p/2} \leq I +II, \label{Proj}
  \end{align}
  where 
  \begin{align}
    I &= \bigg\|\sum_{t=(k-1)\mathcal{D}_n +1}^{(k\mathcal{D}_n) \wedge n}(X_t -X_{t, \{(k-h-1)\mathcal{D}_n+l\}})\sum_{s=(t-\mathcal{D}_n) \vee 1}^t a_{s,t}X_s\bigg\|_{p/2} \label{I.quad},\\
    II &= \| \sum_{s=[(k-2)\mathcal{D}_n +1] \vee 1}^{(k\mathcal{D}_n) \wedge n} (X_s - X_{s, \{(k-h-1)\mathcal{D}_n+l\}}) \sum_{t=s \vee(k-1) \mathcal{D}_n+1}^{(s+\mathcal{D}_n) \wedge n} a_{s,t} X_{t, \{(k-h-1)\mathcal{D}_n+l\}} \|_{p/2} \label{II.quad}.
  \end{align}
  In order to tackle $I$, we start off by noting the following assertion. In view of Burkholder's inequality (\cite{Rio2009MomentIF}; Theorem 2.1), it follows that
  \begin{align}
    \|\sum_{s=1}^t c_s X_s \|_{p} \leq \sum_{r=0}^{\infty} \| \sum_{s=1}^t c_s \mathcal{P}_{s-r}^s X_s\|_p &\leq C_p \sum_{r=0}^{\infty}\sqrt{\sum_{s=1}^t \|c_s\mathcal{P}_{s-r}^s X_s\|_p^2} \\ &\leq C_p\sum_{r=0}^{\infty} \sqrt{\sum_{s=1}^t c_s^2} \delta_p(r)\nonumber\\&= C_p \Theta_{0,p}\sqrt{\sum_{s=1}^t c_s^2},\label{d.5iii}
  \end{align}
  which entails, invoking H\"older's inequality, that, 
  \begin{align}
    I & \leq \sum_{t=(k-1)\mathcal{D}_n +1}^{(k\mathcal{D}_n) \wedge n}\|X_t -X_{t, \{(k-h-1)\mathcal{D}_n+l\}} \|_p \bigg\|\sum_{s=(t-\mathcal{D}_n) \vee 1}^t a_{s,t}X_s \bigg\|_p \nonumber \\
    & \leq C_p \Theta_{0,p} \sqrt{\mathcal{D}_n} \sum_{t=(k-1)\mathcal{D}_n +1}^{(k\mathcal{D}_n) \wedge n} \delta_p(t-(k-h-1)\mathcal{D}_n-l) \label{Ibound}.
  \end{align}
  Similarly, 
  \begin{align}
    II & \leq C_p \Theta_{0,p} \sqrt{\mathcal{D}_n} \sum_{t=(k-1)\mathcal{D}_n +1}^{(k\mathcal{D}_n) \wedge n} \delta_p(t-(k-h-1)\mathcal{D}_n-l) \label{IIbound}.
  \end{align}
  Thus combining \eqref{Ibound} and \eqref{IIbound} with \eqref{Proj} yields,
  \begin{align}
    \|\mathcal{P}_{(k-h-1)\mathcal{D}_n + l}^{k \mathcal{D}_n} V_k\|_{p/2} \leq C_p \Theta_{0,p} \sqrt{\mathcal{D}_n} \sum_{t=(k-1)\mathcal{D}_n +1}^{(k\mathcal{D}_n) \wedge n} \delta_p(t-(k-h-1)\mathcal{D}_n-l). \label{Pjnorm}
  \end{align}
 Finally, for $p>4$, \citet{Rio2009MomentIF}'s version of Burkholder's inequality (Theorem 2.1 of \cite{Rio2009MomentIF}) along with \eqref{Pjnorm} implies
  \begin{align}
    \| V_{k,h}- V_{k,h-1} \|_{p/2}^2 \leq C_p \sum_{l=1}^{\mathcal{D}_n} \|\mathcal{P}_{(k-h-1)\mathcal{D}_n + l}^{k \mathcal{D}_n} V_k \|_{p/2}^2 &\leq C_p^3 \Theta_{0,p}^2 \mathcal{D}_n \sum_{l=1}^{\mathcal{D}_n} \left(\sum_{d=(h-1)\mathcal{D}_n -l+1}^{(h+1)\mathcal{D}_n -l} \delta_p(d) \right)^2 \nonumber \\
    & \leq C_p^3 \Theta_{0,p}^2 \mathcal{D}_n^2 \left(\sum_{d=(h-2)\mathcal{D}_n+1}^{(h+1)\mathcal{D}_n} \delta_p(d) \right)^2 \nonumber,
  \end{align}
  which completes the proof for $p>4$. For the case $2<p\leq 4$, one proceeds using Theorem 3.2 of \cite{burk73} as follows:
  \begin{align}
    \| V_{k,h}- V_{k,h-1} \|_{p/2}^{p/2} &\leq C_p \IE\left(\left(\sum_{l=1}^{\mathcal{D}_n}|\mathcal{P}_{(k-h-1)\mathcal{D}_n + l}^{k \mathcal{D}_n} V_k|^2\right)^{p/4}\right) \\ & \leq C_p \sum_{l=1}^{\mathcal{D}_n} \IE[|\mathcal{P}_{(k-h-1)\mathcal{D}_n + l}^{k \mathcal{D}_n} V_k|^{p/2}] \nonumber \\&\leq C_p \Theta_{0,p}^{p/2} \mathcal{D}_n^{p/4} \sum_{l=1}^{\mathcal{D}_n}\left( \sum_{d=(h-2)\mathcal{D}_n+1}^{(h+1)\mathcal{D}_n} \delta_p(d)\right)^{p/2} \nonumber \\
    & \leq C_p \Theta_{0,p}^{p/2} \mathcal{D}_n^{1+p/4} \left(\sum_{d=(h-2)\mathcal{D}_n+1}^{(h+1)\mathcal{D}_n} \delta_p(d) \right)^{p/2} \nonumber,
  \end{align}
  where we applied $(|a_1|+\ldots+|a_n|)^{p/4}\leq |a_1|^{p/4} + \ldots + |a_n|^{p/4}$ for $2<p \leq 4$. 
  This completes the proof. 
\end{proof}
Next we will use a version of (41), Proposition 8 of \cite{xiaowu}. 
\begin{lemma}\label{lemma8.2}
  Grant the process \eqref{eq:model_Z} with $\IE(X_t)=0$ and $\Theta_{0,p} < \infty$ for some $p >2$. Then,
  \begin{align}
    \| \sum_{s,t=1}^n a_{s,t}(X_s X_t - \IE(X_sX_t))\|_{p/2} \leq \begin{cases}
      C_p \mathcal{C} \Theta_{0,p}^2 n^{2/p}, & \ 2<p\leq 4 \\
      C_p \mathcal{C} \Theta_{0,p}^2 \sqrt{n}, & \ p \geq 4
    \end{cases}
  \end{align}
  where $\mathcal{C}=\max\{\max_{1 \leq t \leq n}(\sum_{s=1}^n a_{s,t}^2)^{1/2}, \max_{1 \leq s \leq n}(\sum_{t=1}^n a_{s,t}^2)^{1/2} \}$.
\end{lemma}
\begin{proof}
  Let $Q:= \sum_{s,t=1}^n a_{s,t}X_s X_t$. Write $ Q- \IE(Q)=\sum_{r=-\infty}^n P_rQ$, where the projections $P_r$ are defined as in the proof of Lemma \ref{lemmacp}. Now, Jensen's inequality yields,
  \begin{align}
    \|P_rQ\|_{p/2} \leq \|\sum_{s,t=1}^n a_{s,t}\left(X_s X_t - X_{s, \{r\}} X_{t, \{r\}} \right) \|_{p/2} \leq I_r + II_r \nonumber,
  \end{align}
  where 
  \begin{align}
    I_r &= \| \sum_{s,t=1}^n a_{s,t}(X_s - X_{s, \{r\}})X_t\|_{p/2},\\
    II_r&=\|\sum_{s,t=1}^n a_{s,t} X_{s, \{r\}}(X_t - X_{t, \{r\}}) \|_{p/2}.
  \end{align}
  To tackle $I_r$, we employ H\"older's inequality and \eqref{d.5iii}, it follows that,
  \begin{align*}
    I_r \leq \sum_{s=1}^n\| X_s - X_{s, \{r\}}\|_p \|\sum_{t=1}^n a_{s,t}X_t \|_p \leq C_p \Theta_{0,p} \mathcal{C} \sum_{s=1}^n \delta_p(s-r). 
  \end{align*}
  The same bound applies to $II_r$. Now, for $p>4$, Burkholder's inequality (\cite{Rio2009MomentIF}) implies that
  \begin{align*}
    \|Q-\IE(Q)\|_{p/2}^2 &\leq C_p \sum_{r=-\infty}^n \|P_r Q\|^2_{p/2}\leq C_p \Theta_{0,p}^2 \mathcal{C}^2 \sum_{r=-\infty}^n \left(\sum_{s=1}^n \delta_p(s-r) \right)^2\leq C_p \Theta_{0,p}^4 n\mathcal{C}^2.
  \end{align*}
     As for $2< p \leq 4$, invoking \cite{burk88} along with elementary inequality $(|a_1|+\ldots+|a_n|)^{p/4}\leq |a_1|^{p/4} + \ldots + |a_n|^{p/4}$, yields,
     \allowdisplaybreaks
   \begin{align*}
      \|Q-\IE(Q)\|_{p/2}^{p/2} \leq \bigg\|\sqrt{\sum_{r=-\infty}^n |P_rQ|^2} \bigg\|_{p/2}^{p/2}  &\leq \IE\left(\sum_{r=-\infty}^n |P_rQ|^{p/2} \right)\\ &\leq C_p \Theta_{0,p}^{p/2} \mathcal{C}^{p/2} \sum_{r=-\infty}^n \left(\sum_{s=1}^n \delta_p(s-r) \right)^{p/2}\\
      & \leq C_p \mathcal{C}^{p/2}n \Theta_{0,p}^p.
   \end{align*}
   This completes the proof. 
\end{proof}
Finally, we will need a Fuk-Nagaev type inequality \cite{fuk-nagaev, borovkov}. In particular, we will use \cite{borovkov}'s argument that the left-hand side $\IP(|S_n|\geq x)$ in Theorems 1-4 in \cite{fuk-nagaev} can be replaced by $\IP\left(\max_{1 \leq i \leq n}|S_i| \geq x \right)$. This, in conjunction with Corollary 1.6 and 1.8  of \cite{nagaev}, can be summarized into the following result. 
\begin{lemma}\label{lemma8.3}
  Let $Z_1, \ldots, Z_n$ be independent zero-mean random variables with $\IE[|Z_i|^p] < \infty$ for $p>1$. Let $S_i=\sum_{j=1}^i Z_j$. Then, for any $x>0$,
  \begin{align}
    \IP\left(\max_{1 \leq i \leq n}|S_i| \geq x \right) \leq \begin{cases}
      C_p x^{-p} \sum_{i=1}^n \IE[|Z_i|^p], \ & 1<p \leq 2,\\
      C_p x^{-p} \sum_{i=1}^n \IE[|Z_i|^p] + \exp\left(-\frac{x^2}{C_p\sum_{i=1}^n \IE[Z_i^2]}\right), \ & p>2. \label{maxfuknagaev} 
    \end{cases}
  \end{align}
\end{lemma}
Now we have all the technical tools required for the proof of Theorem \ref{thm:maxpartialquad}.
\begin{proof}[Proof of Theorem \ref{thm:maxpartialquad}]
Observe that $Q_n=\sum_{j=1}^{\lceil n/\mathcal{D}_n \rceil} V_j$. Let $U_n= \lceil n/\mathcal{D}_n \rceil$. If $U_n=1$, the conclusion readily follows from Markov's inequality. Therefore let $U_n \geq 2$. 

Denote by $\boldsymbol{\eta}_k=(\varepsilon_{(k-1)\mathcal{D}_n+1}, \ldots, \varepsilon_{k\mathcal{D}_n})$.
Recall $V_k$ from \eqref{Vj}. Let $V_{k, \tau}=\IE[V_k | \boldsymbol{\eta}_k, \ldots, \boldsymbol{\eta}_{k - \tau}]$. Let $L_n=\lfloor \log U_n/\log 2 \rfloor$. We will omit the subscript $n$ from $U_n$ and $L_n$ for presentation purposes, their dependence on $n$ being implicit. Let $\tau_l=2^l$, $1 \leq l \leq L-1$, and $\tau_L=U$. Let
\begin{equation}
  M_{k,l}=\sum_{j=1}^k (V_{j, \tau_l}- V_{j, \tau_{l-1}}), \text{  for $1\leq k \leq U$, and $1 \leq l \leq L$}.
\end{equation}
Define $D_k=\sum_{j=1}^k V_j$ for $1 \leq k \leq U$, and let $D_{k,\tau}=\IE[D_k | \boldsymbol{\eta}_k, \ldots, \boldsymbol{\eta}_{k - \tau}]$. 
Note that
\begin{equation}\label{decomp}
D_k - \IE(D_k)=\sum_{j=1}^k(V_j - V_{j,U}) + \sum_{l=2}^L M_{k,l} + \sum_{j=1}^k (V_{j,2}- \IE(V_{j,2})) .
\end{equation}
Thus,
\begin{align} 
  \max_{1 \leq k \leq U} |D_k - \IE(D_k)| &\leq \max_{1 \leq k \leq U} |\sum_{j=1}^k(V_j - V_{j,U})| + \sum_{l=2}^L  \max_{1 \leq k \leq U}| M_{k,l} | \nonumber\\  & \hspace*{4cm}+ \max_{1 \leq k \leq U} |\sum_{j=1}^k (V_{j,2}- \IE(V_{j,2}))|.\label{eq:2.27}
\end{align}
For the first term in the above sum, note that
\begin{equation} \label{2.28}
  \left \| \max_{1 \leq k \leq U} |D_k - D_{k,U}| \right\|_{p/2} \leq \left\|D_U - D_{U,U} \right\|_{p/2} + \left\|\max_{1 \leq i \leq U-1} \left|\sum_{k=U-i}^{U} (V_k - V_{k, U}) \right| \right\|_{p/2}.
\end{equation}
Now, $V_{k}- V_{k,U}=\sum_{i=U+1}^{\infty}(V_{k,i}- V_{k, i-1})$. 
Since $V_{k,i}- V_{k, i-1}$ are martingale differences with respect to $\sigma(\boldsymbol{\eta}_{k-i}, \boldsymbol{\eta}_{k-i+1}, \ldots)$, hence, using Doob's Inequality we obtain, 
\begin{equation} \label{2.29}
  \left \|\max_{1 \leq i \leq U-1} \left|\sum_{k=U-i}^{U} (V_{k,j} - V_{k, j-1}) \right| \right\|_{p/2} \leq C_p \|D_{U,j} - D_{U,j-1} \|_{p/2} \ .
\end{equation}
Therefore, Lemma \ref{lemma8.1} along with Burkholder inequality (\cite{burk73} for $2<p<4$ along with $(|a_1|+\ldots+|a_n|)^{p/4}\leq |a_1|^{p/4} + \ldots + |a_n|^{p/4}$, and \cite{Rio2009MomentIF}'s version for $p>4$) implies,
\begin{align}
  \| D_{U,j}- D_{U, j-1}\|_{p/2} &\leq \begin{cases}
    C_p U^{2/p}\mathcal{D}_n^{1/2 + 2/p}\Theta_{0,p} \sum_{d=(j-1)\mathcal{D}_n+1}^{(j+1)\mathcal{D}_n} \delta_p(d), & \ 2<p\leq 4, \\
    C_p \sqrt{U}\mathcal{D}_n\Theta_{0,p} \sum_{d=(j-1)\mathcal{D}_n+1}^{(j+1)\mathcal{D}_n} \delta_p(d), & \ p>4.
  \end{cases}\label{d_u}
\end{align}
\color{black}
Therefore, using \[\left\|\max_{1 \leq i \leq U-1} \left|\sum_{k=U-i}^{U} (V_k - V_{k, U}) \right| \right\|_{p/2} \leq \sum_{j=U+1}^{\infty} \left \|\max_{1 \leq i \leq U-1} \left|\sum_{k=U-i}^{U} (V_{k,j} - V_{k, j-1}) \right| \right\|_{p/2},\] we have,
\begin{equation} \label{2.30}
  \left\|\max_{1 \leq i \leq U-1} \left|\sum_{k=U-i}^{U} (V_k - V_{k, U}) \right| \right\|_{p/2} \leq \begin{cases}
    C_p U^{2/p}\mathcal{D}_n^{1/2 + 2/p}\mu_{p,A}^2 n^{-A}, & \ 2<p\leq 4, \\
    C_p \sqrt{U}\mathcal{D}_n\mu_{p,A}^2 n^{-A}, & \ p>4,
  \end{cases}
\end{equation}
where we have used Condition \ref{cond:fdm3}: $\Theta_{U\mathcal{D}_n+1, p} \leq C (U\mathcal{D}_n+1)^{-A}\mu_{p,A}\leq C n^{-A}\mu_{p,A}$.
Proceeding similarly, 
\[\||D_U - D_{U,U}|\|_{p/2} \leq \begin{cases}
    C_p U^{2/p}\mathcal{D}_n^{1/2 + 2/p}\mu_{p,A}^2 n^{-A}, & \ 2<p\leq 4, \\
    C_p \sqrt{U}\mathcal{D}_n\mu_{p,A}^2 n^{-A}, & \ p>4.
  \end{cases} .\]
Hence, by Markov's inequality,
\begin{equation} \label{2.31}
\allowdisplaybreaks
  \IP\left( \max_{1 \leq k \leq U} |D_k - D_{k,U}| \geq x \right) \leq \begin{cases}
  C_px^{-p/2} n^{1-Ap/2} \mathcal{D}_n^{p/4}\mu_{p,A}^{p}, & \ 2<p\leq 4\\
    C_px^{-p/2} n^{p/4-A p/2} \mathcal{D}_n^{p/4}\mu_{p,A}^p, & \ p>4.
  \end{cases} 
\end{equation}
\color{black}
For the second term in \eqref{eq:2.27}, define the following quantities:
\allowdisplaybreaks
\begin{align}
  Y_{h,l}&=\sum_{j=(h-1)\tau_l +1}^{(h\tau_l) \wedge U} (V_{j,\tau_l}- V_{j, \tau_{l-1}}), \ \ 1\leq h \leq \lceil U / \tau_l \rceil:=U_0, \label{2.32}\\
  R^e_{s,l}&=\sum_{h \text{ even}}^{s} Y_{h,l} \ , \ R^o_{k,l}=\sum_{h \text{ odd}}^{s}Y_{h,l}, \ 1 \leq s \leq U_0. \label{2.33}
\end{align}
Further let $\{\lambda_j\}_{1 \leq j \leq L}$ be a positive sequence that $\sum_{l=1}^L \lambda_l \leq 1$. We will specify the choice of $\lambda_j$ later. For some $s \in \N$, denote by $s_l:=s\tau_l \wedge U$. 
Therefore,
\begin{align}
  \IP(\max_{1 \leq k \leq U}|M_{k,l}| \geq 3\lambda_l x) \leq \IP(\max_{1 \leq s \leq U_0} |R^e_{s,l}| \geq \lambda_l x) &+ \IP(\max_{1 \leq s \leq U_0} |R^o_{s,l}| \geq \lambda_l x) + \nonumber\\ &\sum_{s=1}^{U_0} \IP(\max_{s_l+1 \leq j \leq (s+1)_l} |M_{j,l}- M_{s_l,l}| \geq \lambda_l x ). \label{evenodd}
\end{align}

\noindent For the first two terms in \eqref{evenodd}, note that $Y_{h_1,l}$ and $Y_{h_2,l}$ are independent for $|h_1-h_2|\geq 2$. Therefore, using Lemma \ref{lemma8.3}, we obtain
\begin{equation} \label{2.34}
  \IP(\max_{1 \leq s \leq U_0} |R^e_{s,l}| \geq \lambda_l x) \leq \begin{cases}
  C_p \frac{\sum_{h \text{ even}} \IE[|Y_{h,l}|^{p/2}]}{(\lambda_l x)^{p/2}}, & \ 2< p \leq 4, \\
    C_p \frac{\sum_{h \text{ even}} \IE[|Y_{h,l}|^{p/2}]}{(\lambda_l x)^{p/2}}+ 2 \exp\left(- C_p \frac{(\lambda_l x)^2}{\sum_{h \text{ even}} \IE[|Y_{h,l}|^2]} \right), & \ p > 4.
  \end{cases} 
\end{equation}
An argument similar to \eqref{2.29} and \eqref{2.30} yields,
\begin{align*}
  \|Y_{h,l}\|_{p/2} &\leq \begin{cases}
    C_p\tau_l^{2/p}\mathcal{D}_n^{1/2+2/p}(\tau_l \mathcal{D}_n)^{-A} \mu_{p,A}^2, \ & 2 < p \leq 4,\\
    C_p\sqrt{\tau_l}\mathcal{D}_n (\tau_l \mathcal{D}_n)^{-A} \mu_{p,A}^2, \ & p>4.
  \end{cases}
\end{align*}
Thus, 
\begin{align}
   &\IP(\max_{1 \leq s \leq U_0} |R^e_{s,l}| \geq \lambda_l x)\nonumber\\ &\leq
   \begin{cases}
   C_p (\lambda_l x)^{-p/2}\frac{U}{\tau_l} \tau_l^{1-A p/2} \mathcal{D}_n^{p/4+1-A p/2} \mu_{p,A}^p, \ & 2<p\leq 4, \\
     C_p (\lambda_l x)^{-p/2}\frac{U}{\tau_l} \tau_l^{p/4-A p/2} \mathcal{D}_n^{p/2-A p/2} \mu_{p,A}^p + 2 \exp \left(-C_p \frac{(\lambda_l x)^2}{\frac{U}{\tau_l}{\tau_l}\mathcal{D}_n^2 (\tau_l \mathcal{D}_n)^{-2A} \mu_{4,A}^4} \right), \ & p>4
   \end{cases} \nonumber\\
   & \leq \begin{cases}
   C_p(\lambda_l x)^{-p/2} \tau_l^{-A p/2} n \mathcal{D}_n^{p/4-A p/2} \mu_{p,A}^p, \ & 2<p\leq 4, \\
     C_p(\lambda_l x)^{-p/2} \tau_l^{p/4-A p/2-1} n \mathcal{D}_n^{p/2-A p/2 -1} \mu_{p,A}^p + 2 \exp \left(-C_p \frac{(\lambda_l x)^2 (\tau_l \mathcal{D}_n)^{2 A}}{n \mathcal{D}_n \mu_{4,A}^4} \right), \ & p>4. 
   \end{cases} \label{even}
\end{align}
A similar inequality holds for $\max_{1 \leq s \leq U_0} |R^o_{s,l}|$.
To tackle the third term $\sum_{s=1}^{U_0} \IP(\max_{s_l+1 \leq j \leq (s+1)_l} |M_{j,l}- M_{s_l, l}| \geq \lambda_l x )$ in \eqref{evenodd}, we employ an argument similar to \eqref{2.28} through \eqref{2.31}. Write 
\begin{align*}
\bigg\|\max_{s_l+1 \leq j \leq (s+1)_l} |M_{j,l}- M_{s_l,l}| \bigg\|_{p/2} &\leq \| M_{(s+1)_l,l}- M_{s_l,l}\|_{p/2} \\ & \hspace{2cm} +\bigg\|\max_{s_l + 2 \leq j \leq (s+1)_l} \bigg|\sum_{k=j}^{(s+1)_l} (V_{k, \tau_l}- V_{k, \tau_{l-1}}) \bigg| \bigg\|_{p/2}. 
\end{align*}
Using Doob's Inequality, 
\[\left\|\max_{s_l + 2 \leq j \leq (s+1)_l} \left|\sum_{k=j}^{(s+1)_l} (V_{k, \tau_l}- V_{k, \tau_{l-1}}) \right| \right\|_{p/2} \leq \frac{p/2}{p/2-1} \| M_{(s+1)_l, l}- M_{s_l,l}\|_{p/2} .\]
An argument similar to \eqref{d_u} yields, \allowdisplaybreaks
\begin{align*}
  \| M_{(s+1)_l, l}- M_{s_l,l}\|_{p/2}  &\leq \begin{cases}
   C_p \tau_l^{2/p}\mathcal{D}_n^{1/2+2/p}(\tau_l \mathcal{D}_n)^{-A} \mu_{p,A}^2, \ & 2 < p \leq 4,\\
    C_p\sqrt{\tau_l}\mathcal{D}_n (\tau_l \mathcal{D}_n)^{-A} \mu_{p,A}^2, \ & p>4.
  \end{cases}
\end{align*}
Therefore, applying Markov's inequality we have
\begin{align} 
\begin{split}
&\sum_{s=1}^{U_0} \IP\left(\max_{s_l+1 \leq j \leq (s+1)_l} |M_{j,l}- M_{s_l,l}|  \geq \lambda_l x \right)\\  &\leq \begin{cases}
   C_p(\lambda_l x)^{-p/2} \tau_l^{-A p/2} n \mathcal{D}_n^{p/4-A p/2} \mu_{p,A}^p, \ & 2<p \leq 4,\\
     C_p(\lambda_l x)^{-p/2} \tau_l^{p/4-A p/2-1} n \mathcal{D}_n^{p/2-A p/2 -1} \mu_{p,A}^p, \ & p>4. 
   \end{cases} \label{eq:Mlast}
\end{split}
\end{align}
Thus, combining \eqref{even} and \eqref{eq:Mlast} in \eqref{evenodd}, we get, 
\begin{align} \label{2ndterm}
   \IP(\max_{1 \leq k \leq U}|M_{k,l}| \geq 3\lambda_l x) &\leq \begin{cases}
   &C_p(\lambda_l x)^{-p/2} \tau_l^{-A p/2} n \mathcal{D}_n^{p/4-A p/2} \mu_{p,A}^p, \hspace*{1.5cm} 2<p\leq 4, \\
    & C_p(\lambda_l x)^{-p/2} \tau_l^{p/4-A p/2-1} n \mathcal{D}_n^{p/2-A p/2 -1} \mu_{p,A}^p \\ & \hspace*{1.5cm} + 2 \exp \left(-C_p \frac{(\lambda_l x)^2 (\tau_l \mathcal{D}_n)^{2 A}}{n \mathcal{D}_n \mu_{4,A}^4} \right), \hspace*{1.35cm} p>4. 
   \end{cases}
\end{align}
Using \eqref{2ndterm}, we have for the second term in \eqref{eq:2.27},
\begin{align}
  \IP\left(\sum_{l=2}^L \max_{1 \leq k \leq U} |M_{k,l} |\geq 3x\right) & \leq \sum_{l=2}^L \IP(\max_{1 \leq k \leq U}|M_{k,l}| \geq 3\lambda_l x) \nonumber\\
  & \leq \begin{cases}
  C_px^{-p/2} n\mathcal{D}_n^{p/4-Ap/2} \mu_{p,A}^p \cdot I_1, \ & 2<p\leq 4, \\
    C_px^{-p/2} n \mathcal{D}_n^{p/2 - A p/2 -1} \mu_{p,A}^p \cdot I_2 + 4 \cdot II, \ & p>4,
  \end{cases} \label{eqnsbeforechoice}
\end{align}
where 
\begin{align*}
&I_1=\sum_{l=2}^L \lambda_l^{-p/2} \tau_l^{-Ap/2}, \\ &I_2=\sum_{l=2}^L \lambda_l^{-p/2} \tau_l^{p/4-A p /2 -1}, \\ &II=\sum_{l=2}^L \exp\left( -C_p \frac{(\lambda_l x)^2 (\tau_l \mathcal{D}_n)^{2 A}}{n \mathcal{D}_n \mu_{4,A}^4}\right).
\end{align*}
Let $\lambda_l=(1/l^2)/(\pi^2/3)$ for $1\leq l \leq L/2$, and $\lambda_l=(1/(L+1-l)^2)/(\pi^2/3)$ for $L/2 < l \leq L$. Clearly $\sum_{l=1}^L \lambda_l \leq 1$. With our choice of $\lambda_l$ and $\tau_l$, elementary calculation using $A>1/2 -1/p$ and $\min _{l \geq 1} \lambda_l^2 \tau_l^{2A}>0$ shows that there exists a constant $C$ such that
\begin{equation} \label{lambdachoice}
  I_1 \leq C ; I_2 \leq C ; \ II \leq C \exp \left(-C_p\frac{x^2}{\mu_{4,A}^4 n \mathcal{D}_n^{1-2A}}\right) .
\end{equation}
Putting \eqref{lambdachoice} in \eqref{eqnsbeforechoice}, one obtains
\begin{equation} \label{2.36}
\IP\left(\sum_{l=2}^L  \max_{1 \leq k \leq U} |M_{k,l} |\geq 3x\right) \leq \begin{cases}
 C_px^{-p/2} \mu_{p,A}^p n \mathcal{D}_n^{p/4}, \ \ 2<p\leq 4, \\
   C_px^{-p/2} \mu_{p,A}^p n \mathcal{D}_n^{p/2-1} + C \exp\left(-C_p \frac{x^2}{n\mathcal{D}_n^{1-2A}\mu_{4,A}^4} \right), \  p>4.
\end{cases} 
\end{equation}
Now finally we tackle the third term in \eqref{eq:2.27}. Note that as $\boldsymbol{\eta}_k$'s are independent, hence $V_{k,2}$ and $V_{k',2}$ are independent if $|k-k'|>2$. We again employ Lemma \ref{lemma8.3} and techniques similar to \eqref{2.32}, \eqref{2.33} and \eqref{2.34} to obtain,
\begin{align}
  &\IP \left( \max_{1 \leq k \leq U}| \sum_{j=1}^k (V_{j,2}- \IE(V_{j,2}))| \geq x \right)\nonumber\\
  & \leq \begin{cases}
  C_px^{-p/2} \sum_{j=1}^U \IE(|V_{j,2}- \IE(V_{j,2})|^{p/2}), \ \ 2< p \leq 4, \\
    C_px^{-p/2} \sum_{j=1}^U \IE(|V_{j,2}- \IE(V_{j,2})|^{p/2}) + \nonumber\\ \hspace*{2cm} \ \ 2 \exp \left(-C_p\frac{x^2}{\sum_{j \text{ even}} \IE(|V_{j,2}- \IE(V_{j,2})|^2)}\right) + 2 \exp \left(-C_p \frac{x^2}{\sum_{j \text{ odd}} \IE(|V_{j,2}- \IE(V_{j,2})|^2)}\right),  \ p>4. 
  \end{cases} \label{3rdterm}
\end{align}
By conditional Jensen's inequality and Lemma \ref{lemma8.2} (noting that $\mathcal{C}=O(\sqrt{\mathcal{D}_n}))$, we get
\[
\IE(|V_{j,2}- \IE(V_{j,2})|^{p/2}) \leq \IE(|V_j - \IE(V_j)|^{p/2}) \leq \begin{cases}
C_p\left(\mathcal{D}_n^{1/2 + 2/p}\right)^{p/2} \mu_{p,A}^p, & 2 <p \leq 4,\\
C_p\mathcal{D}_n^{p/2}  \mu_{p,A}^p, & p>4, 
\end{cases}\]
which yields
\begin{equation} \label{2.37}
  \IP \left( |\max_{1 \leq k \leq U} \sum_{j=1}^k (V_{j,2}- \IE(V_{j,2}))| \geq x \right) \leq \begin{cases}
   C_px^{-p/2} n \mathcal{D}_n^{p/4} \mu_{p,A}^p,  \ \ 2<p\leq 4,\\
    C_px^{-p/2} n \mathcal{D}_n^{p/2-1} \mu_{p,A}^p + 4 \exp \left(-C_p \frac{x^2}{n\mathcal{D}_n \mu_{4,A}^4} \right),  \ p>4.
  \end{cases}
\end{equation}
Combining \eqref{2.31}, \eqref{2.36} and \eqref{2.37}, we have the result. 
\end{proof}
\subsection{Proof of Lemma \ref{lemmacp}} \label{prooflemmacp}
\begin{proof}
Define the projection operator $P_i$ as $P_iX=\IE[X | \mathcal{F}_i]- \IE[X | \mathcal{F}_{i-1}]$ where $\mathcal{F}_i=\sigma(\ldots, \varepsilon_{i-1}, \varepsilon_i)$.
  Note that for $l >k$, 
\begin{align}\label{eq:product}
  |\IE(X_{k}X_{l}) | = |\sum_{i \in \mathbb{Z}} \sum_{j \in \mathbb{Z}} \IE\left((P_iX_{k})(P_jX_{l}) \right)| &\leq \sum_{i \in \mathbb{Z}} \| P_i(X_{k}) \| \|P_i(X_{l})\| \nonumber \\
  & \leq \sum_{i=\infty}^k \delta_{p}(k-i)\delta_{p}(l-i)= \sum_{i=0}^{\infty} \delta_p(i) \delta_p(i + l-k).
\end{align}
Using (\ref{eq:product}) repeatedly, 
\begin{align*}
  &\max_{1\leq k \leq \lfloor\frac{n}{m}\rfloor} |\IE(B_k B_{k+1})| \nonumber\\ &\leq \sum_{j=0}^{2m} (m- |m-j|)\sum_{i=0}^{\infty} \delta_p(i) \delta_p(i+j)\nonumber\\
  &\leq \sum_{i=0}^{\infty} \delta_p(i) \left(\Theta_{i+1,p} + \Theta_{i+2,p} + \ldots + \Theta_{i+2m-1,p} \right)\nonumber\\
  &\leq \sum_{i=0}^{\infty} \delta_p(i) \sum_{j=1}^{2m-1} \Theta_{j,p}=\Theta_{0,p}\sum_{j=1}^{2m-1} \Theta_{j,p}= \mu_{p,A}O\left(\sum_{j=1}^{2m-1} (j+1)^{-A} \right) = O(1), 
\end{align*}
 since $A>1$ in Condition \ref{cond:fdm3}. Moreover, for fixed $i,j$, via an exact same argument as above, one obtains, 
\begin{align}\label{blockprod}
   |\IE(B_i B_j)| \leq \Theta_{0,p} \sum_{k=|i-j-1|m +1}^{|i-j+1|m-1}\Theta_{k,p}. 
\end{align}
Then \eqref{blockprod} in conjunction with \eqref{eq:Thetaip} directly implies that
\begin{align}
  \max_{1 \leq k \leq \lceil n/m \rceil} \sum_{i: |i-k|\geq 2}|\IE( B_iB_k)| & \leq\Theta_{0,p} \max_{1 \leq k \leq \lceil n/m \rceil}\sum_{i=1}^{k-2} \sum_{j=im+1}^{(i+2)m -1}\Theta_{j,p} \nonumber  \\ &\leq 2 \Theta_{0,p} \sum_{j=m+1}^{\infty} \Theta_{j,p}= O(m^{1-A}). 
\end{align}
This completes the proof of \eqref{eq:blockproduct}.
\end{proof}
\ignore{
\subsection{Proof of Proposition \ref{prop:Bj2-EBj2}}
\label{sec:proofBj2-EBj2}
\begin{proof}
Note that we have $a_{s,t}=\begin{cases}1/2 \ & s=t\\ 1 \ & 1\leq |s-t|< 2m \\ 0 \ & \text{otherwise} \end{cases}$. 
Thus, taking $\mathcal{D}_n=2m$, Theorem \ref{thm:maxpartialquad} implies that, 
\begin{equation} \label{approx1}
  \max_{1 \leq k \leq \lfloor n/m \rfloor}\bigg|\sum_{j=1}^k (B_j^2 + 2B_j B_{j+1} - \IE[B_j^2 + 2B_j B_{j+1}])\bigg|=o_{\IP}(n^{\max\{2/p, 1/2\}} m^{1/2} l(n)) 
\end{equation}
\text{ for a slowly decreasing function $l(n)$.} Now, using $\max_{1 \leq j \leq \lfloor n/m \rfloor}\IE[\max_{1 \leq k\leq m} |X_{j+1} + \ldots X_{j+k}|^p]=O(m)$ along with Markov's inequality, one obtains, 
\begin{equation}\label{approx2}
  \max_{1 \leq i \leq n} |\mathcal{T}_i - \sum_{j=1}^{\lfloor i/m \rfloor} (B_j^2 + 2B_j B_{j+1})|=o_{\IP}(n^{\max\{2/p, 1/2\}} m^{1/2}).
\end{equation} 
Next we observe that \eqref{eq:blockproduct} yields
\begin{equation}\label{approx3}
  \max_{1 \leq i \leq n}|\IE(S_i^2) - \sum_{j=1}^{\lfloor i/m \rfloor} \IE(B_j^2 + 2B_j B_{j+1})|=O(nm^{-A}).
\end{equation}
\eqref{approx1}, \eqref{approx2} and \eqref{approx3} completes the proof. 
\end{proof}
}
\subsection{Proof of Theorem \ref{thm:suboptimal}} \label{sec:proofsuboptimal}
 We will define $L$ and $\alpha$ as in the proof of Theorem \ref{thm:GA}. Let $\nu=\min\{(1+A)/(2+4A), (1 + 4A/p)/(2+4A)\}$. The theorem follows trivially from Theorem \ref{thm:GA} if $A>A_0$. Thus let $A \leq A_0$. 
Specifically, with $1< A \leq A_0$, our choice of $L$ and $\alpha$ satisfies the following, which will be used in our proofs:
\begin{align}
  &\frac{1}{2} - \nu - \frac{LA}{2} < 0, \label{eq1n:suboptimal}\\
  & L\left(\frac{\alpha}{2}-1\right)+ 1 - \alpha \nu <0, \label{eq2n:suboptimal} \\
  &\alpha \geq \max\{p, 2(1+p+pA)/3\}, \label{eq3n:suboptimal}\\
  & 1/p - 1/\alpha + L - L(A+1)p/\alpha =0. \label{eq4n:suboptimal}
\end{align}
We will need a slightly different version of Lemma \ref{Rosenthal}. To avoid confusion, we state and prove it separately. In the following $C$ will denote a constant whose value will depend on $p$ and $A$, and whose value might change from line to line.
\begin{lemma}\label{Rosenthal:suboptimal}
  Assume Conditions \ref{cond:fdm3} and \ref{cond:ui}, along with \eqref{eq1n:suboptimal}, \eqref{eq2n:suboptimal}, \eqref{eq3n:suboptimal} and \eqref{eq4n:suboptimal} for $A$, $L$ and $\alpha$. Let $m=\lfloor n^{L} \rfloor$ and let \[\Tilde{R}_{s,t}=\Tilde{X}_s + \ldots+ \Tilde{X}_t, \] 
  where $\Tilde{X}_i$ is as defined in \eqref{m-dep}. Then 
  \begin{align}
   \max_s \IE \left[\max_{1 \leq t \leq m} |\tilde{R}_{s,t}|^{\alpha} \right] = o(mn^{\alpha \nu-1}).
  \end{align}
  \end{lemma}
 \begin{proof}
   The proof of this lemma is almost same as that of Lemma \ref{Rosenthal}. The only point of differences are the use of \eqref{eq2n:suboptimal} instead of \eqref{eq2n}, as well as a different treatment of the term $III$ in \eqref{liudecomp}. The latter difference is necessitated as we no longer have $\alpha< 2(1+p+pA)/3$ as we had in \eqref{eq3n}.

   In fact for term $III$ we will proceed as follows. For the case $\alpha>2(1+p+pA)/3$, using \eqref{eq3n:suboptimal}, and using same argument as \eqref{eq2.3kmt}, one obtains
   \begin{align}
     III = m^{1/\alpha} \sum_{j=1}^m j^{1/2 -1/\alpha}\Tilde{\delta}_{\alpha}(j) \nonumber
     &\leq C m^{1/ \alpha} n^{1/p -1/\alpha} \sum_{l=1}^{\lceil \log_2 m \rceil} \sum_{j=2^l}^{2^{l+1}-1} j^{1/2 - 1/\alpha} \delta_p(j)^{p/\alpha} \nonumber \\
     & \leq C m^{1/ \alpha} n^{1/p -1/\alpha} \sum_{l=1}^{\lceil \log_2 m \rceil} 2^{l(3/2 - 1/\alpha - p/\alpha)} O(2^{-lAp/\alpha}) \nonumber \\
     & \leq C m^{3/2 - p/\alpha - Ap/\alpha} n^{1/p -1/\alpha}=m^{1/2}, \label{8.41}
   \end{align}
   where the last equality follows from \eqref{eq4n:suboptimal}. Therefore, in view of \eqref{eq2n:suboptimal}, we obtain
   \begin{equation}
     \frac{n^{1-\alpha \nu}}{m} III^{\alpha} = n^{1 -\alpha \nu} m^{-1} O(m^{\alpha/2}) =o(1).
   \end{equation}
 In case $\alpha =2(1+p+pA)/3$, same treatment as \eqref{8.41} yields,
 \begin{align}
   III \leq C m^{1/\alpha} n^{1/p - 1/\alpha} \log_2 m \leq C L m^{1/\alpha} n^{1/p - 1/\alpha} \log_2 n.
 \end{align}
   One immediately obtains,
   \begin{align}
      \frac{n^{1-\alpha \nu}}{m} III^{\alpha} = L (\log_2 n) n^{\alpha/p - \alpha \nu} O(1) = o(1),
   \end{align}
   where the last assertion is due to $\nu > 1/p$. This completes the proof of this lemma.
 \end{proof} 
 \begin{proof}[Proof of Theorem \ref{thm:suboptimal}]
 The proof follows mostly along the lines of the proof of Theorem \ref{thm:GA}, with $S_i^{\oplus}$ and $\Tilde{S}_i$ defined as in that proof. We list below the points of differences from that proof.
\begin{itemize}
  \item As above, use \eqref{eq1n:suboptimal} and Lemma \ref{Rosenthal:suboptimal} instead of whenever \eqref{eq1n} and Lemma \ref{Rosenthal} is used in the proof of Theorem \ref{thm:GA}.
  \item Proposition \ref{prop2} now holds with a rate of $n^{\nu}$. 
  \item Proposition \ref{prop3} also holds with a rate of $n^{\nu}$. For the proof, we will investigate $\IP(\max_{1 \leq i \leq n} |\Tilde{S}_i - S_i^{\oplus}| > n^{1/4} \delta )$.
  \item Instead of \eqref{eq:condn1}, we will reach a rate of $n^{\nu}$ using $x=n^{\nu}$ and Lemma \ref{Rosenthal:suboptimal} in the previous step. 
  \item Investigate $ \IP(\max_{1\leq k \leq l_n} |\tilde{V}_{2k}(\boldsymbol{\eta}_{3k})| \geq cn^{2\nu})$ to obtain a rate of $n^{\nu}$ instead of \eqref{eq:small2}. 
\end{itemize}
\end{proof}
\ignore{
\section{Some important propositions}
The following series of propositions, leading up to Proposition \ref{prop5} enables us to use Theorem 4 of \cite{gotzezaitsev} on our conditionally independent processes $Y_l^{\boldsymbol{a}}$ as well as their unconditional counterparts $S_i^{\natural}$. 

\begin{proposition}\label{prop1zaitsev}
Recall $B_j$ from \eqref{blockdefn} and $c$ from Condition \ref{cond:regularity}. Then it holds that
\begin{equation}
  \Omega(c m) = \rho_{\star}(\text{Var}(\tilde{B_j})) \leq \rho^{\star}(\text{Var}(\tilde{B_j})) = O(m \Theta_{0,2}^2). 
\end{equation}
\end{proposition}
\begin{proof}
  Without loss of generality assume $j=1$. Observe that $\lim_{m \to \infty} \frac{\text{Var}(B_1)}{m} = \Sigma$. Therefore it holds that
  \begin{equation}\label{twosidedeigen}
    \Omega(\lambda m )=\rho_{\star}(\text{Var}(B_1)) \leq \rho^{\star}(\text{Var}(B_1)) \leq \| B_1 \|^2 = O(m \Theta_{0,2}^2),
  \end{equation}
  where the first and second equality follows from \eqref{loweigen} and Burkholder's inequality respectively. Moreover, $\|S_m^{\oplus} - B_1\|=o(m)$ and in view of \cite{liu_wu_2010}, $\|S_m^{\oplus} - \Tilde{B_1}\| =O(\sqrt{m}\Theta_{m,2}) = o(\sqrt{m})$. This along with \eqref{twosidedeigen} completes the proof.
\end{proof}
\begin{proposition}\label{prop2zaitsev}
  For a sequence $\boldsymbol{a}$, recall $Y_j^{\boldsymbol{a}}$ from \eqref{Yj}.Let $\boldsymbol{\eta}=(\ldots, \boldsymbol{\eta}_0, \boldsymbol{\eta}_3, \ldots)$, where $\boldsymbol{\eta}_{k}=(\varepsilon_{(k-1)m+1}, \ldots, \varepsilon_{km})$. Then 
  \begin{equation}\label{blockeigen}
     \Omega(m) = \rho_{\star}(\text{Var}(Y_j^{\boldsymbol{\eta}})) \leq \rho^{\star}(\text{Var}(Y_j^{\boldsymbol{\eta}})) = O(m).
  \end{equation}
\end{proposition}
\begin{proof}
  \eqref{blockeigen} follows directly from \eqref{twosidedeigen} in view of 
  \begin{align*}
    |\|Y_j^{\boldsymbol{\eta}}\|^2- \|\tilde{B}_{3j-2} + \tilde{B}_{3j-1} + \Tilde{B}_{3j}\|^2 |=\|M_{3j-2}^{\boldsymbol{\eta}}\|^2 + \|M_{3j}^{\boldsymbol{\eta}}\|^2= O(m\Theta_{0,2}^2). 
  \end{align*}
\end{proof}
\begin{proposition}\label{prop3zaitsev}
  Let $J=K^{2/\gamma}/\log^2 K$. Then, there exists a constant $c$ such that
  \begin{equation}
    \IP\left(\boldsymbol{a}: \max_{1 \leq t \leq K/J} \bigg|\text{Var}\bigg(\sum_{l=(t-1)J}^{tJ-1}Y_l^{\boldsymbol{a}}\bigg) - \IE_{\boldsymbol{a}}[\text{Var}\bigg(\sum_{l=(t-1)J}^{tJ-1}Y_l^{\boldsymbol{a}}\bigg)] \bigg |\geq cJm \right) \to 0 \text{ as } n \to \infty. 
  \end{equation}
\end{proposition}
\begin{proof}
  Without loss of generality assume that $V_l^{\boldsymbol{a}}$ are independent for different $l$, otherwise we can break the sums inside the probability statement into the even and odd sums. Further we assume $d=1$. The proof easily generalizes for the multivariate case. Therefore, in view of \eqref{decompvar}, it is enough to show that 
  \begin{equation}\label{propld}
    K \max_{1 \leq t \leq K/J} \max_{(t-1)J \leq l \leq tJ-1} \bigg[\IP\bigg( |\tilde{V}_{2l}({\boldsymbol{a}_{3l}}) - \IE(\tilde{V}_{2l}({\boldsymbol{a}_{3l}}))| \geq clm\bigg)+ \IP\bigg( |\tilde{V}_{2l-1}({\boldsymbol{a}_{3l-3}}) - \IE(\tilde{V}_{2l-1}({\boldsymbol{a}_{3l-3}}))| \geq cJm\bigg)\bigg] \to 0.
  \end{equation}
  Assume without loss of generality $l=1$. Observe that ,
  \begin{align}\label{decomp1}
    |\tilde{V}_1({\boldsymbol{a}_0})- \IE[\tilde{V}_1({\boldsymbol{a}_0})]| &\leq |\|\Tilde{B}_{1}(\boldsymbol{a}_{0})\|^2 - \IE[\|\Tilde{B}_{1}(\boldsymbol{a}_{0})\|^2]| + |\|M_{1}(\boldsymbol{a}_{0})\|^2 - \IE[\|M_{1}(\boldsymbol{a}_{0})\|^2]|\nonumber\\ & \hspace*{2cm}+ 2 |\IE[\Tilde{B}_1 C_{2}(\boldsymbol{\eta}_1)|a_{1-m}, \ldots, a_0] - \IE[\Tilde{B}_1 C_{2}(\boldsymbol{\eta}_1)] |. 
  \end{align}
  For the first term in \eqref{decomp1}, note that $\|\Tilde{B}_{1}(\boldsymbol{a}_{0}) \|^2=\IE[\Tilde{S}_m^2|a_0, \ldots, a_{1-m}]$. Therefore, Burkholder's inequality yields
  \begin{align}\label{1stvar}
    \IE\bigg[ |\|\Tilde{B}_{1}(\boldsymbol{a}_{0})\|^2 - \IE[\|\Tilde{B}_{1}(\boldsymbol{a}_{0})\|^2]|^{\gamma/2}\bigg] = \|\sum_{j=-m}^0 P_j \Tilde{S}_m^2 \|_{\gamma/2}^{\gamma/2} \leq C_{\gamma}\bigg(\sum_{j=-m}^0 \|P_j \Tilde{S}_m^2 \|_{\gamma/2}^2 \bigg)^{\gamma/4}.
  \end{align}
  Using $\|P_j X_i\|_{\gamma} \leq \delta_{i-j, \gamma}$, we obtain 
  \begin{equation}\label{pjsm}
  \|P_j \Tilde{S}_m^2\|_{\gamma/2}= O(m^{1/2})\sum_{r=1}^m\Tilde{\delta}_{r-j, \gamma} =O(m) n^{1/p-1/\gamma}\sum_{r=1}^m{\delta}_{r-j, p}^{p/\gamma}.
   \end{equation}
  In view of the fact that there exists $A^{\star}> A_0$ such that $3 - 2(A+1)p/\gamma=0$, observe that $\Theta_{i,p}=O(i^{-A})=O(i^{-A'}(\log i)^{-B})$ for some $A_0<A'<\min\{A, A^{\star}\}$ and $B> 2\gamma/p$. The entire proof of the main theorem goes through with $A'$ instead of $A$. Therefore, without loss of generality we assume $3 - 2(A+1)p/\gamma>0$. Putting \eqref{pjsm} back in \eqref{1stvar}, 
   \begin{align}
     \sum_{j=-m}^0 \|P_j \Tilde{S}_m^2 \|_{\gamma/2}^2 &=O(m)n^{2/p - 2/\gamma} \sum_{j=0}^m \bigg(\sum_{r=1}^m{\delta}_{r+j, \gamma}^{p/\gamma} \bigg)^2 \nonumber \\ 
     &= O(m) n^{2/p -2/\gamma} \sum_{j=0}^m \sum_{l=0}^{\log_2 m } 2^{2l(1-p/\gamma)} \Theta_{2^l+j, p}^{2p/\gamma} \nonumber \\
     &= O(m) n^{2/p - 2/\gamma} m^{3 - 2(A+1)p/\gamma} (\log n)^{-2p/\gamma}\nonumber,
   \end{align}
   which immediately yields,
   \begin{equation} \label{finalde}
     \IE\bigg[ |\|\Tilde{B}_{1}(\boldsymbol{a}_{0})\|^2 - \IE[\|\Tilde{B}_{1}(\boldsymbol{a}_{0})\|^2]|^{\gamma/2}\bigg] = O(1)m^{\gamma - (A+1)\frac{p}{2\gamma}}n^{\frac{\gamma}{2p}-1/2}(\log n)^{-2p/\gamma}=o((Jm)^{\gamma/2}),
   \end{equation} 
   where the last equality follows from \eqref{eq4n} and $\log J \asymp \log m \asymp \log n$. For the second and third terms of \eqref{decomp1}, note that using Cauchy-Schwarz and Jensen's inequality,
\begin{equation} \label{decomp2nd}
  |\|M_{1}(\boldsymbol{a}_{0})\|^2 - \IE[\|M_{1}(\boldsymbol{a}_{0})\|^2]| \leq \|\Tilde{B}_1(\boldsymbol{a}_0) \|^2 \text{ and } \|\Tilde{S}_m\|^2 = O(m),
\end{equation}
and 
\begin{equation} \label{decomp3rd}
  \IE\bigg|\IE[\Tilde{B}_1 C_2(\boldsymbol{\eta}_1) | {a}_0, \ldots, {a}_{1-m}] \bigg|^{\gamma/2} = O(m^{\gamma/2}) \| \IE[\Tilde{B}_1(\boldsymbol{a}_0)\|^2.
\end{equation} 
For $ |\tilde{V}_{2}({\boldsymbol{a}_{3}}) - \IE(\tilde{V}_{2}({\boldsymbol{a}_{3}}))|$, a treatment similar to \eqref{decomp2nd} and \eqref{decomp3rd} reveals that we only need to bound $\IP\bigg( | \|\Tilde{B}_3(\boldsymbol{a}_3)\|^2 - \IE[\|\Tilde{B}_3(\boldsymbol{a}_3)\|^2] | > C Jm\bigg)$. Proceeding as in \eqref{finalde}
\begin{align}
  \|\Tilde{S}_m^2- \IE[\Tilde{S}_m^2| a_1, \ldots, a_m] \|_{\gamma/2}^{\gamma/2}= o((Jm)^{\gamma/2}).
\end{align}
Finally, using Nagaev Inequality and in view of the fact $\IE[\Tilde{S}_m^2]=O(m)$, we have
\begin{equation}\label{nagaev}
  \IP\left(|\Tilde{S}_m^2 - \IE[\Tilde{S}_m^2]| > C Jm \right) \leq \IP\left(|\Tilde{S}_m| > C \sqrt{Jm} \right) \leq C\frac{m}{(Jm)^{\gamma/2}} \Theta_{0,p}^p+ \exp(-C_1J). 
\end{equation}
For the second term in \eqref{nagaev}, clearly $K\exp(-C_1J) \to 0$. For the first term, $K\frac{m}{(Jm)^{\gamma/2}} = \frac{\log^{\gamma} K }{m^{\gamma/2- 1}}\to 0$. This completes the proof.
\end{proof}
\begin{proposition}\label{prop4}
Suppose $L_{\gamma}^{\boldsymbol{a}}=\sum_{l=1}^K\IE[|Y_l^{\boldsymbol{a}}|^{\gamma}]$ with $Y_l^{\boldsymbol{a}}$ defined in \eqref{Yj}. Then for some constants $c$ and $C$ it holds that
\begin{equation}
  \IP( cKm^{\gamma/2} \leq L_{\gamma}^{\boldsymbol{a}} \leq CKm^{\gamma/2} ) \to 1.
\end{equation}  
\end{proposition}
\begin{proof}
  Observe that $\IE[L_{\gamma}^{\boldsymbol{a}}]=O(Km^{\gamma/2})$, and $\IE[|\Tilde{S}_m- M_1^{}\boldsymbol{a}|^{\gamma}| \boldsymbol{a}]\leq \IE[|\Tilde{S}_m|^{\gamma} | \boldsymbol{a}]$. Therefore, it is enough to show that 
  \begin{equation}
    \IP\left(|\sum_{l=1}^K Q_l- \IE[Q_l]|> CK \right) \to 0,
  \end{equation}
  where $Q_l=m^{-\gamma/2} \IE[|\Tilde{S}_{3ml}- \Tilde{S}_{3m(l-1)}|^{\gamma} | \boldsymbol{a}_{3(l-1)}, \boldsymbol{a}_{3l}]$. 
   Without loss of generality we can assume $Q_l$ to be independent, otherwise we can consider even and odd sums separately. A treatment similar to \eqref{decomp2nd} and \eqref{nagaev} yields,
  \begin{equation}\label{beftrunc}
    \IP\left(Q_j > J^{\gamma/2} \right)= o(K^{-1}).
  \end{equation}
  In light of \eqref{beftrunc}, we will be done if we show
  \begin{equation}\label{finaltrunc}
    \IP\left(\sum_{l=1}^K [T_{J^{\gamma/2}}(Q_l)- \IE[T_{J^{\gamma/2}}(Q_l)] \right) \to 0.
  \end{equation}
  \eqref{finaltrunc} is immediate using Markov's inequality upon realizing that $\IE[Q_l]=O(1)$, which implies that $\IE[T_{J^{\gamma/2}}(Q_l)^2]= O(J^{\gamma/2})$. 
\end{proof}
\begin{proposition}\label{prop5}
  Choose $\nu_l=lJ$ and $s\asymp K/J$ with $J=\lfloor K^{2/\gamma} / \log^2 K \rfloor$. Further let $\Gamma_k^{\boldsymbol{a}}=\text{Var}(\sum_{l=\nu_{k-1}+1}^{\nu_k} Y_l^{\boldsymbol{a}})$. Then for some constants $c_1, c_2$ we have, with probability going to 1,
  \begin{equation}\label{gotzezaitsevcondn1}
    c_1 w^2 \leq \rho_{\star}(\Gamma_k^{\boldsymbol{a}}) \leq \rho^{\star}(\Gamma_k^{\boldsymbol{a}}) \leq c_2 w^2,
  \end{equation}
  where $w=(L_{\gamma}^{\boldsymbol{a}})^{1/\gamma}/ \log s$. Further, if $\zeta_{k,\gamma}^{\boldsymbol{a}}=\sum_{l=\nu_{k-1}+1}^{\nu_k} \IE[|Y_l^{\boldsymbol{a}}|^{\gamma}]$, then for some $0<\varepsilon<1$ and constant $c_3$, it holds with probability going to 1,
  \begin{equation}\label{gotzezaitsevcondn2}
    c_3 d^{\gamma/2} s^{\varepsilon} (\log s)^{\gamma+3} \max_{1 \leq k \leq s}\zeta_{k,\gamma}^{\boldsymbol{a}} \leq L_{\gamma}^{\boldsymbol{a}}.
  \end{equation}
\end{proposition}
\begin{proof}
  \eqref{gotzezaitsevcondn1} is immediate from Propositions \ref{prop2}, \ref{prop3} and \ref{prop4} along with our choice $J$. Moreover, exactly as in Proposition \ref{prop4}, we can show $\zeta_k^{\boldsymbol{a}} \asymp Jm^{\gamma/2}$ with probability going to 1. Then \eqref{gotzezaitsevcondn2} follows as $n \to \infty$. 
\end{proof}

}
\ignore{
\section{Appendix C: Additional lemma for Theorem \ref{thm:scb}}
The following result is a technical lemma required to bound the total variation of the weights for the local linear estimate. 
\begin{lemma}\label{lemma3zhibiao}
  Let $K$ be a smooth symmetric kernel on $[-\omega, \omega]$, satisfying the bounded variation condition \eqref{condn:regularity}. Let \[\mathcal{S}_n(t)=\begin{pmatrix} S_0(t) & S_1(t)\\ S_1(t) & S_2(t) 
  \end{pmatrix} \] with $S_{j}(t)$ be defined as in \eqref{eq:Sj}. Then with $h_n \to 0$ and $nh_n \to \infty$, it holds that 
  \begin{equation}
    \sup_{t \in [\omega h_n, 1 - \omega h_n]} \mathcal{S}_n(t) = nh \begin{pmatrix}1 & 0 \\ 0 & 2\beta h_n^2 \end{pmatrix} (1+ o(1)),
  \end{equation}
  where $\beta=\frac{1}{2} \int u^2 K(u) \text{d} u$.
\end{lemma}
\vspace{-0.25 in}
\begin{proof}
Observe that \[S_j(t)= \int_0^n \left(\frac{\lfloor 1+u \rfloor}{n}-t\right)^j K\left(\frac{\lfloor 1+u \rfloor - nt}{nh_n}\right)\ \text{d}u.\]. Let $m_j =\int v^j K(v) \text{d}v$. Note that $m_0=1, m_1=0$ and $m_2=2\beta$. Thus, using \eqref{condn:regularity} and $t \in [\omega h_n, 1- \omega h_n]$, 
  \begin{align*}
    S_{j}(t)&= \int_{0}^n (\frac{u}{n}-t)^j K(\frac{u - nt}{nh_n}) \ \text{d} u + O(h_n^j) \\
    &= m_j n h_n^{j+1} + O(h_n^j),
  \end{align*}
  which completes the proof in view of $(nh_n)^{-1}=o(1)$. 
\end{proof}
}
\ignore{
\section{Bootstrap}\label{proofsec4}

Note that, our main theorem \ref{thm:GA} uses $T_k$ from (\ref{eq:T_k}) which arises from the expected block variance. However, from a practical perspective, these values are seldom known and thus given a series in hand, the best one can hope for, is an empirical version of $T_k$ as a plugged-in estimator. First we state the following algorithm that does this implementation and put forth some remarks.

\begin{algorithm}[H]
	\renewcommand{\algorithmicrequire}{\textbf{Input:}}
	\renewcommand{\algorithmicensure}{\textbf{Output:}}
	\caption{Bootstrap for non-stationary process}
	\label{alg:Boot_ns}
	\begin{algorithmic}[2]
		\REQUIRE Observed data $\{X_i\}_{i =1}^n$; $b$, the size of Bootstrap; $m$, the block size. 
		\FORALL{$j = 1,2,\ldots,\lfloor n/m \rfloor$} 
		\STATE $B_j \gets \sum_{t = (j-1) m +1}^{jm} X_t$
		\ENDFOR
		\FORALL{$i = 1,2,\ldots,b$} 
		\STATE Update $\boldsymbol{\eta} \gets \lfloor i/m \rfloor$
		\STATE Update $R \gets \sum_{t = \boldsymbol{\eta} m +1}^{i} X_t$; $R \gets 0$ if $i = \boldsymbol{\eta} m$ 
		\STATE Update $(C_1, \ldots, C_{\boldsymbol{\eta}}, C_{\boldsymbol{\eta} + 1}) \gets (B_1, \ldots, B_{\boldsymbol{\eta}}, R)$
		\STATE Store $\hat T_i \gets \sum_{t = 1}^{\boldsymbol{\eta} + 1} C_{t}^2 + 2 \sum_{t =1}^{\boldsymbol{\eta}} C_t C_{t+1}$
		\ENDFOR
		\STATE Store $(\hat H_1, \ldots, \hat H_n) \gets \operatorname{sort} \left (\hat T_1, \ldots, \hat T_n \right)$
		\STATE Store $(r_1, \ldots, r_n) \gets \operatorname{rank} \left (\hat T_1, \ldots, \hat T_n \right)$
		
		\IF{$H_1 < 0$}
		\STATE Store $w \gets \max\{ t: \hat H_t < 0 \}$.
		\FORALL{$s = 1,2,\ldots,b$} 
		\STATE Generate $Z_1, Z_2, \ldots, Z_n$ i.i.d. standard normal. 
		\STATE Update $S_{s,w} \gets \sqrt{-\hat H_w} Z_w$, $S_{s,t} \gets S_{s,t + 1} + \sqrt{\hat H_{t+1} - \hat H_{t}} Z_t$, $t = 1,\ldots,w-1$.
		\STATE Update $S_{s,w+1} \gets \sqrt{\hat H_{w+1}} Z_{w+1}$, $S_{s,t} \gets S_{s,t - 1} + \sqrt{\hat H_t - \hat H_{t-1}} Z_t$, $t = w+2,\ldots,n$.
		\STATE Update $Y_{s,1} \gets S_{s,r_1}$, $Y_{s,t} \gets S_{s,r_t} - S_{s,r_{t-1}}$, $t = 2, \ldots,n$.
		\STATE Store $\mathbf{Y}_s \gets (Y_{s,1}, Y_{s,2}, \ldots, Y_{s,n})$
		\ENDFOR
	  \ELSE
		\FORALL{$s = 1,2,\ldots,b$} 
		\STATE Generate $Z_1, Z_2, \ldots, Z_n$ i.i.d. standard normal. 
		\STATE Update $S_{s,1} \gets \sqrt{\hat H_1} Z_1$, $S_{s,t} \gets S_{s,t - 1} + \sqrt{\hat H_t - \hat H_{t-1}} Z_t$, $t = 2,\ldots,n$.
		\STATE Update $Y_{s,1} \gets S_{s,r_1}$, $Y_{s,t} \gets S_{s,r_t} - S_{s,r_{t-1}}$, $t = 2, \ldots,n$.
		\STATE Store $\mathbf{Y}_s \gets (Y_{s,1}, Y_{s,2}, \ldots, Y_{s,n})$
		\ENDFOR
		\ENDIF

		\ENSURE Output the data frame $(\mathbf{Y}_1, \mathbf{Y}_2, \ldots, \mathbf{Y}_b)$ .
	\end{algorithmic} 
\end{algorithm}
}
\section{Appendix D: Additional lemma for Theorem \ref{thm:scb}}\label{appendix:scb}
Here we will prove a technical lemma required to bound the total variation of the weights for the local linear estimate. This lemma also helps control the bias of the estimate $\hat{\mu}_{h_n}(t)$. 
\begin{lemma}[Consistency]\label{lemma3zhibiao}
  Let $S_{j}(t)$ be defined as in \eqref{eq:Sj}. Then with $h_n \to 0$ and $nh_n \to \infty$, under the assumptions of Theorem \ref{thm:scb}, it holds that 
  \begin{equation}\label{11.1}
    \sup_{t \in [\omega h_n, 1 - \omega h_n]} \bigg|\frac{{S}_j(t)}{nh_n^{j+1}f(t)}\bigg|=m_j + o(1), \text{ for $j=0,1,2$,}
  \end{equation}
with $m_0=1$, $m_1=0$ and $m_2=2\beta=\int u^2 K(u) \text{d} u$. Moreover, for $\Omega_n$ in \eqref{eq:omega} with $w_{h_n}(t,i)$ as in \eqref{eq:hat_mu}, it holds that $\Omega_n=O(n^{-1}h_n^{-1})$. 
\end{lemma}
\begin{proof}
Observe that \[S_j(t)= \int_0^n \left(F^{-1}(\frac{\lfloor 1+u \rfloor}{n})-t\right)^j K\left(\frac{F^{-1}(\lfloor 1+u \rfloor /n) - t}{h_n}\right)\ \text{d}u.\] Let $m_j =\int v^j K(v) \text{d}v$. Note that $m_0=1, m_1=0$ and $m_2=2\beta$. 
Consider the corresponding smoothed version 
\[\Tilde{S}_j(t)= \int_0^n \left(F^{-1}(\frac{u}{n})-t\right)^j K\left(\frac{F^{-1}(u/n) - t}{h_n}\right)\ \text{d}u.\]
Let $v=(F^{-1}(u/n) -t)/{h_n}$. Also let $g$ be such that $g(v)=(F^{-1}(\lfloor 1+u \rfloor /n) - t)/{h_n}$. Clearly, since $C_1 \leq f(t) \leq C_2$ for all $t$, therefore, $|g(v)-v|=O(n^{-1}h_n^{-1})$ for all $t$. Now, by change of variables techniques and noting that $t \in [\omega h_n, 1-\omega h_n]$, it holds that
\begin{align*}
  S_j(t) - \Tilde{S}_j(t) = nh_n^{j+1} \int_{-\omega}^{\omega} \left[ (g(v))^j K(g(v)) - v^j K(v) \right] f(vh_n+t) \text{d}v.
\end{align*}
For $j=0$, $S_j(t) - \Tilde{S}_j(t)=O(h_n^j)$ follows from \eqref{condn:regularity} directly. For $j=1,2$, note that $(g(v))^j - v^j=O(n^{-1}h_n^{-1})$ since $v \in [-\omega, \omega]$. Therefore, again invoking \eqref{condn:regularity} yields that 
\begin{align}\label{diffrate}
  S_j(t) - \Tilde{S}_j(t)=O(h_n^j). 
\end{align}
Finally, observing $f(vh_n+t)=f(t)+O(vh_n)$,
\begin{align}
  \Tilde{S}_j(t) &= nh_n^{j+1} \int_{-\omega}^{\omega} v^j K(v) f(vh_n+t) \nonumber\text{d}v = nh_n^{j+1}m_jf(t)+ O(nh_n^{j+2}) \label{imparg},
\end{align}
 since $c\leq f(t)\leq C$. Thus, using $h_n\to 0$ and $nh_n\to \infty$, it holds that 
  \begin{align*}
    S_{j}(t)&= \int_{0}^n (F^{-1}(\frac{u}{n})-t)^j K(\frac{F^{-1}(u/n) - t}{h_n}) \ \text{d} u + O(h_n^j) = n f(t) h_n^{j+1} (m_j + o(1)).
  \end{align*}
  which completes the proof of \eqref{11.1}. To observe $\Omega_n=O(n^{-1}h_n^{-1})$, note that for fixed $t$, $w_n(t,i)=0$ unless $t_i \in [t- \omega h_n, t+\omega h_n]$, and therefore in \eqref{eq:hat_mu}, $|t-t_i|=O(h_n)$. Putting the approximations of $S_j(t)$ in \eqref{eq:hat_mu}, and noting that $K$ is bounded, one obtains $|w_n(t,1)|=O(n^{-1}h_n^{-1})$. On the other hand, $ \sum_{i=2}^n |w_i(t)-w_{i-1}(t)|$ can be bounded by $O(n^{-1}h_n^{-1})+O(1/(nh_n)^2)$ by noting that $|\sum_{i=1}^n ((t-t_i) K((t-t_i)/h_n)-(t-t_{i-1}) K((t-t_{i-1})/h_n )|=O(h_n) $. This completes the proof.
\end{proof}
\color{black}
\section{Appendix E: Additional simulations for Section \ref{sec:simu}}\label{appendix:simu}
\subsection{Empirical accuracy of the Gaussian approximation with estimated variance}\label{subsec:simu1} For the strongly dependent settings with $\theta=-0.8$ and $0.8$, we will explore the finite-sample accuracy of our Gaussian approximation when the variance of the Brownian motion is estimated using bootstrap. For a particular model, we take $n=600$ and $m=\lfloor n^{1/3} \rfloor$, and simulate $B=1000$ many samples, each of size $n$ to estimate $U_1:= \max_{1 \leq i \leq n} S_i$. Next we randomly generate a data of size $n$ from that model, and simulate $B$ many bootstrap samples of \[\hat{U}_2:=\max_{1 \leq i \leq n}\mathbb{W}(\mathcal{T}_i), \, \hat{U}_3:= \max_{1 \leq i \leq n}\mathbb{W}(\mathcal{T}_{i}^{-}), \, \hat{U}_4 := \max_{1 \leq i \leq n} \mathbb{W}(\mathcal{T}_{i}^{\diamond}),\]
 where `` $\hat{}$ '' in $\hat{U}_i$, $i=2,3,4$ emphasize their dependence on the randomly generated data based on which bootstrap is performed. 
 Figures \ref{Fig:ar_nonsim_bootstrap} and \ref{Fig:ar_sq_nonsim_bootstrap_t} depict typical QQ-plots of $\hat{U}_2$, $\hat{U}_3$ and $\hat{U}_4$ against $U_X$ for $\varepsilon_t \sim N(0,1)$, where the ``typical'' is used to emphasize that $\hat{U}_i$'s are generated via bootstrap based on one typical draw of $(X_i)_{i=1}^n$ from the corresponding models. We note that as expected from our theoretical discussion, $\mathcal{T}_i$ yields much better approximation to the quantiles of $\max_{1 \leq i \leq n} S_i$ compared to $\mathcal{T}_i^{-}$ and $\mathcal{T}_i^{\diamond}$.
\begin{figure}[!htbp]
\centering
\includegraphics[height=7cm, width=13cm]{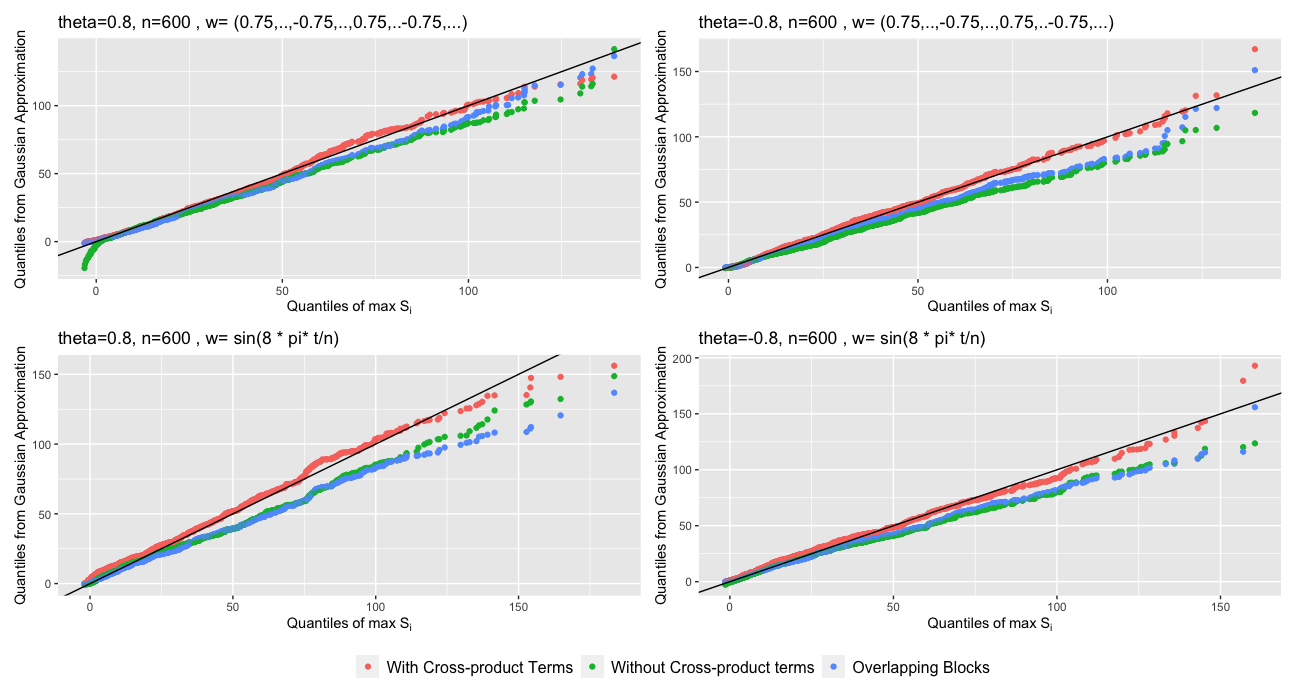} 
\caption{Comparison of theoretical quantiles with the bootstrap Gaussian approximation quantiles based on $X_1, \ldots, X_n \sim$ Model \ref{ARnonsim} with $N(0,1)$ innovations, with and without cross-product terms.}
\label{Fig:ar_nonsim_bootstrap}  
\end{figure}

 \begin{figure}[!htbp]
\centering
\includegraphics[height=7cm, width=13cm]{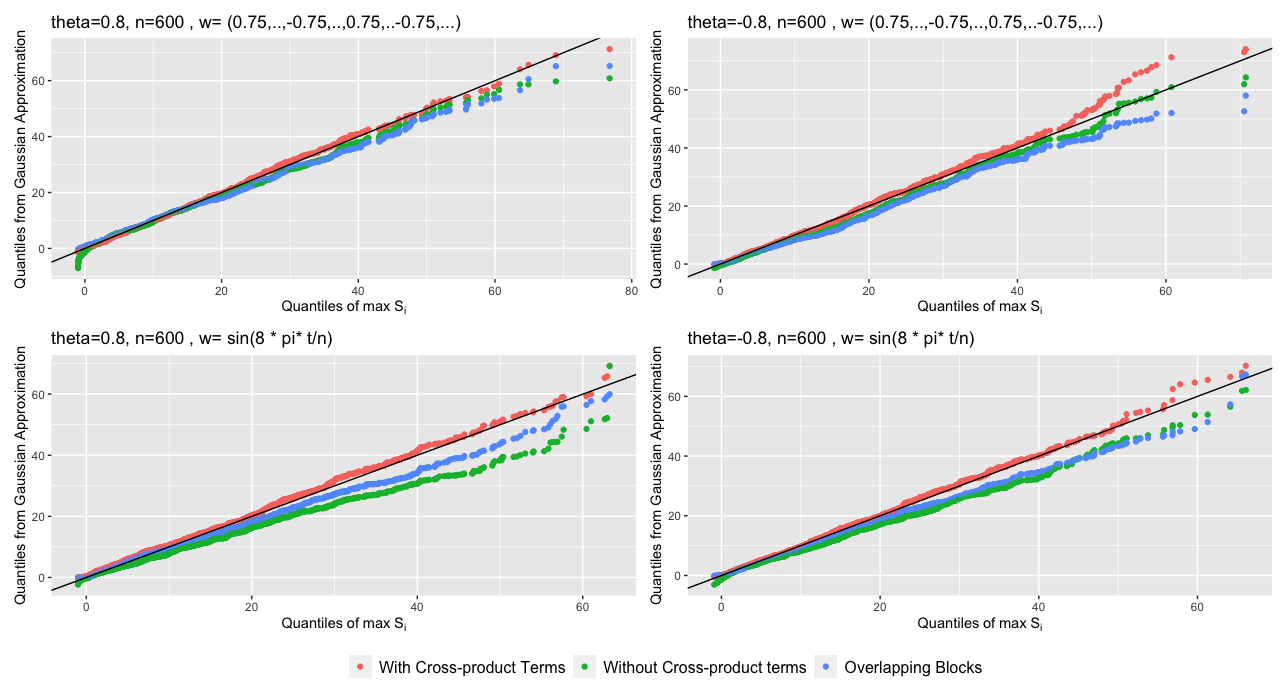} 
\caption{Comparison of theoretical quantiles with the bootstrap Gaussian approximation quantiles based on $X_1, \ldots, X_n \sim$ Model \ref{ARnonsimsq} with $N(0,1)$ innovations, with and without cross-product terms.}
\label{Fig:ar_sq_nonsim_bootstrap_t}  
\end{figure}

\subsection{Empirical accuracy for Gaussian approximation}\label{extrasimu5.1}
In this section we further explore the performance of the Models \ref{ARnonsim} and \ref{ARnonsimsq}. In addition to $N(0,1)$ innovations, we will also consider suitably normalized $t_6$ innovations (subsequently we will omit "normalized" while describing the errors). 
Figures \ref{Fig:ar_sim_bootstrap_t}-\ref{Fig:ar_nonsim_sine_bootstrap_t} depict the ``typical'' Q-Q plots of $1000$ data-based bootstrap samples of $\IB(\mathcal{T}_i)$, $\mathcal{W}(\mathcal{T}_i^{-})$ and $\IB(\mathcal{T}_i^{\diamond})$ against the theoretical quantiles of $\max_{1 \leq i \leq n} S_i$ (based on 1000 monte carlo samples) for different settings. The conclusions reflect those of Figures \ref{Fig:ar_nonsim_bootstrap} and \ref{Fig:ar_sq_nonsim_bootstrap_t}, and justify our use of $\mathcal{T}_i$ as a plug-in estimate for $\IE(S_i^2)$.
\begin{figure}[!htbp]
\centering
\includegraphics[height=7cm, width=13cm]{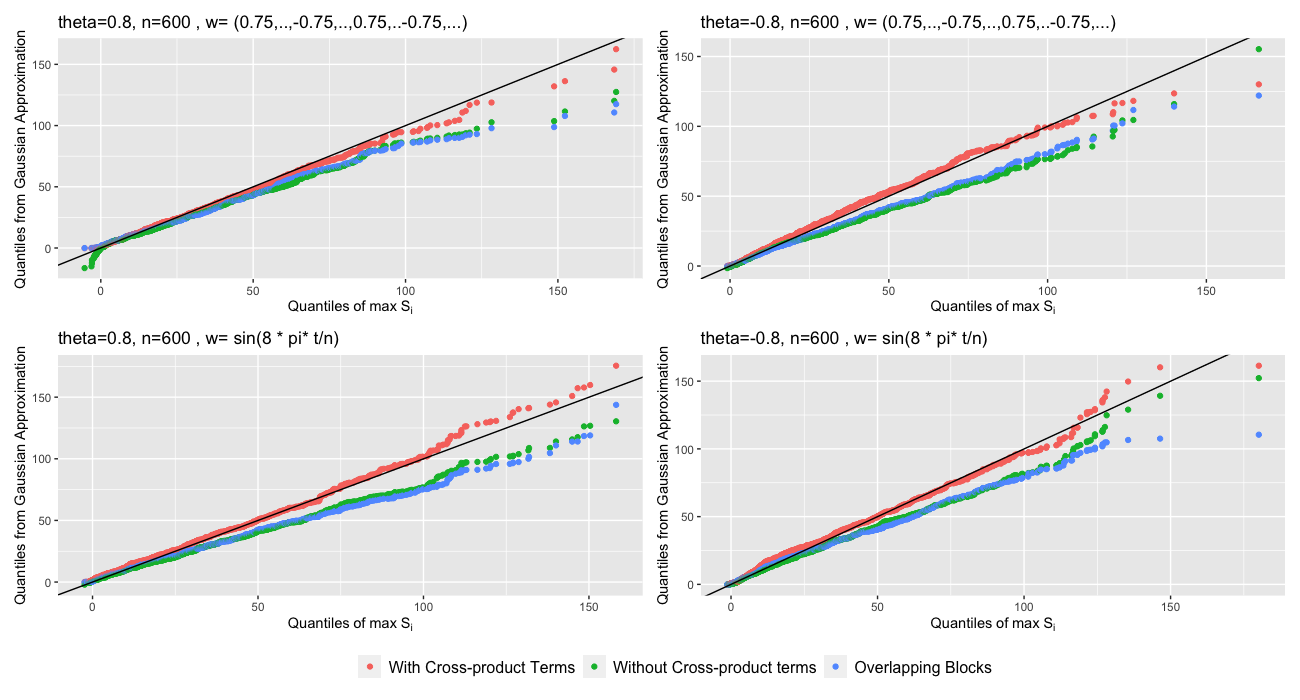} 
\caption{Comparison of theoretical quantiles with the bootstrap Gaussian approximation quantiles based on $X_1, \ldots, X_n \sim$ Model \eqref{ARnonsim} with $t_6$ innovations, with and without cross-product terms.}
\label{Fig:ar_sim_bootstrap_t}  
\end{figure}
\begin{figure}[!htbp]
\centering
\includegraphics[height=5.5cm, width=13cm]{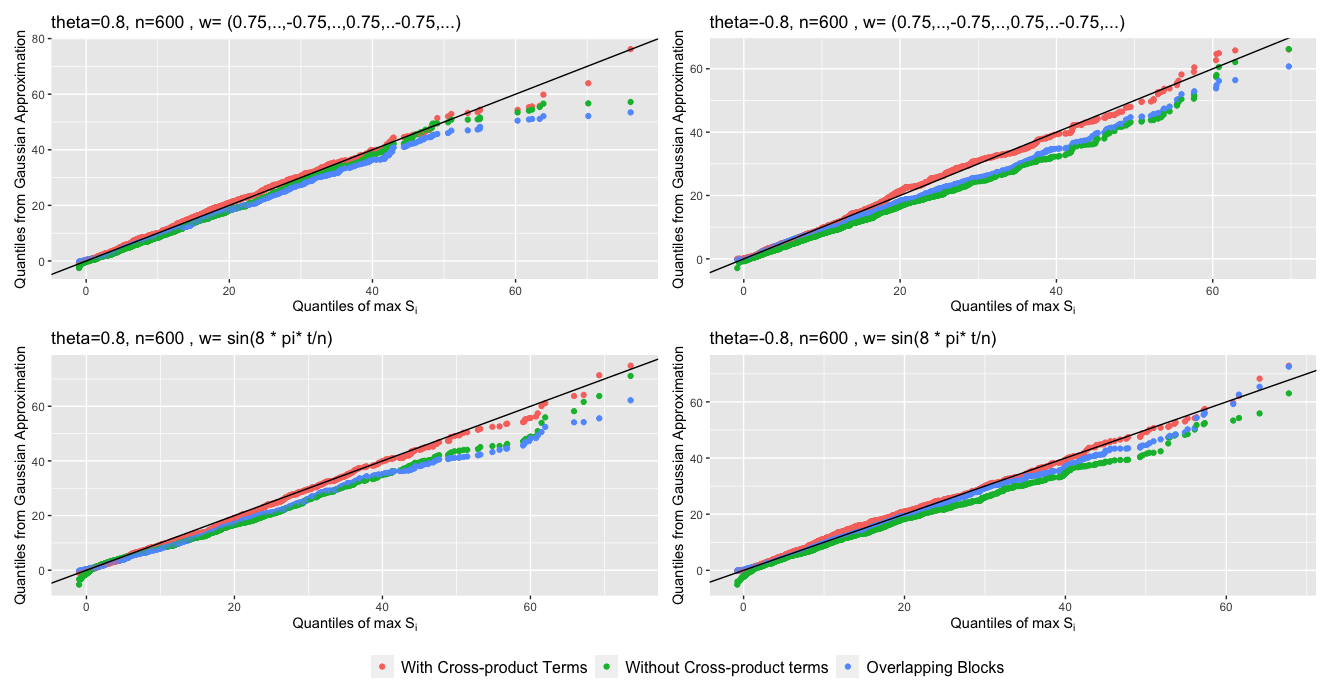} 
\caption{Comparison of theoretical quantiles with the bootstrap Gaussian approximation quantiles based on $X_1, \ldots, X_n \sim$ Model \eqref{ARnonsimsq} with $t_6$ innovations, with and without cross-product terms.}
\label{Fig:ar_nonsim_sine_bootstrap_t}  
\end{figure}

\subsection{Change-point detection}
\label{subsec:simu2}
 In the simulation below, we consider the following model:
\saveenum
\begin{enumerate}\resetenum
\item \label{mod2} $ X_t=\delta_t+\varepsilon_t \ , \varepsilon_t \sim \text{Model \ref{ARnonsim}} \ , \ \delta_t=\delta \mathbb{I}\{t > n/2 \}.$
  \ignore{
  \item \textbf{Model 5.4}:
  \begin{equation}\label{mod4}
    X_t=\delta_t+\varepsilon_t^2 \ , \varepsilon_t \sim \text{Model \ref{ARnonsim}} \ , \ \delta_t=\delta \mathbb{I}\{t > n/2 \}.
  \end{equation}
  }
\end{enumerate}
 For power calculation, we vary $\delta \in \{0.1, 0.2,\ldots, 1\}$. Note that $\delta=0$ leads to the type-1-error. 
\vspace{-0.1in} 
\subsubsection{Simulation based on theoretical cut-offs}\label{subsec:theocp}
The discussion in Section \ref{subsec:cpt} enables us to define an approximately valid level-$\alpha$ test $\psi_{n}$, which we identify as an oracle test, as follows:
    \begin{equation*} \label{test_cp}
    \psi_{n}:= \mathbb{I}\{U_n > c_{\alpha}\}, \text{where }c_{\alpha}:= \inf_{r}\{\IP(V >r )\leq \alpha \},\text{and } V = \max_{1\leq i\leq n} { \big| \mathbb{B}(\IE(S_i^2)) - \frac{i}{n} \mathbb{B}(\IE(S_n^2))\big| \over \sqrt n}.
  \end{equation*}
In our simulation we theoretically calculate $(\IE(S_i^2))_{i=1}^n$ for i.i.d. $\varepsilon_t \sim N(0,1)$, and estimate $c_{\alpha}$ based on $1000$ Monte-Carlo simulations of $V$. For each $\delta$, power is estimated based on $1000$ many samples of each size $n$, where we vary $n \in \{300,600\}$. The type-I error and estimated power are shown in the Table \ref{tab:pow_theo1}. As expected, the type-I error of the test is at the nominal level, and even though $\psi_n$ seems slightly conservative, power grows quickly as $n$ and $\delta$ increases. Note that for $\theta=-0.8$ and $0.8$, the dependence is strong, and this results in slightly lesser power compared to other cases. \ignore{Note that the comparatively worse Gaussian approximation for Model \eqref{ARnonsimsq} materializes with lesser power for Model \ref{mod4} as in Table \ref{tab:pow_theo2}.} 
\begin{table}[H]
  \centering
  \resizebox{\columnwidth}{ 0.09\textheight}{%
\begin{tabular}{|l|c|c|c|c||c|c|c|c||c|c|c|c||c|c|c|c||}
 \hline& \multicolumn{8}{c||}{Weights : $w= 0.75, \ldots, -0.75, \ldots, 0.75, \ldots, -0.75,\ldots$} & \multicolumn{8}{c||}{Weights: $w=\sin(8\pi t/n)$}\\
\hline & \multicolumn{4}{c|}{$n=300$} & \multicolumn{4}{c||}{$n=600$} & \multicolumn{4}{c|}{$n=300$} & \multicolumn{4}{c||}{$n=600$}\\
\hline & $\theta=-0.8$ & $\theta=-0.4$ & $\theta=0.4$ & $\theta=0.8$ & $\theta=-0.8$ & $\theta=-0.4$ & $\theta=0.4$ & $\theta=0.8$ & $\theta=-0.8$ & $\theta=-0.4$ & $\theta=0.4$ & $\theta=0.8$ & $\theta=-0.8$ & $\theta=-0.4$ & $\theta=0.4$ & $\theta=0.8$\\ 
\hline \hline Cutoff & $2.409$ & $1.46$ & $1.614$& $2.414$ & $2.574$ & $1.551$ & $1.588$ & $2.482$ &$2.717$ & $1.508$ & $1.441$& $2.833$ & $2.86$ & $1.505$ & $1.518$ & $2.84$ \\
\hline \hline Type-1 error & $0.055$ & $0.066$ & $0.027$ & $0.047$ & $0.035$ & $0.051$ & $0.046$ & $0.042$ & $0.043$ & $0.044$ & $0.059$ & $0.024$ & $0.048$ & $0.036$ & $0.047$ & $0.048$ \\
\hline \hline Power: $\delta=0.1$ & $0.048$ & $0.120$ & $0.059$ & $0.057$& $0.065$ & $0.137$ & $0.117$ & $0.065$ & $0.046$ & $0.121$ & $0.117$ & $0.028$ & $0.056$ & $0.177$ & $0.150$ & $0.059$\\ 
\hline Power: $\delta=0.2$ & $0.110$ & $0.283$ & $0.200$ & $0.100$& $0.129$ & $0.435$ & $0.431$ & $0.195$ & $0.083$ & $0.263$ & $0.305$ & $0.073$ & $0.135$ & $0.486$ & $0.497$ & $0.149$ \\
\hline Power: $\delta=0.3$ & $0.244$ & $0.592$ & $0.419$ & $0.212$& $0.337$ & $0.838$ & $0.812$ & $0.418$ & $0.145$ & $0.515$ & $0.587$ & $0.128$ & $0.287$ & $0.851$ & $0.863$ & $0.287$\\
\hline Power: $\delta=0.4$  & $0.378$ & $0.810$ & $0.730$ & $0.358$& $0.601$ & $0.972$ & $0.979$ & $0.668$ & $0.280$ & $0.813$ & $0.819$ & $0.227$ & $0.491$ & $0.976$ & $0.984$ & $0.495$ \\
\hline Power: $\delta=0.5$  & $0.527$ & $0.944$ & $0.915$ & $0.571$& $0.852$ & $0.999$ & $0.999$ & $0.844$ & $0.425$ & $0.952$ & $0.953$ & $0.391$ & $0.705$ & $0.998$ & $0.999$ & $0.725$ \\
\hline Power: $\delta=0.6$  & $0.704$ & $0.994$ & $0.982$ & $0.735$& $0.946$ & $1$ & $1$ & $0.964$ & $0.617$ & $0.991$ & $0.994$ & $0.549$ & $0.875$ & $1$ & $1$ & $0.863$\\
\hline Power: $\delta=0.7$  & $0.847$ & $0.999$ & $0.999$ & $0.866$& $0.994$ & $1$ & $1$ & $0.994$ & $0.722$ & $0.998$ & $0.999$ & $0.693$ & $0.968$ & $1$ & $1$ & $0.965$\\
\hline Power: $\delta=0.8$  & $0.929$ & $1$ & $1$ & $0.954$& $0.998$ & $1$ & $1$ & $0.999$ & $0.838$ & $0.999$ & $1$ & $0.831$ & $0.991$ & $1$ & $1$ & $0.990$\\
\hline Power: $\delta=0.9$  & $0.977$ & $1$ & $1$ & $0.986$& $1$ & $1$ & $1$ & $1$ & $0.922$ & $1$ & $1$ & $0.908$ & $0.999$ & $1$ & $1$ & $0.998$\\
\hline Power: $\delta=1$  & $0.996$ & $1$ & $1$ & $0.996$& $1$ & $1$ & $1$ & $1$ & $0.967$ & $1$ & $1$ & $0.951$ & $0.999$ & $1$ & $1$ & $0.999$\\ \hline
\end{tabular}
}
\caption{Type-I error and power of test $\psi_n$ for $X_1, \ldots, X_n\sim$ Model \ref{mod2}.}
  \label{tab:pow_theo1}
\end{table}
\ignore{
\begin{table}[H]
  \centering
  \resizebox{\columnwidth}{!}{%
\begin{tabular}{|l|c|c|c|c|c|c||c|c|c|c|c|c||}
 \hline& \multicolumn{6}{c||}{Weights : $w_{t}=0.75 \mathbb{I}\{t \leq n/2\} -0.75 \mathbb{I}\{t>n/2\}$} & \multicolumn{6}{c||}{Weights: $w_t=\sin(2\pi t/n)$}\\
\hline & \multicolumn{3}{c|}{\rho=0.8} & \multicolumn{3}{c||}{\rho=-0.8} & \multicolumn{3}{c|}{\rho=0.8} & \multicolumn{3}{c||}{\rho=-0.8}\\
\hline & $n=100$ & $n=200$ & $n=300$ & $n=100$ & $n=200$ & $n=300$ & $n=100$ & $n=200$ & $n=300$ & $n=100$ & $n=200$ & $n=300$\\ 
\hline \hline Cutoff & 4.155 & 4.192& 4.208& 4.148 & 4.172&4.187 & 5.818 &5.93 &6.099 & 5.850 &5.867 &5.842 \\
\hline \hline Type-1 error & 0.062 &0.062 &0.055 & 0.057 &0.063 & 0.065& 0.060 &0.065 & 0.056& 0.053 & 0.060& 0.059 \\
\hline \hline $\delta=0.2$ & 0.071 & 0.084 &0.071 & 0.062 &0.07 &0.074 & 0.066 & 0.062 & 0.053 & 0.068 & 0.061 & 0.064 \\
\hline  $\delta=0.4$ & 0.073 &0.123 &0152 & 0.075 &0.128 & 0.149& 0.067 &0.072 &0.086 & 0.064 & 0.066& 0.086 \\
\hline $\delta=0.6$ & 0.111 & 0.189& 0.286& 0.115 &0.193 &0.284 & 0.082 & 0.091&0.106 & 0.062 &0.081 &0.112 \\
\hline $\delta=0.8$ & 0.142 &0.330 &0.461 & 0.163 & 0.313&0.458 & 0.074 &0.125 & 0.157& 0.073 & 0.127& 0.182\\
\hline  $\delta=1$ & 0.240 &0.498 &0.670 & 0.209 & 0.488& 0.695& 0.091 &0.166 &0.252 & 0.097 &0.176 &0.296 \\
\hline $\delta=1.2$ & 0.339 &0.599 &0.807 & 0.320 &0.640 &0.844 & 0.131 &0.220 & 0.376& 0.101 &0.272 &0.438 \\
\hline $\delta=1.4$ & 0.449 & 0.782&0.927 & 0.478 & 0.788& 0.931& 0.165 &0.386 & 0.531& 0.152 &0.338 & 0.598\\
\hline $\delta=1.6$ & 0.588 & 0.889&0.983 & 0.592 & 0.885&0.982 & 0.216 &0.490 & 0.708& 0.215 &0.499 &0.766 \\
\hline $\delta=1.8$ & 0.675 &0.953 &0.993 & 0.708 &0.963 &0.991 & 0.269 &0.640 & 0.829& 0.275 &0.645 & 0.845\\
\hline $\delta=2$ & 0.830 &0.981 &0.998 & 0.812 &0.976 & 1& 0.401 & 0.758&0.906 & 0.334 &0.753 &0.916\\ \hline
\end{tabular}
}
\caption{Type-1 error and power of test $\psi_n$ for $X_1, \ldots, X_n\sim$ Model \ref{mod4}.}
  \label{tab:pow_theo2}
\end{table}
}
\subsubsection{Simulation based on Bootstrap} \label{sec:simu_boot_cpt}
 Following our discussion in Section \ref{ssc:estvar} as well as 
Section \ref{subsec:simu1}, we can estimate $\IE(S_i^2)$ by $\mathcal{T}_i$ as in \eqref{eq:T_k}, $\mathcal{T}_i^{-}$ as in \eqref{T_noblock} and $\mathcal{T}_i^{\diamond}$ as in \eqref{Ti_Mies} respectively, to yield three bootstrap-based tests. 
We will numerically compare the efficacy of the these bootstrap procedures in approximating the CUSUM test statistics. In order to estimate the asymptotic distribution under $H_0$, we estimate $\mathcal{T}_i$, $\mathcal{T}_i^{-}$ and $\mathcal{T}_i^{\diamond}$ by plugging in $X_i - \hat{\mu}_i$ instead of $Z_i$, where $\hat{\mu}_i=\tau^{-1}\sum_{j=1}^{\tau} X_j \mathbb{I}\{i \leq \tau\} + (n-\tau)^{-1}\sum_{j=\tau+1}^{n} X_j \mathbb{I}\{i > \tau\}$, with $\tau = {\rm argmax}_t |\sum_{i = 1}^{t} (X_t - \bar{X})| / \sqrt{n} $. \ignore{A complete algorithm is provided below.

\begin{algorithm}[H]
	\renewcommand{\algorithmicrequire}{\textbf{Input:}}
	\renewcommand{\algorithmicensure}{\textbf{Output:}}
	\caption{Testing for the existence of change point based on CUSUM}
	\label{alg:change_point}
	\begin{algorithmic}[2]
		\REQUIRE Observed data $\{X_i\}_{i =1}^n$; $b$, the size of Bootstrap; $m$, the block size. 
	  \FORALL{$i = 1,2,\ldots,n$} 
		\STATE $\varepsilon_i \gets X_i -\Bar{X}$
		\ENDFOR 
		\STATE With $\{\varepsilon_i\}$, $b$, and $m$ as the input of Algorithm~\ref{alg:Boot_ns}, obtain Bootstrap sample $\{\mathbf{Y}_i \}_{i = 1}^b$.
		\FORALL{$j = 1,2,\ldots,b$} 
		\STATE $U_j \gets \max_{i\leq n } |\sum_{t = 1}^{i} (Y_{j,t} - \bar{\mathbf{Y}}_j)| / \sqrt{n}$
		\ENDFOR 
		\STATE $U_0 \gets \max_{i\leq n } |\sum_{t = 1}^{i} (X_t - \bar{X})| / \sqrt{n}$.
		\STATE $\hat{p} \gets \frac{1}{b+1} (1 + \sum_{h=1}^b \mathbb{I}\{ U_h > U_0 \})$
		\ENSURE Estimated p value $\hat{p}$.
	\end{algorithmic} 
\end{algorithm}}
\ignore{
We use 500 bootstrap samples to estimate the data-based cut-offs $c_{\alpha}^T$, $c_{\alpha}^{T-}$ and $c_{\alpha}^{T^{\diamond}}$. To calculate the type-1 error and power, we repeat the experiment 500 times. For the purpose of this simulation, we take $\varepsilon_t \sim t_6\sqrt{2/3}$ for $X_t \sim$ Model \ref{mod2} and \ref{mod4}. 
}
Similar to Figures \ref{Fig:ar_nonsim_bootstrap} and \ref{Fig:ar_sq_nonsim_bootstrap_t}, Figures \ref{Fig:cp_nonsim_bootstrap_norm} and \ref{Fig:cpsq_nonsim_bootstrap_norm} depict "typical" QQ-plots of the CUSUM test statistic calculated from bootstrap samples based on $\mathcal{T}_i$, $\mathcal{T}_i^{-}$ and $\mathcal{T}_i^{\diamond}$ respectively, against the CUSUM statistic calculated from original random sample $\{X_1, \ldots, X_n \}$, with $\mathcal{T}_i$ generally providing the best approximation in line with our arguments in Section \ref{sec:nocp}. 
\begin{figure}[!htbp]
\centering
\includegraphics[height=6cm, width=13cm]{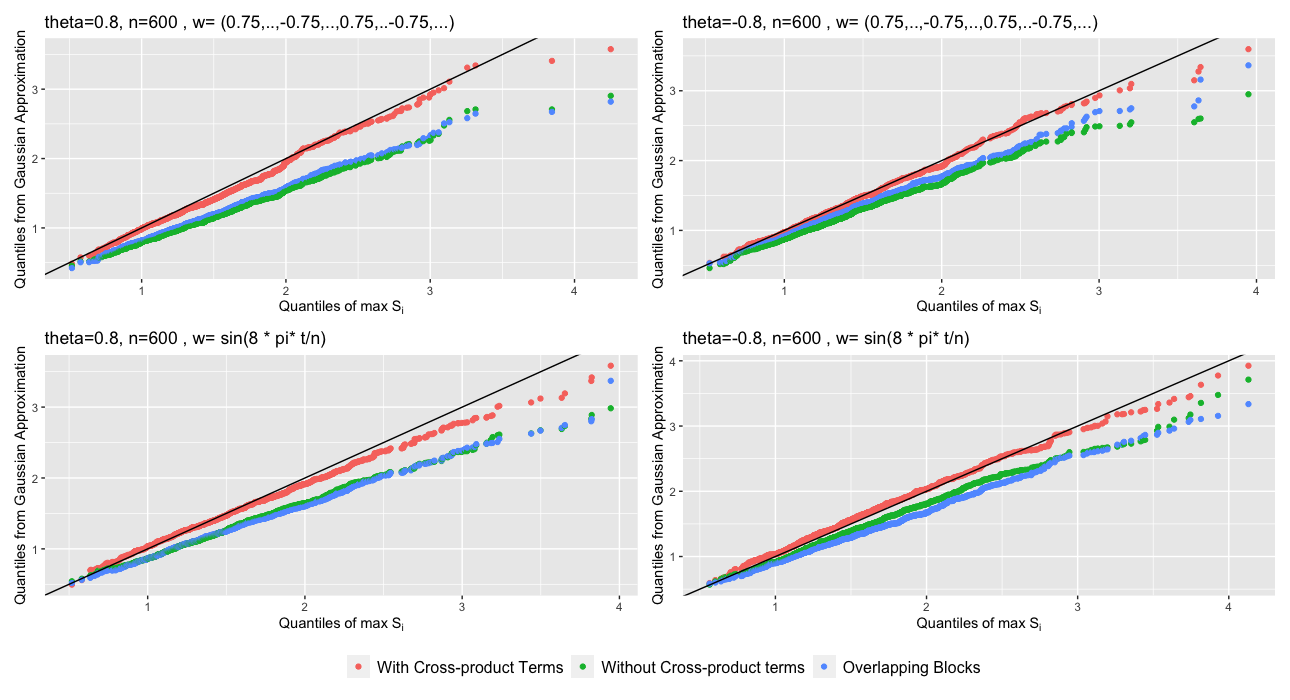} 
\caption{Comparison of theoretical quantiles of CUSUM statistic $U_n$ with quantiles of bootstrap Gaussian approximation of CUSUM quantiles based on $X_1, \ldots, X_n \sim$ Model \ref{ARnonsim} with $N(0,1)$ innovations, with and without cross-product terms.}
\label{Fig:cp_nonsim_bootstrap_norm}  
\end{figure}
\begin{figure}[!htbp]
\centering
\includegraphics[height=6cm, width=13cm]{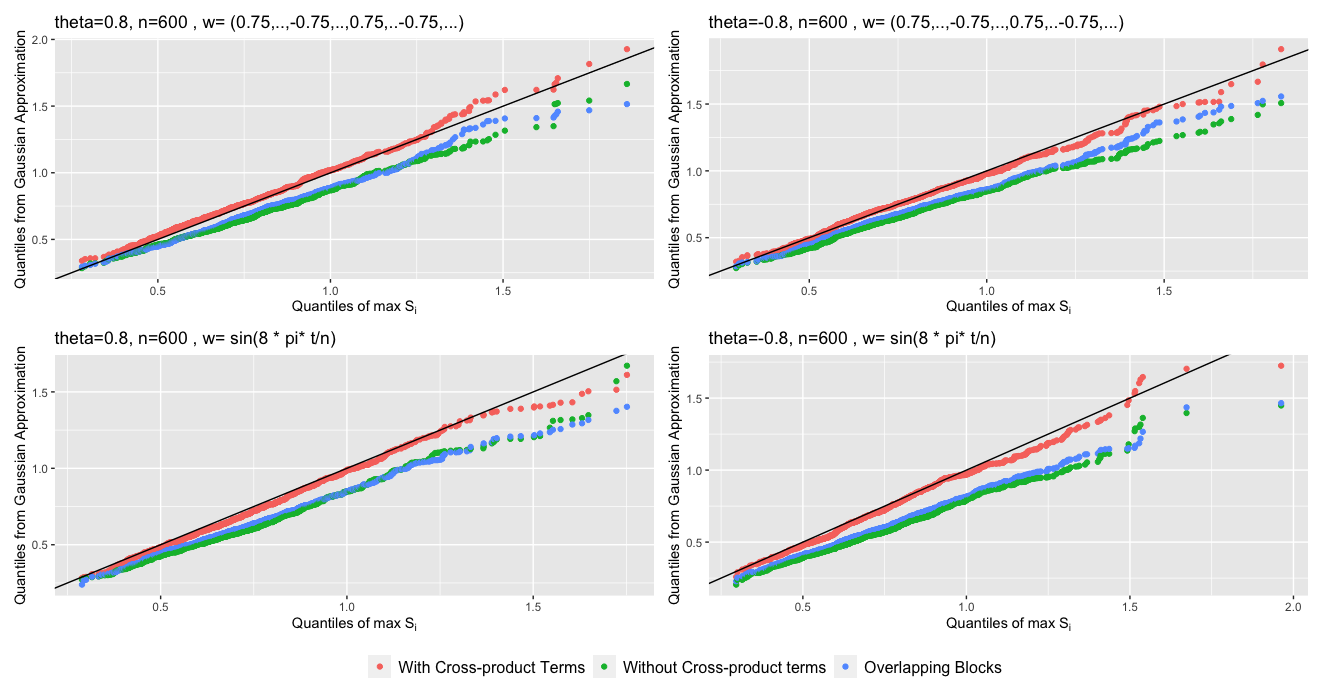} 
 \caption{Comparison of theoretical quantiles of CUSUM statistic $U_n$ with quantiles of bootstrap Gaussian approximation of CUSUM quantiles based on $X_1, \ldots, X_n \sim$ Model \ref{ARnonsimsq} with $N(0,1)$ innovations, with and without cross-product terms.}
\label{Fig:cpsq_nonsim_bootstrap_norm}  
\end{figure}
\vspace{-0.1 in}

\subsection{Change-point detection: simulation based on Bootstrap}\label{extracpsimu5.1}
In this section we further explore the performance of the bootstrap-based test as described in Section \ref{sec:simu_boot_cpt} for $t_6$ innovations. We consider $X_t \sim$ Model \ref{ARnonsim} and \ref{ARnonsimsq} respectively. Figures \ref{Fig:cp_nonsim_bootstrap_t}-\ref{Fig:cp_nonsim_bootstrap_sin_t} show the plots corresponding to Figures \ref{Fig:cp_nonsim_bootstrap_norm} and \ref{Fig:cpsq_nonsim_bootstrap_norm} for each of the models \ref{ARnonsim} and \ref{ARnonsimsq}. 

\begin{figure}[!htbp]
\centering
\includegraphics[height=6cm, width=13cm]{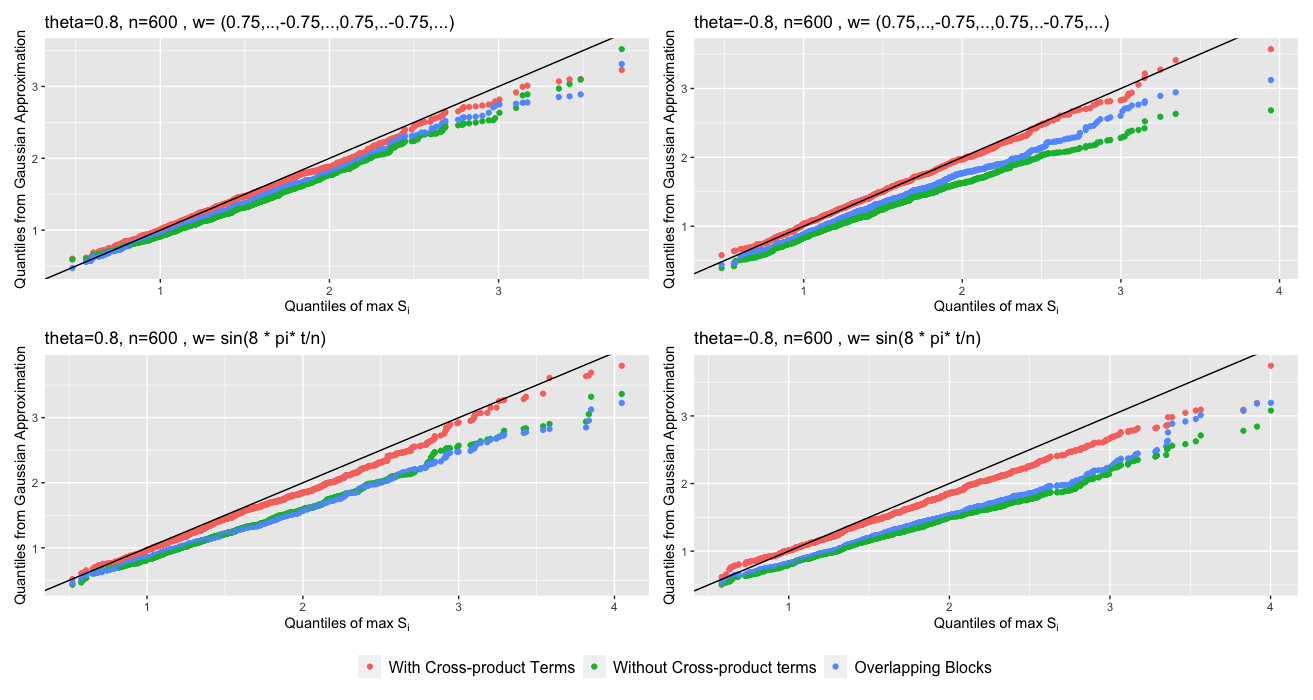} 
\caption{Comparison of theoretical quantiles of CUSUM statistic $U_n$ with quantiles of bootstrap Gaussian approximation of CUSUM quantiles based on $X_1, \ldots, X_n \sim$ Model \eqref{ARnonsim} with $t_6$ innovations, with and without cross-product terms.}
\label{Fig:cp_nonsim_bootstrap_t}  
\end{figure}
\begin{figure}[!htbp]
\centering
\includegraphics[height=6cm, width=13cm]{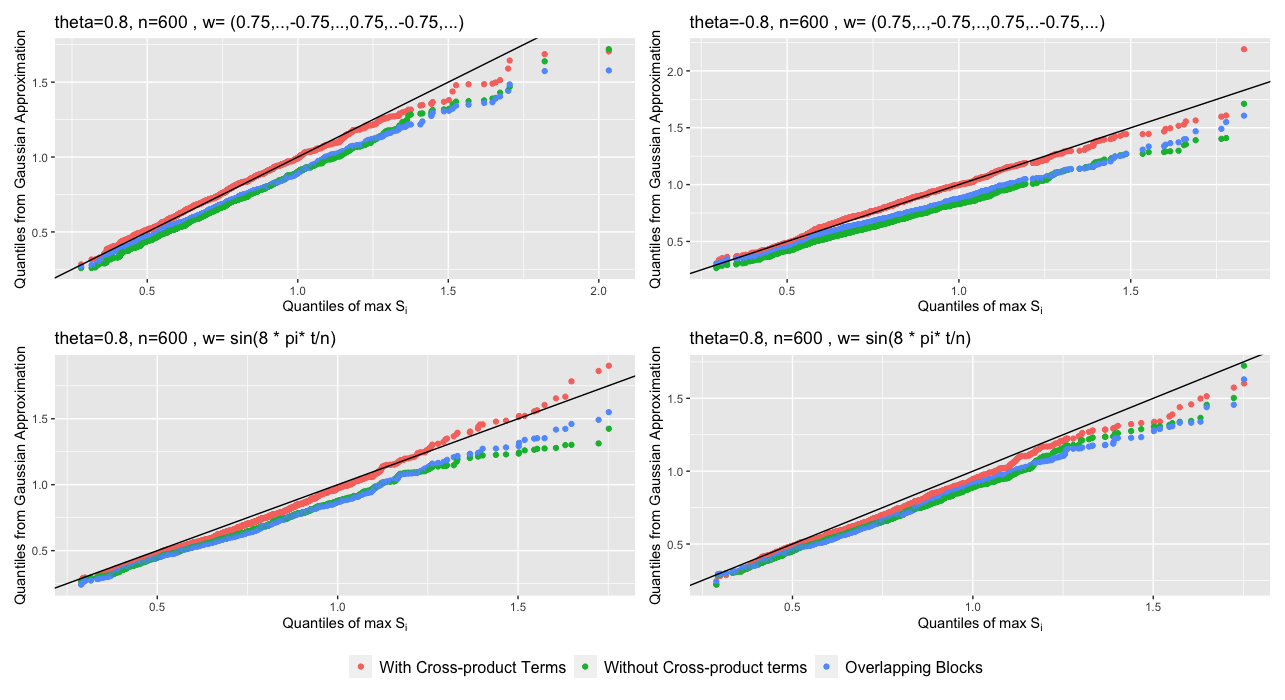} 
\caption{Comparison of theoretical quantiles of CUSUM statistic $U_n$ with quantiles of bootstrap Gaussian approximation of CUSUM quantiles based on $X_1, \ldots, X_n \sim$ Model \eqref{ARnonsimsq} with $t_6$ innovations, with and without cross-product terms.}
\label{Fig:cp_nonsim_bootstrap_sin_t}  
\end{figure}

\end{document}